\documentclass{amsart}

\usepackage{amsmath,amssymb,amsthm,enumerate} 
\usepackage{amsfonts} 
\usepackage[top=30mm,bottom=30mm,left=25mm,right=25mm]{geometry}
\usepackage[dvipdfmx]{graphicx}
\usepackage{color}

\usepackage[T1,T2A]{fontenc}
\usepackage[utf8]{inputenc}

\theoremstyle{plain}
\newtheorem{pretheo}{Theorem}[section]
\newtheorem{preassu}[pretheo]{Assumption}
\newtheorem{precoro}[pretheo]{Corollary}
\newtheorem{predefi}[pretheo]{Definition}
\newtheorem{preexam}[pretheo]{Example}
\newtheorem{prelemm}[pretheo]{Lemma}
\newtheorem{preprop}[pretheo]{Proposition}
\newtheorem{prerema}[pretheo]{Remark}
\newtheorem{preexer}{Exercise}[section]

\newenvironment{theo}{\begin{pretheo}}{\end{pretheo}}

\newenvironment{coro}{\begin{precoro}}{\end{precoro}}
\newenvironment{defi}{\begin{predefi}}{\end{predefi}}

\newenvironment{lemm}{\begin{prelemm}}{\end{prelemm}}
\newenvironment{prop}{\begin{preprop}}{\end{preprop}}
\newenvironment{rema}{\begin{prerema}\rm}{\end{prerema}}

\DeclareMathOperator{\di}{div}

\DeclareMathOperator{\tr}{tr}
\DeclareMathOperator{\Di}{Div}
\DeclareMathOperator{\sign}{sign}

\newcommand{\supp}{{\rm supp}\,}

\newcommand{\pa}{\partial}

\newcommand{\wh}[1]{\widehat{#1}}
\newcommand{\wt}[1]{\widetilde{#1}}

\newcommand{\Hol}{{\rm Hol}\,}

\newcommand{\BBA}{\mathbb{A}}
\newcommand{\BBB}{\mathbb{B}}
\newcommand{\BBC}{\mathbb{C}}
\newcommand{\BBD}{\mathbb{D}}

\newcommand{\BBF}{\mathbb{F}}
\newcommand{\BBG}{\mathbb{G}}

\newcommand{\BBI}{\mathbb{I}}

\newcommand{\BBL}{\mathbb{L}}

\newcommand{\BBN}{\mathbb{N}}

\newcommand{\BBR}{\mathbb{R}}
\newcommand{\BBS}{\mathbb{S}}

\newcommand{\Bf}{\boldsymbol{f}}
\newcommand{\Bg}{\boldsymbol{g}}

\newcommand{\Bn}{\boldsymbol{n}}

\newcommand{\Bu}{\boldsymbol{u}}
\newcommand{\Bv}{\boldsymbol{v}}
\newcommand{\Bw}{\boldsymbol{w}}

\newcommand{\BF}{\boldsymbol{F}}
\newcommand{\BG}{\boldsymbol{G}}

\newcommand{\BM}{\boldsymbol{M}}

\newcommand{\BU}{\boldsymbol{U}}

\newcommand{\BZ}{\boldsymbol{Z}}

\newcommand{\CA}{\mathcal{A}}

\newcommand{\CD}{\mathcal{D}}
\newcommand{\CE}{\mathcal{E}}
\newcommand{\CF}{\mathcal{F}}
\newcommand{\CG}{\mathcal{G}}
\newcommand{\CH}{\mathcal{H}}

\newcommand{\CL}{\mathcal{L}}
\newcommand{\CM}{\mathcal{M}}

\newcommand{\CP}{\mathcal{P}}

\newcommand{\CR}{\mathcal{R}}
\newcommand{\CS}{\mathcal{S}}
\newcommand{\CT}{\mathcal{T}}

\newcommand{\CV}{\mathcal{V}}
\newcommand{\CW}{\mathcal{W}}
\newcommand{\CX}{\mathcal{X}}
\newcommand{\CY}{\mathcal{Y}}

\newcommand{\Fa}{\mathfrak{a}}

\newcommand{\Fg}{\mathfrak{g}}
\newcommand{\Fh}{\mathfrak{h}}

\newcommand{\Fp}{\mathfrak{p}}

\newcommand{\FX}{\mathfrak{X}}

\newcommand{\ep}{\varepsilon}

\newcommand{\0}{\boldsymbol{0}}
\newcommand{\Btau}{\boldsymbol{\tau}}

\newcommand{\Bphi}{\boldsymbol{\phi}}

\newcommand{\BPhi}{\boldsymbol{\Phi}}
\newcommand{\BPsi}{\boldsymbol{\Psi}}

\usepackage{mathrsfs}
\newcommand{\sA}{\mathscr{A}}

\newcommand{\sF}{\mathscr{F}}

\def\sAP{\mathscr{A}_{\Phi}}

\def\zij{\zeta^i_j}
\def\tzij{\widetilde{\zeta}^i_j}

\numberwithin{equation}{section}

\usepackage{tikz}
\usetikzlibrary{calc}
\usepackage{subcaption}

\newcommand{\ccomma}{\mathbin{\raisebox{0.5ex}{,}}}

\usepackage{scalerel,stackengine}
\stackMath
\newcommand\reallywidehat[1]{%
\savestack{\tmpbox}{\stretchto{%
  \scaleto{%
    \scalerel*[\widthof{\ensuremath{#1}}]{\kern-.6pt\bigwedge\kern-.6pt}%
    {\rule[-\textheight/2]{1ex}{\textheight}}
  }{\textheight}%
}{0.5ex}}%
\stackon[1pt]{#1}{\tmpbox}%
}
\newcommand{\rwh}[1]{\reallywidehat{#1}}

\begin{document}
\title[Barotropic compressible flow with free surface]{The $\CR-$bounded operator families arising from the study of the barotropic compressible flows with free surface}

\author[Xin Zhang]{Xin Zhang}
\address{Research Institute of Science and Engineering, Waseda University, 3-4-1 Ookubo, Shinjuku-ku, Tokyo, 169-8555, Japan}
\email{xinzhang@aoni.waseda.jp}




\date{\today}


\begin{abstract}
In this paper, we study some model problem associated to
the free boundary value problem of the barotropic compressible Navier-Stokes equations in general smooth domain with taking surface tension into account.
To obtain the maximal $L_p-L_q$ regularity property of the model problem,  
we prove the existence of $\CR-$bounded operator families of the resolvent problem via Weis' theory on operator valued Fourier multipliers.
\end{abstract}
\maketitle

%
%

\section{Introduction}\label{sec:intro}
\subsection{Model}
In this paper, we study the following boundary value problem in some general (bounded or unbounded) domain $\Omega\subset \BBR^N$ ($N\geq 2$) surrounded by two disjoint sharp surfaces $\Gamma_0$ and $\Gamma_1,$
\begin{equation}\label{eq:L_CNS_L}
	\left\{\begin{aligned}
	&\pa_t \eta + \gamma_1 \di \Bu  = d
	    &&\quad\hbox{in}\quad  \Omega \times \BBR_+, \\
&\gamma_1 \pa_t \Bu -\Di\big(\BBS(\Bu)-\gamma_2 \eta \BBI \big)= \Bf
		&&\quad\hbox{in}\quad  \Omega \times \BBR_+,\\
		&\big( \BBS(\Bu) - \gamma_2 \eta \BBI\big) \Bn_{\Gamma_0} 
+\sigma (m-\Delta_{\Gamma_0})h \,\Bn_{\Gamma_0} = \Bg
	    &&\quad\hbox{on}\quad  \Gamma_0 \times \BBR_+, \\
	    &\pa_t  h - \Bu \cdot \Bn_{\Gamma_0}  = k
	    &&\quad\hbox{on}\quad  \Gamma_0 \times \BBR_+,\\
&\Bu  = \0 &&\quad\hbox{on}\quad  \Gamma_1 \times \BBR_+, \\
		&(\eta, \Bu,h)|_{t=0} = (\eta_0, \Bu_0, h_0) &&\quad\hbox{in}\quad  \Omega.
	\end{aligned}\right.
\end{equation}
Above $\gamma_1=\gamma_1(x),$ $\gamma_2=\gamma_2(x)>0$ are smooth, 
the matrix $\BBS(\Bu)$ is defined by
\begin{equation*}
\BBS(\Bu):= \mu \BBD(\Bu)+ (\nu-\mu)\di \Bu\, \BBI
\end{equation*}
for the viscosity constants $\mu, \nu>0,$ 
$\BBD(\Bu):= \begin{bmatrix}\pa_k u_j + \pa_j u_k\end{bmatrix}_{N\times N}$ 
is called the (double) deformation tensor of $\Bu$ and $\BBI:= \begin{bmatrix}
\delta_{jk}
\end{bmatrix}_{N\times N}.$
In addition, for any vector $\Bu$ and any matrix $\BBA=\begin{bmatrix} A_{jk}\end{bmatrix}_{N\times N},$ we write $\di \Bu := \sum_{j=1}^N \pa_j u_j$
and $\Di \BBA := \sum_{k=1}^N \pa_k A_{jk}.$ 
In $\eqref{eq:L_CNS_L}_3,$ the constants $\sigma,m>0,$ $\Bn_{\Gamma_0}$ stands for the unit normal vector along $\Gamma_0$ and $\Delta_{\Gamma_0}$ for the Laplace-Beltrami operator of $\Gamma_0.$
Given the initial states $(\eta_0,\Bu_0, h_0)$ and source terms $d,\Bf,\Bg,k,$
the aim is to predict the variation of unknowns $(\eta, \Bu, h).$
\smallbreak

In fact, the model problem \eqref{eq:L_CNS_L} arises from the study the motion of the viscous gases governed by the barotropic compressible Navier-Stokes  equations (CNS) in some bounded or unbounded domain $\Omega_t \subset \BBR^N$ ($N\geq 2$) with taking the \emph{surface tension} into account.
For the free boundary value problem of (CNS), we need to determine not only the amplitude of the density, the velocity field of the fluid particles, but also the shape of moving domain $\Omega_t.$
In fact, the solvability of (CNS) can be reduced to the linearization form \eqref{eq:L_CNS_L}, which will be our forthcoming work. But let us emphasize here that the role of $h$ in \eqref{eq:L_CNS_L} is to handle the variation of the pattern $\Omega_t.$
\medskip

The study of (CNS) attracts the attention of many mathematicians for a long time.
One may refer to the following works and the references therein for a more complete list of previous works.   
The study of (CNS) is challenging even for the initial value problem because of the hyperbolicity from the conservation law of the mass. 
For instance, the long time issue for the initial value problem of (CNS) in the whole space was investigated by 
Matsumura \& Nishida \cite{MN1979,MN1980}, Hoff \cite{Hoff1995}, Hoff \& Zumbrun \cite{HZ1995} and Danchin \cite{Dan2000}.
On the other hand, for the non-slip (Dirichlet) boundary condition, we refer to the works \cite{MN1983} by Matsumura \& Nishida and \cite{KS1999} by Kobayashi \& Shibata  in the exterior domain, and \cite{KK2002,KK2005} by Kagei \& Kobayashi in the half space $\BBR^N_+$ ($N\geq 2$), and \cite{Kagei2008} by Kagei in the layer.
 
Next, concerning the free boundary value problem for (CNS) in some smooth bounded domain, Tani in \cite{Ta1981} and 
Secchi \& Valli in \cite{SV1983} established the short time solutions, and Zajaczkowski in \cite{Zaja1993} found some long time solutions by ignoring the role of surface tension  (i.e. $\sigma=0$).
The extension to the surface tension case was studied in \cite{SolTa1991,SolTa1992,Zaja1994}.
In particular, the authors of \cite{SolTa1991,SolTa1992,Zaja1994} proved the long time stability with respect to some trivial equilibrium states within the anisotropic Sobolev framework.

However, the aforementioned works on the free boundary value problem are in $L_2$ or H\"older regularity framework. 
To obtain the solutions with $L_p$ in time and $L_q$ in space ($L_p-L_q$ for short) regularity
\footnote{The maximal $L_p-L_q$ regularity of (CNS) is verified by Kakizawa in \cite{Kaki2011} with the Navier boundary condition in the bounded domain as well.}, 
we refer to the recent works \cite{EvBS2014,GS2014} by Shibata and his group for the case $\sigma=0.$ 
Moreover, for the study of the motion of the two-phase compressible flows, one may refer to \cite{JTW2016,KSS2016} and the references therein.
\medskip

Our purpose here is to tackle \eqref{eq:L_CNS_L} with the surface tension (i.e. $\sigma>0$) in maximal $L_p$-$L_q$ regularity framework. 
Here let us emphasize that $\Omega$ is not necessary bounded, as long as the boundaries of $\Omega$ are uniformly smooth.
Of course, $\Gamma_1=\emptyset$ in \eqref{eq:L_CNS_L} is allowed by refining our later proof. 
That is, we may consider the motion of some bounded isolated mass or the gases in some exterior region. 
More precisely, we prove that \eqref{eq:L_CNS_L} has a semigroup structure by imposing $(d,\Bf,\Bg,k)=\0$ (see Theorem \ref{thm:semigroup}). In addition, if the initial data vanish, then the solutions of \eqref{eq:L_CNS_L} admits the maximal $L_p$-$L_q$ regularity (see Theorem \ref{thm:maximal}). 

The idea to prove Theorem \ref{thm:semigroup} and Theorem \ref{thm:maximal}
are based on the analysis of the resolvent problem of \eqref{eq:L_CNS_L} (i.e. \eqref{eq:GR_CNS} below).
Most importantly, we show that the solution operator families of \eqref{eq:GR_CNS} are $\CR$-bounded, which allows us to apply the Weis' theory on operator valued Fourier multipliers in \cite{Weis2001}.
Furthermore, to overcome the main difficulty of the free boundary condition in \eqref{eq:L_CNS_L}, our study is reduced to some model problem in $\BBR^N_+$ associated to the generalized Lam\'e operator. In order to use the Weis's theory for the model in $\BBR^N_+,$ we have to treat the explicit solution formula in terms of Fourier transformation. 
Especially, tackling the surface equation in $\eqref{eq:L_CNS_L}_4$ is crucial in our work.
\smallbreak 

This paper is folded as follows. In the next section, we will state the main theorem 
(i.e.Theorem \ref{thm:GR_CNS}) concerning the generalized resolvent problem \eqref{eq:GR_CNS} and then the proofs of Theorem \ref{thm:semigroup} and Theorem \ref{thm:maximal} in view of Theorem \ref{thm:GR_CNS}.
Afterwards we will study resolvent problem over the half space $\BBR^N_+$ in Section \ref{sec:halfspace} and the bent half space in Section \ref{sec:bh} respectively.
Finally, we combine the estimates to obtain the results in general domain $\Omega$ in the last part of the paper.

\subsection{Notations and functional spaces}
Let us fix the notations in this paper.
In what follows, we denote the Fourier transform in $\BBR^N$ and its inverse by 
\begin{align*}
\CF_{x}[f] (\xi):= \int_{\BBR^N} e^{-i x\cdot \xi} f(x) dx, \quad
\CF_{\xi}^{-1}[g] (x):= \frac{1}{(2\pi)^N}\int_{\BBR^N} e^{i x\cdot \xi} g(\xi) d\xi.
\end{align*}
For $x=(x',x_{_N}), \xi=(\xi',\xi_{_N}) \in \BBR^N,$ we sometimes use the partial Fourier (inverse) transform with respect to the horizontal variables,
\begin{align*}
\CF_{x'}[f] (\xi',x_{_N})&:= \int_{\BBR^{N-1}} e^{-i x'\cdot \xi'} f(x',x_{_N}) dx', \\
\CF_{\xi'}^{-1}[g] (x',\xi_{_N})&:= \frac{1}{(2\pi)^{N-1}}\int_{\BBR^{N-1}} e^{i x'\cdot \xi'} g(\xi',\xi_{_N}) d\xi'.
\end{align*}
Besides, the letter $C(a,b,c,\cdots)$ or $C_{a,b,c,\dots}$denotes that the constant $C$ depends on $a,b,c,\dots$
\smallbreak

Hereafter, $L_q(\Omega)$ is the standard Lebesgue space in domain $G \subset \BBR^N,$ 
and $H^k_p(G)$ with $k\in \BBN$ and $1<q<\infty$ stands for the Sobolev space.
In addition, the Besov space $B^{s}_{q,p}(G)$ 
for some $k-1<s\leq k$ and $(p,q) \in ]1,\infty[^2$
is defined by the real interpolation functor
\begin{equation*}
B_{q,p}^{s}(G):= \big(L_q(G),H^{k}_q(\Omega)\big)_{s\slash k,p}.
\end{equation*} 
In particular, we write $W^{s}_q(G)=B^s_{q,q}(G)$ for simplicity, 
and $W^{-\bar{s}}_q(G)$ is the dual space of $W^{\bar{s}}_{q'}(G)$ for $0<\bar{s}<1$ and the conjugate index $q':=q\slash (q-1).$ 

For any Banach spaces $X,Y,$ the total of the bounded linear transformations from $X$ to $Y$ is denoted by $\CL(X;Y).$ We also write $\CL(X)$ for short if $X=Y.$ 
In addition, $\Hol (\Lambda;X)$ denotes the set of $X$ valued mappings defined on some domain $\Lambda \subset \BBC.$
\smallbreak

Now we give the conditions of the (uniformly) smoothness of $\Omega$ in the sequel. 
\begin{defi}\label{def:domain}
We say that a connected open subset $\Omega$ in $\BBR^N$ ($N \geq 2$) is of class uniform $W^{m-1\slash r}_r$ for some integer $m \geq 2$ and $1<r<\infty.$ if and only if the boundary $\pa \Omega$ is uniformly characterized by local $W^{m-1\slash r}_r$ graph functions.
That is, for any point $x_0=(x_0',x_{0N}) \in \pa \Omega,$ one can choose a Cartesian coordinate system with origin $x_0$ and coordinates $y=(y',y_{_N}):=(y_1,...,y_{_{N-1}},y_{_N}),$ as well as positive constants $\alpha, \beta, K$ and some $W^{m-1\slash r}_r$ function $h$ with $\|h\|_{W^{m-1\slash r}_r(B_{\alpha}'(x_0'))} \leq K$ such that 
\begin{gather*}
 \{(y',y_{_N}): h(y')-\beta <y_{_N} < h(y'), |y'|<\alpha \}=\Omega \cap U_{\alpha,\beta, h}(x_0),\\
 \{(y',y_{_N}): y_{_N} = h(y'), |y'|<\alpha \}  = \pa\Omega \cap U_{\alpha,\beta, h}(x_0),
\end{gather*}
where $U_{\alpha,\beta, h}(x_0):= \{(y',y_{_N}): h(y')-\beta <y_{_N} < h(y')+\beta, |y'|<\alpha \}$ and $B_\alpha'(x_0'):=\{y'\in \BBR^{N-1}: |y'-x_0'| < \alpha\}.$
Moreover, the choices of $\alpha, \beta, K$ are independent of the location of $x_0,$

Assume that $\Omega$ is some domain in $\BBR^N$ with disjoint boundaries $\Gamma_0$ and $\Gamma_1,$ where the case $\Gamma_0 = \emptyset$ or $\Gamma_2=\emptyset$ is allowed.
We say $\Omega$ is of type $W^{3,2}_r$ for simplicity, if $\Gamma_k$ is uniformly $W^{3-1\slash r -k}_{r}$ for $k=0,1.$ 
\end{defi}

At last, we recall \cite[Theorem 2.1]{Shi2016} on the Laplace-Beltrami operator.
For any $\lambda_0>0$ and $0<\ep<\pi\slash 2,$ we introduce the sectorial regions
\footnote{One may also refer to the Figure \ref{fig:sect} below for $\Sigma_{\ep,\lambda_0}.$}
\begin{gather*}
\Sigma_{\ep}:=\big\{z \in \BBC \backslash \{0\} : |\arg{z}| \leq \pi -\ep \big\},
\quad 
\Sigma_{\ep,\lambda_0}:=\{z \in \Sigma_{\ep} : |z|\geq \lambda_0\}.
\end{gather*}
\begin{prop}\label{prop:resolvent_LB}
Let $0<\ep<\pi\slash 2,$ $1<q,q':=q\slash (q-1)<\infty,$ $N<r<\infty$ and 
$r \geq  \max\{q,q'\}.$ For any uniform $W^{2-1\slash r}_r$ boundary $\Gamma \subset \pa \Omega,$ there exists a constant $\lambda_1=\lambda_1(\ep,\Gamma)>0,$ such that $\Sigma_{\ep,\lambda_1}$ is contained in the resolvent set $\rho(\Delta_{\Gamma})$ of $\Delta_{\Gamma}.$ That is, for any $\lambda \in \Sigma_{\ep,\lambda_1}$ and $f \in W^{-1\slash q}_q(\Gamma),$ 
the resolvent problem 
\begin{equation*}
(\lambda -\Delta_{\Gamma}) u =f \,\,\, \text{on} \,\,\, \Gamma
\end{equation*}
admits a unique solution $u \in W^{2-1\slash q}_q(\Gamma)$ possessing the estimates
\begin{equation*}
\|u\|_{W^{2-1\slash q}_q(\Gamma)} \leq C_{\ep,q,r,\Gamma} \|f\|_{W^{-1\slash q}_q(\Gamma)}.
\end{equation*}
\end{prop}

\section{Main results}
In this section, we first state the results for the resolvent problem of \eqref{eq:L_CNS_L} in Subsection \ref{subsec:RP}. 
Then by applying the estimates of the resolvent problem, we can prove the existence of the semigroup of solution operators associated to \eqref{eq:L_CNS_L} in Subsection \ref{subsec:SG} and the property of maximal regularity of \eqref{eq:L_CNS_L} in the last part.

\subsection{Reduced resolvent problem}
\label{subsec:RP}
Now, let us begin with the following resolvent problem of \eqref{eq:L_CNS_L},
\begin{equation}\label{eq:GR_CNS}
	\left\{\begin{aligned}
	&\lambda \eta + \gamma_1 \di \Bu  = d
	    &&\quad\hbox{in}\quad  \Omega, \\
&\gamma_1 \lambda \Bu -\Di\big(\BBS(\Bu)-\gamma_2 \eta \BBI \big)= \BF
		&&\quad\hbox{in}\quad  \Omega,\\
		&\big( \BBS(\Bu) - \gamma_2 \eta \BBI\big) \Bn_{\Gamma_0} 
+\sigma (m-\Delta_{\Gamma_0})h \,\Bn_{\Gamma_0} = \BG
	    &&\quad\hbox{on}\quad  \Gamma_0, \\
	    &\lambda  h - \Bu \cdot \Bn_{\Gamma_0}  = K
	    &&\quad\hbox{on}\quad  \Gamma_0,\\
&\Bu  = \0 &&\quad\hbox{on}\quad  \Gamma_1, 
	\end{aligned}\right.
\end{equation}
where $\gamma_1 =\gamma_1(x)$ and $\gamma_2 =\gamma_2(x)$ are uniformly continuous functions defined on $\overline{\Omega}.$ Moreover, there exist constants $\rho_1,\rho_2,\rho_3$ such that
\begin{gather}\label{hyp:gamma_GR}
0<\rho_1 \leq \gamma_1(x) \leq \rho_2,\quad 
0 <\gamma_2(x) \leq \rho_2, \,\,\, \forall \,\, x \in \overline{\Omega},\\ \nonumber
 \|(\nabla \gamma_{1},\nabla \gamma_{2})\|_{L_r(\Omega)}  \leq \rho_2,  \quad 
  \rho_3:= \max\big\{ \rho_2, \|\gamma_1\gamma_2\|_{L_\infty(\Omega) \cap \wh H^1_r(\Omega)}\big\},
\end{gather}
with $N<r<\infty.$
For any $\nu, \lambda_0>0$ and $0<\ep<\pi\slash 2,$ we set that
\begin{gather}
\Lambda_{\ep, \lambda_0} := \Big\{z \in \Sigma_{\ep,\lambda_0} : 
\big(\Re z + \frac{\rho_3}{\nu} +\ep\big)^2 +(\Im z )^2 
\geq \big( \frac{\rho_3}{\nu} +\ep\big)^2  \Big\}.\label{def:Lambda}
\end{gather}
One may refer to the Figure \ref{fig:sect} for the graph of $\Lambda_{\ep,\lambda_0}.$
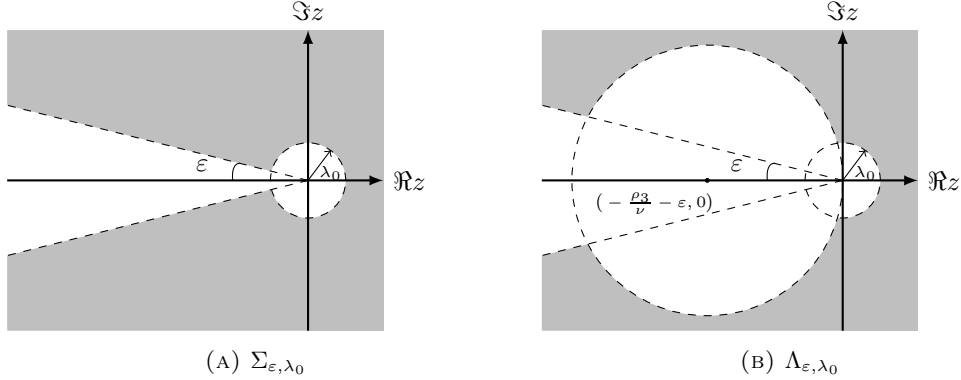
\begin{figure}[h]  
\centering 
  \begin{subfigure}[b]{0.4\linewidth}
    \begin{tikzpicture}   
     \path [fill=lightgray] (-4,2)--(-4,1)--(0,0)--(0,2) --(-4,2);  
     \path [fill=lightgray] (-4,-2)--(-4,-1)--(0,0)--(0,-2) --(-4,-2);  
     \path [fill=lightgray] (0,2) --(0,-2)--(1,-2)--(1,2)--(0,2);  
     \path [fill=white] (0,0) circle [radius=0.5];
     \coordinate (Origin)   at (0,0);
    \coordinate (XAxisMin) at (-4,0);
    \coordinate (XAxisMax) at (1,0);
    \coordinate (YAxisMin) at (0,-2);
    \coordinate (YAxisMax) at (0,2);
    \draw [thick,-latex] (XAxisMin) -- (XAxisMax) 
             node[right] {$\Re z$};
    \draw [thick,-latex] (YAxisMin) -- (YAxisMax)
              node[above] {$\Im z$};       
      \draw [dashed]  (0,0) -- (-4,1);  
      \draw [dashed]  (0,0) -- (-4,-1);   
      \draw (-1,0) to [out=100, in=170] (-0.9,0.225) ;
      \draw (-1.2,0.2) node[left]{$\varepsilon$};
      \draw [dashed] (0,0) circle [radius=0.5];
      \draw [->] (0,0) -- (0.3, 0.4);      
      \draw  (0.3, 0.1) node{\tiny $\lambda_0$};
    \end{tikzpicture}
    \caption{$\Sigma_{\varepsilon,\lambda_0}$}
  \end{subfigure}  
  \quad
\begin{subfigure}[b]{0.4\linewidth}
  \begin{tikzpicture}  
     \path [fill=lightgray] (-4,2)--(-4,1)--(0,0)--(0,2) --(-4,2);  
     \path [fill=lightgray] (-4,-2)--(-4,-1)--(0,0)--(0,-2) --(-4,-2);  
     \path[fill=lightgray] (0,2) --(0,-2)--(1,-2)--(1,2)--(0,2);  
     \path [fill=white] (0,0) circle [radius=0.5];
     \path [fill=white] (-1.8,0) circle [radius=1.8];
     \draw[fill] (-1.8,0) circle [radius=0.025];
     \draw (-2.5,-0.3) node{\tiny 
     $\big(-\frac{\rho_3}{\nu}
      -\varepsilon ,0\big)$};
    \coordinate (Origin)   at (0,0);
    \coordinate (XAxisMin) at (-4,0);
    \coordinate (XAxisMax) at (1,0);
    \coordinate (YAxisMin) at (0,-2);
    \coordinate (YAxisMax) at (0,2);
    \draw [thick,-latex] (XAxisMin) -- (XAxisMax) 
             node[right] {$\Re z$};
    \draw [thick,-latex] (YAxisMin) -- (YAxisMax)
              node[above] {$\Im z$};   
      \draw [dashed]  (0,0) -- (-4,1);  
      \draw [dashed]  (0,0) -- (-4,-1);   
      \draw (-1,0) to [out=100, in=170] (-0.9,0.225) ;
      \draw (-1.2,0.2) node[left]{$\varepsilon$};
      \draw [dashed] (0,0) circle [radius=0.5];
       \draw [dashed] (-1.8,0) circle [radius=1.8];
      \draw [->] (0,0) -- (0.3, 0.4);      
      \draw  (0.3, 0.1) node{\tiny $\lambda_0$};   
\end{tikzpicture}
\caption{$\Lambda_{\varepsilon,\lambda_0}$} 
\end{subfigure}
\caption{Sectorial regions $\Sigma_{\ep,\lambda_0}$ and $\Lambda_{\ep,\lambda_0}$}
\label{fig:sect}
\end{figure}  
\medskip

Next, we recall the basic theory of the $\CR-$boundedness of operator families (see \cite{DHP2003,KW2004} for further discussions).
\begin{defi}\label{def:R-bounded}
Let $X,Y$ be two Banach spaces and $\CL(X;Y)$ be the collection of all bounded linear operators from $X$ to $Y.$
We say that a family of bounded operators $\tau \subset \CL(X,Y)$ is $\CR$-bounded if for any $N \in {\BBN},$ $T_j \in \tau,$ $x_j \in X$ and the Rademacher functions $r_j (t):= \sign (\sin 2^j \pi t )$ defined for $t \in [0,1],$  the following inequality holds,
\begin{equation*}
\Big\|\sum_{j=1}^N r_j T_j x_j\Big\|_{L_p([0,1];Y)} \leq C_p \Big\|\sum_{j=1}^N r_j x_j\Big\|_{L_p([0,1];X)}
\,\,\, \mbox{for some} \,\,p \in [1,\infty[.
\end{equation*} 
Above the choice of $C_p$ depends only on $p$ but not on $N,$ $T_j,$ $x_j,$ $r_j$ and $1 \leq j\leq N.$ The smallest $C_p$ is called $\CR$-bound of $\tau$, denoted by $R_{\CL(X;Y)}(\tau).$
\end{defi}
Some useful comments on Definition \ref{def:R-bounded}:
\begin{rema}\label{rmk:R-bounded}
Let $X, Y, Z$ be Banach spaces. 
\begin{enumerate}
\item $\tau \subset \CL(X,Y)$ is $\CR$-bounded for any $p\in [1,\infty[$ if $\tau$ is $\CR$-bounded for some $p_0 \in [1,\infty[.$ 
\item Suppose that $\CT$ and $\CS$ are two $\CR-$bounded famlies in $\CL(X;Y).$ Then the sum set
$$\CT + \CS:=\{T+S: T \in \CT, S\in \CS\}$$  
is $\CR-$bounded as well, and 
$\CR_{\CL(X;Y)} (\CT + \CS) \leq \CR_{\CL(X;Y)} (\CT )+\CR_{\CL(X;Y)} (\CS).$ 
\item Assume that the families $\CT\subset \CL(X;Y)$ and $\CS \subset \CL(Y;Z)$ are $\CR-$bounded. Then the composition set
$$\CS\CT :=\{S\circ T: T \in \CT, S\in \CS\} \subset \CL(X;Z)$$  
is $\CR-$bounded as well, and $\CR_{\CL(X;Z)} (\CS\CT) \leq \CR_{\CL(X;Y)} (\CT ) \CR_{\CL(Y;Z)} (\CS).$ 
\item Let $\CT \subset \CL\big(L_{q}(G)\big)$ be a family of operators for $1 < q<\infty$ and some domain $G \subset \BBR^N.$ Then $\CT$ is $\CR-$bounded if and only if there is a constant $C$ such that
\begin{equation*}
\Big\| \big(\sum_{j=1}^{N_0} |T_j f_{j}|^2 \big)^{1\slash 2} \Big\|_{L_q(G)}
\leq C\Big\| \big(\sum_{j=1}^{N_0} |f_{j}|^2 \big)^{1\slash 2} \Big\|_{L_q(G)},
\end{equation*} 
for any $N_0 \in \BBN,$ $f_j \in L_q(G)$ and $T_j \in \CT.$
In particular, $\{ T_{\lambda} : T_{\lambda}f := \lambda^{-s} f, \lambda \in \Sigma_{\ep,\lambda_0}\}$ for $s,\lambda_0>0$ is a $\CR$-bounded family in $\CL\big(L_q(G)\big).$
\end{enumerate}
\end{rema}
To prove the maximal regularity property of the model problem, we need the following theorem on the Fourier multiplier obtained in \cite{Weis2001}.
\begin{theo}[Weis] \label{thm:Weis}
Let $X$ and $Y$ be two UMD Banach spaces 
\footnote{One may refer to \cite{KW2004} for the UMD property.}
and $1<p<\infty.$ Let $M(\cdot)$ be a mapping in  
$C^{1}\big(\BBR\backslash \{0\}; \CL(X;Y) \big)$ such that
\begin{equation*}
\CR_{\CL(X; Y)} \big(\big\{ (\tau \pa_{\tau})^{\ell}   M (\tau) : \tau \in \BBR\backslash \{0\} \big\}\big) 
\leq r_b \quad  (\ell =0,1),
\end{equation*} 
with some constant $r_b>0.$ Then the multiplier operator $T_{M}(\varphi) := \CF^{-1} \big[M\CF[\varphi] \big]$ for any $\varphi \in \CS(\BBR;X)$ can be uniquely extended to a bounded linear operator from $L_p(\BBR;X)$ into $L_p(\BBR;Y)$ with the bound
\begin{equation*}
\|T_{M}\|_{\CL\big(L_p(\BBR;X);L_p(\BBR;Y)\big)} \leq C_{p,X,Y} r_b.
\end{equation*}
\end{theo}

With above definitions and comments, our main result for the model problem \eqref{eq:GR_CNS} is as follows.
\begin{theo}\label{thm:GR_CNS}
Let $0<\ep<\pi\slash 2,$ $\sigma, \mu, \nu>0,$ $1<q,q':=q\slash (q-1)<\infty,$ $N<r<\infty$ and 
$r \geq  \max\{q,q'\}.$ Assume that $\Omega$ is of type $W^{3,2}_r,$ $m\geq \lambda_1(\ep,\Gamma_0)$ by Proposition \ref{prop:resolvent_LB}, and  \eqref{hyp:gamma_GR} is satisfied. 
Set that 
\begin{gather*}
X_q(\Omega) := H^1_q(\Omega) \times L_q(\Omega)^N \times H^{1}_q(\Omega)^N\times W^{2-1\slash q}_q(\Gamma_0),\\
\CX_q(\Omega) :=H^1_q(\Omega) \times L_q(\Omega)^N  \times L_q(\Omega)^N \times H^{1}_q(\Omega)^N\times W^{2-1\slash q}_q(\Gamma_0).
\end{gather*}
For any $(d,\BF,\BG,K) \in X_q(\Omega),$ there exist constants $\lambda_0,r_b \geq 1$ and operator families
\begin{align*}
\CP(\lambda,\Omega) & \in 
\Hol\Big( \Lambda_{\ep,\lambda_0} ; \CL\big(\CX_q(\Omega);H^1_q(\Omega) \big) \Big),\\
\CA(\lambda,\Omega) & \in 
\Hol\Big( \Lambda_{\ep,\lambda_0} ; \CL\big(\CX_q(\Omega);H^2_q(\Omega)^N \big) \Big),\\
\CH(\lambda,\Omega) & \in 
 \Hol\Big( \Lambda_{\ep,\lambda_0} ; \CL\big(\CX_q(\Omega);W^{3-1\slash q}_q(\Gamma_0) \big) \Big),
\end{align*}
such that $(\eta,\Bu,h):=\big( \CP(\lambda,\Omega), \CA(\lambda,\Omega), \CH(\lambda,\Omega)   \big)   (d,\BF,\lambda^{1\slash 2}\BG,\BG,K)$
is the unique solution of \eqref{eq:GR_CNS}.  Moreover, we have 
\begin{gather*}
\CR_{\CL\big(\CX_q(\Omega); H^{1}_q(\Omega) \big)}
 \Big( \Big\{ (\tau \pa_{\tau})^{\ell}\big( \lambda \CP(\lambda,\Omega)\big) :
  \lambda \in \Lambda_{\ep,\lambda_0} \Big\}\Big) \leq r_b,\\
\CR_{\CL\big(\CX_q(\Omega); H^{2-j}_q(\Omega)^N \big)}
 \Big( \Big\{ (\tau \pa_{\tau})^{\ell}\big( \lambda^{j\slash 2}\CA(\lambda,\Omega)\big) :
  \lambda \in \Lambda_{\ep,\lambda_0} \Big\}\Big) \leq r_b,\\
 \CR_{\CL\big(\CX_q(\Omega); W^{3-1\slash q-j'}_q(\Gamma_0) \big)}
 \Big( \Big\{ (\tau \pa_{\tau})^{\ell}\big( \lambda^{j'}\CH(\lambda,\Omega)\big) : 
 \lambda \in \Lambda_{\ep,\lambda_0} \Big\}\Big) \leq r_b,
\end{gather*}
for $\ell, j'=0,1,$ $j=0,1,2,$ and $\tau := \Im \lambda.$
Above the choices of $\lambda_0$ and $r_b$ depend solely on the parameters 
$\ep,$ $\sigma,$ $m,$ $\mu,$ $\nu,$ $q,$ $r,$ $N,$ $\rho_1,$ $\rho_2,$ $\rho_3$ and $\Omega.$
\end{theo}

In fact, we shall see that Theorem \ref{thm:GR_CNS} is a consequence of 
some generalized model problem. 
To this end, let us introduce some parameter $\zeta$ fulfilling $|\zeta|\leq \zeta_0$ and either of the following cases
\begin{equation*}
\hbox{(C1)}\,\, \zeta = \lambda^{-1};\quad
\hbox{(C2)}\,\,\zeta \in \Sigma_{\ep} \,\, \hbox{and} \,\,\Re \zeta <0;\quad 
\hbox{(C3)}\,\,\Re \zeta \geq 0. \hspace*{1cm}
\end{equation*}
Then we define
\begin{equation}\label{eq:Gamma}
\Gamma_{\ep,\lambda_0,\zeta} := 
\begin{cases}
\Lambda_{\ep,\lambda_0} & \hbox{for (C1)},\\
\big\{\lambda \in \BBC : \Re \lambda \geq \big| \frac{\Re \zeta}{\Im \zeta}\big| |\Im \lambda|, 
\,\, \Re \lambda \geq \lambda_0 \big\} & \hbox{for (C2)},\\
\{\lambda \in \BBC :  \Re \lambda \geq \lambda_0 \} & \hbox{for (C3)}.
\end{cases}
\end{equation}
For $\lambda \in \Gamma_{\ep, \lambda_0,\zeta},$  we consider the model problem 
\begin{equation}\label{eq:RR_CNS_0}
\left\{\begin{aligned}
& \lambda  \Bv -\gamma_{1}^{-1}\Di \big( \BBS(\Bv) + \zeta \gamma_3 \di \Bv \BBI \big)
= \Bf  &&\quad\hbox{in}\quad  \Omega,\\
& \big( \BBS(\Bv)  +\zeta \gamma_3 \di \Bv \BBI \big)\Bn_{\Gamma_{0}}
+\sigma (m-\Delta_{\Gamma_{0}})h \,\Bn_{\Gamma_{0}} = \Bg  &&\quad\hbox{on}\quad  \Gamma_0, \\
&\lambda  h - \Bv \cdot \Bn_{\Gamma_0}  = k  &&\quad\hbox{on}\quad  \Gamma_0,\\
&\Bv  = \0 &&\quad\hbox{on}\quad  \Gamma_1,\\
	\end{aligned}\right.
\end{equation} 
where $\gamma_1$ and $\gamma_3$ are uniformly continuous functions on $\overline{\Omega}$ such that 
\begin{gather}\label{hyp:gamma_GR_2}
0<\rho_1 \leq \gamma_1(x) \leq \rho_2,\quad 
0 <\gamma_3(x) \leq \rho_3, \,\,\, \forall \,\, x \in \overline{\Omega},\quad
 \|(\nabla \gamma_{1},\nabla \gamma_{3})\|_{L_r(\Omega)}  \leq \rho_3, 
\end{gather}
for some constants $\rho_1,\rho_2,\rho_3 >0$ and $N<r<\infty.$
The following result concerning \eqref{eq:RR_CNS_0} will be established later.
\begin{theo}\label{thm:RR_CNS}
Let $0<\ep<\pi\slash 2,$ $\sigma, \mu, \nu>0,$ $1<q<\infty,$ $N<r<\infty$ and 
$r \geq  q.$ Assume that $\Omega$ is of type $W^{3,2}_r,$ $m\geq \lambda_1(\ep,\Gamma_0)$ by Proposition \ref{prop:resolvent_LB}, 
and \eqref{hyp:gamma_GR_2} is satisfied.  Set that 
\begin{equation*}
Y_q(\Omega) := L_q(\Omega)^N \times H^{1}_q(\Omega)^N\times H^{2}_q(\Omega),\quad
\CY_q(\Omega) := L_q(\Omega)^N \times Y_q(\Omega).
\end{equation*}
For any $(\Bf,\Bg,k) \in Y_q(\Omega),$ there exist constants $\lambda_0, r_b \geq 1$ and operator families
\begin{align*}
\CA_0(\lambda,\Omega) & \in 
\Hol\Big( \Gamma_{\ep,\lambda_0,\zeta} ; \CL\big(\CY_q(\Omega);H^2_q(\Omega)^N \big) \Big),\\
\CH_0(\lambda,\Omega) & \in 
 \Hol\Big( \Gamma_{\ep,\lambda_0,\zeta} ; \CL\big(\CY_q(\Omega);H^3_q(\Omega) \big) \Big),
\end{align*}
such that $(\Bv, h):=\big(\CA_0(\lambda,\Omega),\CH_0(\lambda,\Omega) \big)(\Bf,\lambda^{1\slash 2}\Bg,\Bg,k)$
is a solutions of \eqref{eq:RR_CNS_0}. Moreover, we have 
\begin{gather*}
\CR_{\CL\big(\CY_q(\Omega); H^{2-j}_q(\Omega)^N \big)}
 \Big( \Big\{ (\tau \pa_{\tau})^{\ell}\big( \lambda^{j\slash 2}\CA_0(\lambda,\Omega)\big) :
  \lambda \in \Gamma_{\ep,\lambda_0,\zeta} \Big\}\Big) \leq r_b,\\
 \CR_{\CL\big(\CY_q(\Omega); H^{3-j'}_q(\Omega) \big)}
 \Big( \Big\{ (\tau \pa_{\tau})^{\ell}\big( \lambda^{j'}\CH_0(\lambda,\Omega)\big) : 
 \lambda \in \Gamma_{\ep,\lambda_0,\zeta} \Big\}\Big) \leq r_b,
\end{gather*}
for $\ell, j'=0,1,$ $j=0,1,2,$ and $\tau := \Im \lambda.$
Above the constants $\lambda_0$ and $r_b$ depend solely on
$\ep,$ $\sigma,$ $m,$ $\mu,$ $\nu,$ $\zeta_0,$ $q,$ $r,$ $N,$ $\rho_1,$ $\rho_2,$ $\rho_3$ and $\Omega.$
\end{theo}

By admitting Theorem \ref{thm:RR_CNS} for a while, it is not hard to prove Theorem \ref{thm:GR_CNS}.
Before that, let us recall a technical lemma from \cite{Shi2013}.
\begin{lemm}\label{lemma:ab_BH}
Let $1\leq q\leq r <\infty$ and $N<r<\infty.$ 
Suppose that $\Omega$ is a uniform $W^{2-1\slash r}_r$ domain.
 Then there exists a positive constant $C=C_{N,r,q,\Omega}$ such that 
\begin{equation*}
\|ab\|_{L_q(\Omega)} \leq \sigma_0 \|b\|_{H^1_q(\Omega)} 
+ C \sigma_0^{-\frac{N}{r-N}} \|a\|_{L_r(\Omega)}^{\frac{r}{r-N}} \|b\|_{L_q(\Omega)},
\end{equation*}
for any $\sigma_0>0,$ $a\in L_r(\Omega)$ and $b \in H^1_q(\Omega).$ In particular, one can replace $\|b\|_{H^1_q(\Omega)}$ by $\|\nabla b\|_{L_q(\Omega)}$ in the inequality above, whenever $\Omega$ is $\BBR^N$ or $\BBR^N_+.$
\end{lemm}

\begin{proof}[The proof of Theorem \ref{thm:GR_CNS}]
By our assumptions on $\Omega,$ there exist linear mappings 
\begin{equation*}
\CR_{\Gamma_0}: H^{3}_q(\Omega) \rightarrow W^{3-1\slash q}_q(\Gamma_0)
\,\,\,\text{and}\,\,\,
\CE_{\Gamma_0}: W^{3-1\slash q}_q(\Gamma_0) \rightarrow H^{3}_q(\Omega),
\end{equation*}
such that $\|\CR_{\Gamma_0}f\|_{W^{n-1\slash q}_q(\Gamma_0)} \leq C_{q,\Omega}\|f\|_{H^n_q(\Omega)}$
and $\|\CE_{\Gamma_0}g\|_{H^{n}_q(\Omega)} \leq C_{q,\Omega}\|g\|_{W^{n-1\slash q}_q(\Gamma_0)}$ for $n=2,3.$
\smallbreak

Next, after eliminating $\eta$ via $\eqref{eq:GR_CNS}_1,$ $(\Bu,h)$ satisfies \eqref{eq:RR_CNS_0}
with $\gamma_3,$ $\Bf,$ $\Bg$ and $k$ given by 
\begin{equation*}
\gamma_3 := \gamma_1\gamma_2,\quad
\Bf:= \gamma_1^{-1} \BF - \gamma_1^{-1} \lambda^{-1}\nabla (\gamma_2  d), \quad
\Bg:= \BG + \lambda^{-1}\gamma_2 d \, \Bn_{\Gamma_0}, \quad 
k := \CR_{\Gamma_0}K.
\end{equation*}
Denote $\CS(\lambda)(d,\BF,\lambda^{1\slash 2}\BG,\BG,K):=(\Bf,\lambda^{1\slash 2} \Bg, \Bg,k)$ 
for any $(d,\BF,\BG,K) \in X_q(\Omega).$
Then Lemma \ref{lemma:ab_BH} yields for $r\geq q,$
\begin{equation}\label{es:R-bdd-S}
 \CR_{\CL\big(\CX_q(\Omega); \CY_q(\Omega) \big)}
 \Big( \Big\{ (\tau \pa_{\tau})^{\ell} \CS(\lambda) : 
 \lambda \in \Sigma_{\ep,1} \Big\}\Big) \leq C_{\rho_1,\rho_2,\rho_3,\Omega}.
\end{equation}
Then we define $\CA(\lambda,\Omega):= \CA_0(\lambda,\Omega) \circ \CS(\lambda),$ 
$\CH(\lambda,\Omega):= \CE_{\Gamma_0} \circ \CH_0(\lambda,\Omega) \circ \CS(\lambda)$ and
\begin{equation*}
\CP(\lambda,\Omega) (d,\BF,\lambda^{1\slash 2}\BG,\BG,K)
:= \lambda^{-1} \Big( d -\gamma_1 \di \big( \CA(\lambda,\Omega)(d,\BF,\lambda^{1\slash 2}\BG,\BG,K)\big) \Big),
\end{equation*}
which are the desired operators due to Theorem \ref{thm:RR_CNS} and \eqref{es:R-bdd-S}.
\smallbreak 

Finally, let us prove the uniqueness. Suppose that $\Bu \in H^2_q(\Omega)^N$ and $h \in W^{3-1\slash q}_q(\Gamma_0)$ satisfy \eqref{eq:GR_CNS} with $(d,\BF,\BG,k)$ vanishing.
For any $\lambda \in \Gamma_{\ep, \lambda_0,\zeta},$ $\BPhi \in C_0^{\infty}(\Omega)^N$ and $\phi \in C^{\infty}_0(\BBR^N),$
we consider that 
\begin{equation}\label{eq:RR_CNS_1}
\left\{\begin{aligned}
& \gamma_{1} \lambda  \Bv -\Di \big( \BBS(\Bv) + \zeta \gamma_3 \di \Bv \BBI \big)
= \BPhi  &&\quad\hbox{in}\quad  \Omega,\\
& \big( \BBS(\Bv)  +\zeta \gamma_3 \di \Bv \BBI \big)\Bn_{\Gamma_{0}}
+\sigma (m-\Delta_{\Gamma_{0}})\theta \,\Bn_{\Gamma_{0}} = \0  &&\quad\hbox{on}\quad  \Gamma_0, \\
&\lambda  \theta - \Bv \cdot \Bn_{\Gamma_0}  = \phi  &&\quad\hbox{on}\quad  \Gamma_0,\\
&\Bv  = \0 &&\quad\hbox{on}\quad  \Gamma_1.\\
	\end{aligned}\right.
\end{equation} 
 Then there exists a solution $(\Bv,\theta) \in H^2_{q'}(\Omega)^N \times W^{3-1\slash {q'}}_{q'}(\Gamma_0)$ of \eqref{eq:RR_CNS_1}. 
 Note the facts that 
 \begin{equation*}
 \sum_{i,j=1}^N \int_{\Omega} \big(\BBS(\Bu) + \zeta\gamma_3 \di \Bu \BBI \big)^i_j \pa_j v^i \,dx 
 = \sum_{i,j=1}^N \int_{\Omega} \big(\BBS(\Bv) + \zeta\gamma_3 \di \Bv \BBI \big)^i_j \pa_j u^i \,dx,
 \end{equation*}
 and $\big( (m-\Delta_{\Gamma_0})h,\theta \big)_{\Gamma_0}=\big( h,(m-\Delta_{\Gamma_0}) \theta \big)_{\Gamma_0}.$
Then it is not hard to see from \eqref{eq:GR_CNS} and \eqref{eq:RR_CNS_1} that
 \begin{equation*}
 (\Bu,\BPhi)_{\Omega} = \big( \sigma (m-\Delta_{\Gamma_0}) h , \phi \big)_{\Gamma_0},
\end{equation*}  
which yields that $\Bu|_{\Omega}=\0$ and $(m-\Delta_{\Gamma_0}) h|_{\Gamma_0}=0$ for the arbitrary choices of $\BPhi$ and $\phi.$ 
As $m-\Delta_{\Gamma_0}$ is invertible by Proposition \ref{prop:resolvent_LB}, $h|_{\Gamma_0}=0.$
This completes the proof of Theorem \ref{thm:GR_CNS}.
\end{proof}

\subsection{Generation of analytic semigroup}
\label{subsec:SG}
In this subsection, we study the following system with nonhomogeneous initial data within the semigroup framework,
\begin{equation}\label{eq:eta-Lame_0}
	\left\{\begin{aligned}
	&\pa_t \eta + \gamma_1 \di \Bu  = 0
	    &&\quad\hbox{in}\quad  \Omega \times \BBR_+, \\
&\gamma_1 \pa_t \Bu -\Di\big(\BBS(\Bu)-\gamma_2 \eta \BBI \big)= \0
		&&\quad\hbox{in}\quad  \Omega \times \BBR_+,\\
		&\big( \BBS(\Bu) - \gamma_2 \eta \BBI\big) \Bn_{\Gamma_0} 
+\sigma (m-\Delta_{\Gamma_0})h \,\Bn_{\Gamma_0} = \0
	    &&\quad\hbox{on}\quad  \Gamma_0 \times \BBR_+, \\
	    &\pa_t  h - \Bu \cdot \Bn_{\Gamma_0}  = 0
	    &&\quad\hbox{on}\quad  \Gamma_0 \times \BBR_+,\\
&\Bu  = \0 &&\quad\hbox{on}\quad  \Gamma_1 \times \BBR_+, \\
		&(\eta, \Bu,h)|_{t=0} = (\eta_0, \Bu_0,h_0) &&\quad\hbox{in}\quad  \Omega.
	\end{aligned}\right.
\end{equation}
Note that the boundary conditions in \eqref{eq:eta-Lame_0} are equivalent to 
\begin{equation}\label{cdt:compatible}
\big(\BBS(\Bu) \Bn_{\Gamma_0}\big)_{\tau}|_{\Gamma_0} = \0,\quad 
\BBS(\Bu)\Bn_{\Gamma_0} \cdot \Bn_{\Gamma_0} -\gamma_2\eta+ \sigma(m-\Delta_{\Gamma_0})h |_{\Gamma_0} =0,\quad
\Bu|_{\Gamma_1}=\0.
\end{equation}
Here $\Bf_{\tau}:=\Bf - (\Bf\cdot \Bn_{\Gamma_0}) \Bn_{\Gamma_0}$ stands for the tangential component of $\Bf$ along $\Gamma_0.$
Then we introduce the functional spaces
\begin{align*}
\FX_q(\Omega)&:=H^1_q(\Omega) \times L_q(\Omega)^{N}\times W^{2-1\slash q}_q(\Gamma_0),\\
\CD_q(\CA)&:=\Big\{ (\eta,\Bu,h) \in \FX_q(\Omega): \Bu \in H^2_q(\Omega)^N,\,\,h\in W^{3-1\slash q}_q(\Gamma_0), \,\,
\text{\eqref{cdt:compatible} holds}\Big\},
\end{align*}
endowed with the norms
\begin{align*}
\|(\eta,\Bu,h)\|_{\FX_q(\Omega)} &:= \|\eta\|_{H^1_q(\Omega)}+\|\Bu\|_{L_q(\Omega)}+\|h\|_{W^{2-1\slash q}_q(\Gamma_0)},\\
\|(\eta,\Bu,h)\|_{\CD_q(\CA)}& := \|\eta\|_{H^1_q(\Omega)}+\|\Bu\|_{H^2_q(\Omega)}+\|h\|_{ W^{3-1\slash q}_q(\Gamma_0)}.
\end{align*}
Furthermore, define the linear operator
\begin{equation*}
\CA \BU:=
\begin{bmatrix}
-\gamma_1 \di \Bu \\
\gamma_1^{-1}\Di\big(\BBS(\Bu)-\gamma_2 \eta \BBI \big) \\
\Bu \cdot \Bn_{\Gamma_0}
\end{bmatrix}
\,\,\,\text{for}\,\,\,\,  \BU:=(\eta,\Bu,h) \in \CD_q(\CA),
\end{equation*}
and the following functional space by the real interpolation theory,
\begin{equation*}
\CD_{q,p}(\Omega) := \big( \FX_q(\Omega), \CD_q(\CA) \big)_{1-1 \slash p,p} 
\subset H^1_q(\Omega) \times B^{2(1-1\slash p)}_{q,p}(\Omega) \times  B^{3-1\slash q-1\slash p}_{q,p}(\Gamma_0),
\end{equation*} 
with $\|(\eta,\Bu,h)\|_{\CD_{q,p}(\Omega)}:=\|\eta\|_{H^1_q(\Omega)} 
+\|\Bu\|_{B^{2(1-1\slash p)}_{q,p}(\Omega)} +\|h\|_{B^{3-1\slash q-1\slash p}_{q,p}(\Gamma_0)}.$
\smallbreak

Thanks to above settings, \eqref{eq:eta-Lame_0} can be regarded as the abstract Cauchy problem
\begin{equation*}
\pa_t \BU - \CA\BU = \0 \,\,\, \text{for}\,\,\, t>0, \quad \BU|_{t=0}=(\eta_0,\Bu_0,h_0),
\end{equation*}
whose resolvent problem is formulated as follows
\begin{equation*}
\lambda \BU - \CA\BU = \BF \,\,\, \text{for}\,\,\, \lambda \in \BBC \,\,\, \text{and} \,\,\, \BF=(d,\Bf,k) \in \FX_q(\Omega).
\end{equation*}
By Theorem \ref{thm:GR_CNS}, there exists $\lambda_0>0$ such that $\Lambda_{\ep,\lambda_0}$ is contained in the resolvent set $\rho(\CA)$ of $\CA.$ 
Moreover, we have
\begin{equation*}
|\lambda| \|\BU\|_{\FX_q(\Omega)} + \|\BU\|_{\CD_q(\CA)} \leq C r_b \|\BF\|_{\FX_q(\Omega)}  
\quad (\lambda \in \Lambda_{\ep,\lambda_0}),
\end{equation*}
for some constant $C>0.$ 
Then by the semigroup theory and interpolation arguments in \cite[Theorem 3.9]{ShiShi2008},
 we can furnish the following results. 
\begin{theo}\label{thm:semigroup}
Let $0<\ep<\pi\slash 2,$ $\sigma, \mu, \nu,\rho_1,\rho_2,\rho_3>0,$ $1<q,q':=q\slash (q-1)<\infty,$ $N<r<\infty$ and 
$r \geq  \max\{q,q'\}.$ Assume that $\Omega$ is of type $W^{3,2}_r,$ $m\geq \lambda_1(\ep,\Gamma_0)$ by Proposition \ref{prop:resolvent_LB}, and  \eqref{hyp:gamma_GR} is satisfied. 
Denote that $\BU_0 := (\eta_0, \Bu_0,h_0).$
Then there exist positive constants $\gamma_0,$ $C$ such that the following assertions hold true.
\begin{enumerate}
\item The operator $\CA$ generates a $C^0$ semigroup $\{T(t)\}_{t\geq 0}$ in $\FX_q(\Omega),$ which is analytic.  
Moreover, we have
\begin{equation*}
\|\BU\|_{\FX_q(\Omega)} + t\big( \|\pa_t  \BU\|_{\FX_q(\Omega)} +\| \BU\|_{\CD_q(\CA)}\big) 
\leq Ce^{\gamma_0 t} \|\BU_0\|_{\FX_q(\Omega)},
\end{equation*}
\begin{equation*}
 \|\pa_t  \BU\|_{\FX_q(\Omega)} +\| \BU\|_{\CD_q(\CA)}
\leq Ce^{\gamma_0 t} \|\BU_0\|_{\CD_q(\CA)},
\end{equation*}
with $\BU:=T(t) \BU_0.$

\item  For any $\BU_0 \in \CD_{q,p}(\Omega),$ \eqref{eq:eta-Lame_0} admits a unique solution
\begin{equation*}
e^{-\gamma_0 t}(\eta, \Bu,h) \in  H^1_{p} \big(\BBR_+;\FX_q(\Omega)\big) 
\cap L_{p} \big(\BBR_+;\CD_q(\CA)\big),
\end{equation*}
satisfying the estimates
\begin{equation*}
\|e^{-\gamma_0 t}\pa_t(\eta, \Bu,h)\|_{L_p(\BBR_+;\FX_q(\Omega))}
+\|e^{-\gamma_0 t}(\eta, \Bu,h)\|_{L_p(\BBR_+;\CD_q(\CA))}
\leq C \|(\eta_0, \Bu_0,h_0)\|_{\CD_{q,p}(\Omega)}.
\end{equation*}
\end{enumerate}
\end{theo}

\subsection{Maximal $L_p-L_q$ regularity}
In this subsection, we consider the following linear evolution equations with trivial initial data,
\begin{equation}\label{eq:eta-Lame_1}
	\left\{\begin{aligned}
	&\pa_t \eta + \gamma_1 \di \Bu  = d
	    &&\quad\hbox{in}\quad  \Omega \times \BBR_+, \\
&\gamma_1 \pa_t \Bu -\Di\big(\BBS(\Bu)-\gamma_2 \eta \BBI \big)= \Bf
		&&\quad\hbox{in}\quad  \Omega \times \BBR_+,\\
		&\big( \BBS(\Bu) - \gamma_2 \eta \BBI\big) \Bn_{\Gamma_0} 
+\sigma (m-\Delta_{\Gamma_0})h \,\Bn_{\Gamma_0} = \Bg
	    &&\quad\hbox{on}\quad  \Gamma_0 \times \BBR_+, \\
	    &\pa_t  h - \Bu \cdot \Bn_{\Gamma_0}  = k
	    &&\quad\hbox{on}\quad  \Gamma_0 \times \BBR_+,\\
&\Bu  = \0 &&\quad\hbox{on}\quad  \Gamma_1 \times \BBR_+, \\
		&(\eta, \Bu,h)|_{t=0} = (0, \0,0) &&\quad\hbox{in}\quad  \Omega.
	\end{aligned}\right.
\end{equation}
To describe the main result for \eqref{eq:eta-Lame_1}, we firstly introduce some useful notations.
For $\lambda =\gamma+ i\tau \in \BBC,$ the Laplace transform and its inverse are formulated by
\begin{equation*}
\CL [f] (\lambda):= \int_{\BBR} e^{-\lambda t} f(t) dt = \CF_t [e^{-\gamma t}f(t)](\tau), \quad 
\CL^{-1}[g] (t):= \frac{1}{2\pi}\int_{\BBR} e^{\lambda t} g(\tau) d\tau =e^{\gamma t}\CF^{-1}_{\tau} [g(\tau)](t).
\end{equation*}
For any $X$ valued function $f,$ we set $\Lambda^s_{\gamma} f (t):= \CL^{-1} \big[ \lambda^s \CL[f](\lambda) \big]$ for any $s>0.$
Then the Bessel potential spaces are defined as follows,
\begin{align*}
H^{s}_{p,\gamma}(\BBR;X) &:= \{f \in L_p(\BBR;X): e^{-\gamma t}  (\Lambda^s_\gamma f)(t) \in L_p(\BBR;X)\},\\
H^{s}_{p,\gamma,0}(\BBR;X) &:= \{f \in H^{s}_{p,\gamma}(\BBR;X): f(t)=0 \,\, \text{for}\,\,t<0\},
\end{align*}
for any $\gamma>0$ and $1<p<\infty.$ For any $(p,q) \in ]1,\infty[^2,$ we say $(d,\Bf,\Bg,k) \in \sF_{p,q,\gamma}$ ($\gamma>0$), 
if $d,$ $\Bf,$ $\Bg$ and $k$ fulfil that
\begin{gather*}
d\in L_{p,\gamma,0}\big(\BBR;H^1_q(\Omega)\big), \quad 
\Bf \in L_{p,\gamma,0}\big(\BBR;L_q(\Omega)^N\big),\\
\Bg \in  L_{p,\gamma,0}\big(\BBR;H^1_q(\Omega)^N\big) 
\cap H^{1\slash 2}_{p,\gamma,0}\big(\BBR;L_q(\Omega)^N\big),\quad 
k \in L_{p,\gamma,0}\big(\BBR;W^{2-1\slash q}_q(\Gamma_0)\big),
\end{gather*}
with the quantity
\begin{align*}
\|(d,\Bf,\Bg,k)\|_{\sF_{p,q,\gamma}}
:= &\|e^{-\gamma t}d\|_{L_p(\BBR;H^1_q(\Omega))}   
+  \|e^{-\gamma t}(\Bf,\Lambda^{1\slash 2}_{\gamma}\Bg) \|_{L_p(\BBR;L_q(\Omega))}   \\
&+  \|e^{-\gamma t}\Bg\|_{L_p(\BBR;H^1_q(\Omega))} 
+\|e^{-\gamma t}k\|_{L_p(\BBR;W^{2-1\slash q}_q(\Gamma_0))} <\infty.
\end{align*}
Thanks to Theorem \ref{thm:GR_CNS}, we can prove the following result.
\begin{theo}\label{thm:maximal}
Let $0<\ep<\pi\slash 2,$ $\sigma, \mu, \nu,\rho_1,\rho_2,\rho_3>0,$ $1<q,q':=q\slash (q-1)<\infty,$ $N<r<\infty$ and 
$r \geq  \max\{q,q'\}.$ Assume that $\Omega$ is of type $W^{3,2}_r,$ $m\geq \lambda_1(\ep,\Gamma_0)$ by Proposition \ref{prop:resolvent_LB}, and  \eqref{hyp:gamma_GR} is satisfied. 
Then there exist constants $\gamma_0,C>0$ such that the following assertions hold true.
For any $(d,\Bf,\Bg,k) \in \sF_{p,q,\gamma_0},$
 \eqref{eq:eta-Lame_1} admits a unique solution 
\begin{gather*}
\eta \in H^1_{p,\gamma_0,0}\big(\BBR;H^1_q(\Omega)\big),\quad 
\Bu \in L_{p,\gamma_0,0}\big(\BBR;H^2_q(\Omega)^N\big) \cap  H^1_{p,\gamma_0,0}\big(\BBR;L_q(\Omega)^N\big),\\
h \in L_{p,\gamma_0,0}\big(\BBR;W^{3-1\slash q}_q(\Gamma_0)\big) \cap  H^1_{p,\gamma_0,0}\big(\BBR;W^{2-1\slash q}_q(\Gamma_0)\big).
\end{gather*}
Moreover, we have
\begin{multline*}
\|e^{-\gamma_0 t}(\pa_t \eta, \eta)\|_{L_p(\BBR;H^1_q(\Omega))} 
+\|e^{-\gamma_0 t}(\pa_t \Bu,\Lambda_{\gamma_0}^{1\slash 2}\nabla \Bu)\|_{L_p(\BBR;L_q(\Omega))}
+\|e^{-\gamma_0 t}  \Bu \|_{L_p(\BBR;H^2_q(\Omega))} \\
+\|e^{-\gamma_0 t} \pa_t h \|_{L_p(\BBR;W^{2-1\slash q}_q(\Gamma_0))} 
+\|e^{-\gamma_0 t} h\|_{L_p(\BBR;W^{3-1\slash q}_q(\Gamma_0))} 
 \leq C\|(d,\Bf,\Bg,k)\|_{\sF_{p,q,\gamma_0}}.
\end{multline*}
\end{theo}

\begin{proof}
For simplicity, we denote the Laplace transform of $f$ by $\wh f := \CL[f](\lambda)$ and
 $\BZ_{\lambda}:= (d,\Bf, \Lambda^{1\slash 2}_{\gamma} \Bg, \Bg, k).$
Firstly, by applying the Laplace transform to \eqref{eq:eta-Lame_1}, we find 
\begin{equation}\label{eq:eta-Lame-2}
	\left\{\begin{aligned}
	&\lambda \wh \eta + \gamma_1 \di \wh\Bu  =\wh d
	    &&\quad\hbox{in}\quad  \Omega, \\
&\gamma_1 \lambda \wh\Bu -\Di\big(\BBS(\wh\Bu)-\gamma_2 \wh\eta \BBI \big)= \wh \Bf
		&&\quad\hbox{in}\quad  \Omega,\\
		&\big( \BBS(\wh\Bu) - \gamma_2 \wh\eta \BBI\big) \Bn_{\Gamma_0} 
+\sigma (m-\Delta_{\Gamma_0}) \wh h \,\Bn_{\Gamma_0} = \wh \Bg
	    &&\quad\hbox{on}\quad  \Gamma_0, \\
	    &\lambda \wh h -\wh \Bu \cdot \Bn_{\Gamma_0}  = \wh k
	    &&\quad\hbox{on}\quad  \Gamma_0,\\
&\wh \Bu  = \0 &&\quad\hbox{on}\quad  \Gamma_1.
	\end{aligned}\right.
\end{equation}
Then Theorem \ref{thm:GR_CNS} and Theorem \ref{thm:Weis} imply that  
\begin{equation*}
(\eta,\Bu,h):= \Big( \CL^{-1} \big[ \CP(\lambda,\Omega) \wh\BZ_{\lambda} \big],
 \CL^{-1} \big[ \CA(\lambda,\Omega) \wh\BZ_{\lambda} \big], \CL^{-1} \big[ \CP(\lambda,\Omega) \wh\BZ_{\lambda} \big] \Big).
\end{equation*}
is a solution of \eqref{eq:eta-Lame_1}. Moreover, there exists $\lambda \in \Lambda_{\ep,\lambda_0}$ with $\Re \lambda=\gamma_0\geq 1$ such that
\begin{multline*}
\|e^{-\gamma_0 t}\pa_t \eta\|_{L_p(\BBR;H^1_q(\Omega))} + \|e^{-\gamma_0 t}\pa_t \Bu\|_{L_p(\BBR;L_q(\Omega))}
+ \sum_{j=0,1,2} \|e^{-\gamma_0 t}\Lambda^{j\slash 2}_{\gamma_0} \Bu\|_{L_p(\BBR;H^{2-j}_q(\Omega))}\\
+ \|e^{-\gamma_0 t}\pa_t h \|_{L_p(\BBR;W^{2-1\slash q}_q(\Omega))} 
+ \|e^{-\gamma_0 t} h \|_{L_p(\BBR;W^{3-1\slash q}_q(\Gamma_0))} 
\leq C_p r_b \|(d,\Bf,\Bg,k)\|_{\sF_{p,q,\gamma_0}}.
\end{multline*}
\medskip

Next, in order to verify $(\eta,\Bu,h)(t)=(0,\0,0)$ for any $t< 0,$ we consider the following dual problem,
\begin{equation}\label{eq:eta-Lame-3}
	\left\{\begin{aligned}
	&\pa_t \rho - \gamma_1 \di \Bv  = 0
	    &&\quad\hbox{in}\quad  \Omega \times ]-\infty,T[, \\
&\gamma_1 \pa_t \Bv +\Di\big(\BBS(\Bv) - \gamma_2 \rho \BBI \big)= \0
		&&\quad\hbox{in}\quad  \Omega \times  ]-\infty,T[,\\
		&\big( \BBS(\Bv) - \gamma_2 \rho \BBI\big) \Bn_{\Gamma_0} 
+\sigma (m-\Delta_{\Gamma_0})\theta \,\Bn_{\Gamma_0} = \0
	    &&\quad\hbox{on}\quad  \Gamma_0 \times  ]-\infty,T[, \\
	    &\pa_t  \theta + \Bv \cdot \Bn_{\Gamma_0}  = 0
	    &&\quad\hbox{on}\quad  \Gamma_0 \times  ]-\infty,T[,\\
&\Bv  = \0 &&\quad\hbox{on}\quad  \Gamma_1 \times  ]-\infty,T[, \\
		&(\rho, \Bv,\theta)|_{t=T} = (\rho_0, \Bv_0, \theta_0) &&\quad\hbox{in}\quad  \Omega,
	\end{aligned}\right.
\end{equation}
with $T\in \BBR$ and $ (\rho_0, \Bv_0, \theta_0) \in C^{\infty}_0(\BBR^N)^{2+N}$ satisfying \eqref{cdt:compatible}.
As $(\wt\rho, \wt\Bv,\wt\theta)(t):=(\rho, \Bv,\theta)(T-t)$ satisfies \eqref{eq:eta-Lame_0},
we infer from Theorem \ref{thm:semigroup} that
\footnote{Here $\gamma_0$ in Theorem \ref{thm:maximal} can be chosen large enough such that the estimates in Theorem \ref{thm:semigroup} hold true. }
\begin{multline*}
\|e^{-\gamma_0 t}(\pa_t \rho,\rho) \|_{L_p(-\infty,T;H^1_q(\Omega))} + \|e^{-\gamma_0 t}\pa_t \Bv\|_{L_p(-\infty,T;L_q(\Omega))}
+ \|e^{-\gamma_0 t} \Bv\|_{L_p(-\infty,T;H^{2}_q(\Omega))}\\
+ \|e^{-\gamma_0 t}\pa_t \theta \|_{L_p(-\infty,T;W^{2-1\slash q}_q(\Gamma_0))} 
+ \|e^{-\gamma_0 t} \theta \|_{L_p(-\infty,T;W^{3-1\slash q}_q(\Gamma_0))} 
\leq C\|(\rho_0,\Bv_0,h_0)\|_{\CD_{q,p}(\Omega)}.
\end{multline*}
Note the fact that $m-\Delta_{\Gamma_0}$ is symmetric. 
Then by \eqref{eq:eta-Lame_1} and \eqref{eq:eta-Lame-3}, we have 
\begin{multline}\label{eq:dual-1}
\frac{d}{dt}\Big( (\gamma_1 \Bu,\Bv)_{\Omega} -(\gamma_2\gamma_1^{-1} \eta, \rho)_{\Omega}
 -\big( \sigma(m-\Delta_{\Gamma_0}) h, \theta\big)_{\Gamma_0} \Big)(t) 
 = -(\gamma_2\gamma_1^{-1}\rho,d)_{\Omega}(t) +(\Bv,\Bf)_{\Omega}(t)\\
 + (\Bv,\Bg)_{\Gamma_0}(t) - \big( \sigma(m-\Delta_{\Gamma_0}) \theta, k\big)_{\Gamma_0} (t),
\end{multline}
for any $t \leq T\in \BBR.$
Then integrating \eqref{eq:dual-1} on $]-\infty,T]$ yields that 
\begin{multline}\label{eq:dual-2}
\Big( (\gamma_1 \Bu,\Bv_0)_{\Omega} -(\gamma_2\gamma_1^{-1} \eta, \rho_0)_{\Omega}
  -\big( \sigma(m-\Delta_{\Gamma_0}) h, \theta_0\big)_{\Gamma_0} \Big)(T) \\
 =- \int_{-\infty}^T \Big( (\gamma_2\gamma_1^{-1}\rho,d)_{\Omega}-(\Bv,\Bf)_{\Omega}
 -(\Bv,\Bg)_{\Gamma_0} +\big( \sigma(m-\Delta_{\Gamma_0}) \theta, k\big)_{\Gamma_0} \Big) (t) \, dt.
\end{multline}
\medskip

If $d(t),\Bf(t),\Bg(t)$ and $k(t)$ vanish for $t<0,$ then \eqref{eq:dual-2} gives us   
$$ \big(\gamma_1 \Bu(T),\Bv_0\big)_{\Omega} -\big(\gamma_2\gamma_1^{-1} \eta(T), \rho_0\big)_{\Omega}
  -\big( \sigma(m-\Delta_{\Gamma_0}) h(T), \theta_0\big)_{\Gamma_0}  =0, \,\,\, \forall\,\,T<0. $$
As $\rho_0,\Bv_0$ and $\theta_0$ are arbitrary smooth functions, $\gamma_1,\gamma_2>0$ and $m \in \rho(\Delta_{\Gamma_0}),$ 
we have $(\eta,\Bu,h)(T)=(0,\0,0)$ for any $T< 0.$
 \smallbreak
  
At last, we prove the uniqueness of \eqref{eq:eta-Lame_1}. Suppose that $(\eta_1,\Bu_1, h_1)$ and $(\eta_2,\Bu_2, h_2)$ are two solutions of \eqref{eq:eta-Lame_1}. Then $(\eta,\Bu, h):=(\eta_2-\eta_1,\Bu_2-\Bu_1,h_2- h_1)$ satisfies \eqref{eq:eta-Lame_1}
by imposing $(d,\Bf,\Bg,k)=(0,\0,\0,0).$  Then \eqref{eq:dual-2} yields that 
  $(\eta,\Bu,h)(T)=(0,\0,0)$ for any $T>0.$
\end{proof}


\section{Generalized model problem in the half space}
\label{sec:halfspace}
In order to prove Theorem \ref{thm:RR_CNS}, we consider the following model problem in $\BBR^N_+$ in this section,  
\begin{equation}\label{eq:RR_CNS_half}
\left\{\begin{aligned}
& \lambda  \Bu -\gamma_{1}^{-1} \Di \big( \BBS(\Bu) + \zeta \gamma_3 \di \Bu \BBI \big)
= \BF &&\quad\hbox{in}\quad  \BBR^N_+,\\
&\big(\BBS(\Bu)  +\zeta \gamma_3 \di \Bu \BBI \big)\Bn_{0}
+\sigma (m-\Delta')h \,\Bn_{0} = \BG &&\quad\hbox{on}\quad  \BBR^N_0, \\
&\lambda  h - \Bu \cdot \Bn_{0}  = K &&\quad\hbox{on}\quad  \BBR^N_0,\\
	\end{aligned}\right.
\end{equation}
where $\BBR^N_0:= \{x=(x',x_{_N})\in \BBR^N: x_{_N}=0\},$ $\Bn_0 := (0,\dots,0,-1)^{\top}$ and $\Delta' := \sum_{j=1}^{N-1}\pa_j^2.$ 
Moreover, the parameter $\zeta$ and the constants $\gamma_1,$ $\gamma_3$ fulfil the conditions
\begin{equation}\label{hyp:half_space}
|\zeta| \leq \zeta_0, \,\,\,
0<\rho_1  \leq \gamma_1 \leq  \rho_2, \,\,\,
0 < \gamma_3 \leq \rho_3
\end{equation}
for some $\rho_1,\rho_2,\rho_3>0.$
Then recalling the definition of $\Gamma_{\ep,\lambda_0,\zeta}$ in \eqref{eq:Gamma}, our main result for \eqref{eq:RR_CNS_half} reads:
\begin{theo}\label{thm:GR_half_0}
Assume that $0<\ep<\pi\slash 2,$ $\sigma, m,\mu, \nu>0,$ $1<q<\infty$ and \eqref{hyp:half_space} is satisfied. Set that 
\begin{gather*}
Y_q(\BBR^N_+) := L_q(\BBR^N_+)^N \times H^{1}_q(\BBR^N_+)^N\times H^{2}_q(\BBR^N_+),\quad 
\CY_q(\BBR^N_+) := L_q(\BBR^N_+)^N \times Y_q(\BBR^N_+).
\end{gather*}
For any $(\BF,\BG,K) \in Y_q(\BBR^N_+),$ there exist constants $\lambda_0, r_b \geq 1$ and operator families
\begin{align*}
\CA_0(\lambda,\BBR^N_+) & \in 
\Hol\Big( \Gamma_{\ep,\lambda_0,\zeta} ; \CL\big(\CY_q(\BBR^N_+);H^2_q(\BBR^N_+)^N \big) \Big),\\
\CH_0(\lambda,\BBR^N_+) & \in 
 \Hol\Big( \Gamma_{\ep,\lambda_0,\zeta} ; \CL\big(\CY_q(\BBR^N_+);H^3_q(\BBR^N_+) \big) \Big),
\end{align*}
such that $(\Bu, h):=\big( \CA_0(\lambda,\BBR^N_+), \CH_0(\lambda,\BBR^N_+) \big) (\BF,\lambda^{1\slash 2}\BG,\BG,K)$
is a solution of \eqref{eq:RR_CNS_half}. Moreover, we have
\begin{gather*}
\CR_{\CL\big(\CY_q(\BBR^N_+); H^{2-j}_q(\BBR^N_+)^N \big)}
 \Big( \Big\{ (\tau \pa_{\tau})^{\ell}\big( \lambda^{j\slash 2}\CA_0(\lambda,\BBR^N_+)\big) : \lambda \in 
 \Gamma_{\ep,\lambda_0,\zeta} \Big\}\Big) \leq r_b,\\
 \CR_{\CL\big(\CY_q(\BBR^N_+); H^{3-j'}_q(\BBR^N_+) \big)}
 \Big( \Big\{ (\tau \pa_{\tau})^{\ell}\big( \lambda^{j'}\CH_0(\lambda,\BBR^N_+)\big) : \lambda \in 
\Gamma_{\ep,\lambda_0,\zeta} \Big\}\Big) \leq r_b,
\end{gather*}
for $\ell,j' =0,1,$ $j=0,1,2,$ and $\tau := \Im \lambda.$ Above the choices of $\lambda_0$ and $r_b$ depend solely on 
$\ep,$ $\sigma,$ $m,$ $\mu,$ $\nu,$ $q,$ $N,$ $\zeta_0,$ $\rho_1,$ $\rho_2,$ $\rho_3.$
\end{theo}

\subsection{Reduction and the main idea}
For convenience of the later calculations, we introduce the notations
\begin{equation*}
\alpha :=\gamma_1^{-1} \mu,\,\,\,
 \beta := \gamma_1^{-1}(\nu-\mu),\,\,\,
 \zeta':=  \gamma_1^{-1}\gamma_3 \zeta,\,\,\,
 \sigma':=\gamma_1^{-1}  \sigma,\,\,\,
 \BG':=\gamma_1^{-1}\BG.
\end{equation*}
Then \eqref{eq:RR_CNS_half} is equivalent to
\begin{equation}\label{eq:GR_half_0}
	\left\{\begin{aligned}
&\lambda \Bu - \alpha \Delta \Bu 
-(\alpha +\beta +\zeta') \nabla \di \Bu
 = \BF
		&&\quad\hbox{in}\quad  \BBR^N_+,\\
		&\alpha (\pa_{_N} u_j+ \pa_j u_{_N})= -G'_j, \,\, j=1,\dots, N-1,
	    &&\quad\hbox{on}\quad  \BBR^N_0, \\
		& 2\alpha \pa_{_N} u_{_N} +(\beta + \zeta') \di \Bu
+\sigma' (m-\Delta')h  =-G'_{N}
	    &&\quad\hbox{on}\quad  \BBR^N_0, \\
	    &\lambda  h + u_{_N}  = K
	    &&\quad\hbox{on}\quad  \BBR^N_0.\\
	\end{aligned}\right.
\end{equation}

Now let us outline the main idea of solving \eqref{eq:GR_half_0}.
Clearly $\Bu:=\Bv+\Bw$ and $h$ satisfy \eqref{eq:GR_half_0}, if $\Bv$ and $(\Bw,h)$ satisfy the following linear systems respectively,
\begin{equation}\label{eq:GR_half_1}
	\left\{\begin{aligned}
&\lambda \Bv - \alpha \Delta \Bv 
-(\alpha +\beta +\zeta') \nabla \di \Bv
 = \BF
		&&\quad\hbox{in}\quad  \BBR^N_+,\\
		&\alpha (\pa_{_N} v_j + \pa_j v_{_N})= -G'_j, \,\, j=1,\dots, N-1,
	    &&\quad\hbox{on}\quad  \BBR^N_0, \\
		& 2\alpha \pa_{_N} v_{_N} +(\beta + \zeta') \di \Bv  =-G'_{N}
	    &&\quad\hbox{on}\quad  \BBR^N_0, \\
	\end{aligned}\right.
\end{equation}
\begin{equation}\label{eq:GR_half_2}
	\left\{\begin{aligned}
&\lambda \Bw - \alpha \Delta \Bw 
-(\alpha +\beta +\zeta') \nabla \di \Bw = \0
		&&\quad\hbox{in}\quad  \BBR^N_+,\\
		&\alpha (\pa_{_N} w_j+ \pa_j w_{_N})= 0, \,\, j=1,\dots, N-1,
	    &&\quad\hbox{on}\quad  \BBR^N_0, \\
		& 2\alpha \pa_{_N} w_{_N} +(\beta + \zeta') \di \Bw
+\sigma' (m-\Delta')h  = 0
	    &&\quad\hbox{on}\quad  \BBR^N_0, \\
	    &\lambda  h + w_{_N}  = K-v_{_N}
	    &&\quad\hbox{on}\quad  \BBR^N_0.\\
	\end{aligned}\right.
\end{equation}
For the solvability of \eqref{eq:GR_half_1}, we refer to \cite[Theorem 2.3]{GS2014}.
\begin{prop}\label{prop:GR_half_1}
Assume that $0<\ep<\pi\slash 2,$ $\mu, \nu>0,$ $1<q<\infty$ and \eqref{hyp:half_space} is satisfied. 
\begin{equation*}
X_q(\BBR^N_+):= L_q(\BBR^N_+)^N \times H^1_q(\BBR^N_+)^N, \quad 
\CX_q(\BBR^N_+):= L_q(\BBR^N_+)^N \times X_q(\BBR^N_+).
\end{equation*}
For any $(\BF,\BG') \in X_q(\BBR^N_+),$ 
there exist constants $\lambda_0,r_b \geq 1$ and a family of operators
\begin{equation*}
\CA_1(\lambda, \BBR^N_+) \in 
\Hol\Big( \Gamma_{\ep,\lambda_0,\zeta} ; \CL\big(\CX_q(\BBR^N_+);H^2_q(\BBR^N_+)^N \big) \Big),
\end{equation*}
such that $\Bv := \CA_1(\lambda; \BBR^N_+)\, (\BF,\lambda^{1\slash2}\BG',\BG')$ solves \eqref{eq:GR_half_1}. 
Moreover, we have 
\begin{equation*}
\CR_{\CL\big(\CX_q(\BBR^N_+); H^{2-j}_q(\BBR^N_+)^N \big)}
 \Big( \Big\{ (\tau \pa_{\tau})^{\ell}\big( \lambda^{j\slash 2}\CA_1(\lambda,\BBR^N_+)\big) : \lambda \in 
\Gamma_{\ep,\lambda_0,\zeta} \Big\}\Big) \leq r_b,
\end{equation*}
for $\ell =0,1,$ $j=0,1,2,$ $\tau := \Im \lambda.$ Above the choices of $\lambda_0$ and $r_b$ depend solely on 
$\ep,$ $\mu,$ $\nu,$  $q,$ $N,$ $\zeta_0,$ $\rho_1,$ $\rho_2,$ $\rho_3.$
\end{prop}
\medskip

Hereafter, we replace $(\sigma',\zeta')$ by $(\sigma,\zeta)$ for simplicity
\footnote{Thanks to \eqref{hyp:half_space}, the new definition of $\zeta$ is harmless to the shape of $\Gamma_{\ep,\lambda_0,\zeta}$ in \eqref{eq:Gamma}.}. 
To handle \eqref{eq:GR_half_2},  it suffices to consider
\begin{equation}\label{eq:GR_half_3}
	\left\{\begin{aligned}
&\lambda \Bu - \alpha \Delta \Bu 
-(\alpha +\beta +\zeta) \nabla \di \Bu = \0
		&&\quad\hbox{in}\quad  \BBR^N_+,\\
		&\alpha (\pa_{_N} u_j+ \pa_j u_{_N})= 0, \,\, j=1,\dots, N-1,
	    &&\quad\hbox{on}\quad  \BBR^N_0, \\
		& 2\alpha \pa_{_N} u_{_N} +(\beta + \zeta) \di \Bu
+\sigma (m-\Delta')h  = 0
	    &&\quad\hbox{on}\quad  \BBR^N_0, \\
	    &\lambda  h + u_{_N}  = k
	    &&\quad\hbox{on}\quad  \BBR^N_0.\\
	\end{aligned}\right.
\end{equation}
For \eqref{eq:GR_half_3}, we will establish that:
\begin{prop}\label{prop:GR_half_2}
Assume that $0<\ep<\pi\slash 2,$ $\sigma, m,\mu, \nu>0,$ $1<q<\infty$ and \eqref{hyp:half_space} is satisfied. 
For any $k \in H^2_q(\BBR^N_+),$
there exist constants $\lambda_0,r_b \geq 1$ and the operator families 
\begin{align*}
\CW(\lambda,\BBR^N_+) \in & 
\Hol\Big( \Gamma_{\ep,\lambda_0,\zeta} ; \CL\big(H^2_q(\BBR^N_+);H^2_q(\BBR^N_+)^N \big) \Big),\\
\CH(\lambda,\BBR^N_+) \in &
 \Hol\Big( \Gamma_{\ep,\lambda_0,\zeta} ; \CL\big(H^2_q(\BBR^N_+);H^3_q(\BBR^N_+) \big) \Big),
\end{align*}
such that $(\Bu,h):=\big( \CW(\lambda,\BBR^N_+),\CH(\lambda,\BBR^N_+) \big)k$ is a solution of \eqref{eq:GR_half_3}.
Moreover, we have
\begin{align*}
\CR_{\CL\big(H^2_q(\BBR^N_+); H^{2-j}_q(\BBR^N_+)^N \big)}
 \Big( \Big\{ (\tau \pa_{\tau})^{\ell}\big( \lambda^{j\slash 2}\CW(\lambda,\BBR^N_+)\big) : \lambda \in 
\Gamma_{\ep,\lambda_0,\zeta} \Big\}\Big) \leq r_b,\\
 \CR_{\CL\big(H^2_q(\BBR^N_+); H^{3-j'}_q(\BBR^N_+) \big)}
 \Big( \Big\{ (\tau \pa_{\tau})^{\ell}\big( \lambda^{j'}\CH(\lambda,\BBR^N_+)\big) : \lambda \in 
\Gamma_{\ep,\lambda_0,\zeta} \Big\}\Big) \leq r_b,
\end{align*}
for $\ell,j' =0,1,$ $j=0,1,2,$ $\tau := \Im \lambda.$
Above the choices of $\lambda_0$ and $r_b$ depend solely on 
$\ep,$ $\sigma,$ $m,$ $\mu,$ $\nu,$ $q,$ $N,$ $\zeta_0,$ $\rho_1,$ $\rho_2,$ $\rho_3.$
\end{prop}

Now let us point out that Theorem \ref{thm:GR_half_0} is an immediate consequence from Proposition \ref{prop:GR_half_1} and Proposition \ref{prop:GR_half_2} above.
\begin{proof}[Completion of the proof of Theorem \ref{thm:GR_half_0}]
Set that $v_{_N} :=\CA_{1N}(\lambda; \BBR^N_+)\, (\BF,\lambda^{1\slash2}\BG',\BG')$ due to Proposition \ref{prop:GR_half_1} for $(\BF, \BG') \in X_q(\BBR^N_+).$ Then thanks to \eqref{eq:GR_half_2} and Proposition \ref{prop:GR_half_2}, introduce
\begin{align*}
\Bw :=& \CA_2(\lambda,\BBR^N_+) (\BF,\lambda^{1\slash 2}\BG',\BG',K)
:=\CW(\lambda,\BBR^N_+) \big( K - \CA_{1N}(\lambda; \BBR^N_+)\, (\BF,\lambda^{1\slash2}\BG',\BG') \big),\\
h :=&\CH_0(\lambda,\BBR^N_+)  (\BF,\lambda^{1\slash 2}\BG',\BG',K)
:= \CH(\lambda,\BBR^N_+)  \big( K 
- \CA_{1N}(\lambda; \BBR^N_+)\, (\BF,\lambda^{1\slash2}\BG',\BG')\big),
\end{align*}
and $(\Bw, h)$ solves \eqref{eq:GR_half_2} clearly.
Finally, we define
\begin{align*}
\Bu&:= \CA_0 (\lambda, \BBR^N_+) (\BF,\lambda^{1\slash 2}\BG',\BG', K) 
:=\CA_1(\lambda; \BBR^N_+)\, (\BF,\lambda^{1\slash2}\BG',\BG') 
+ \CA_2(\lambda,\BBR^N_+) (\BF,\lambda^{1\slash 2}\BG',\BG',K).
\end{align*}
Then thanks to Proposition \ref{prop:GR_half_1}, Proposition \ref{prop:GR_half_2} and Remark \ref{rmk:R-bounded}, $\CA_0(\lambda,\BBR^N_+)$ and $\CH_0(\lambda,\BBR^N_+)$ are the desired operator families.
\end{proof}

 To tackle \eqref{eq:GR_half_3},  our main task is  to construct $\CW(\lambda,\BBR^N_+)$ and $\CH(\lambda,\BBR^N_+)$ in Proposition \ref{prop:GR_half_2}.
To this end, we apply the partial fourier transformation $\CF_{x'}$ to \eqref{eq:GR_half_3},
\begin{equation}\label{eq:GR_half_4}
	\left\{\begin{aligned}
&(\lambda + \alpha |\xi'|^2) \wh{u_j} - \alpha \pa_{N}^2 \wh{u_j} 
-(\alpha +\beta +\zeta) (i\xi_j) 
(i\xi'\cdot \wh{\Bu'} + \pa_{_N} \wh{u_{_N}}) = 0
		&&\quad\hbox{for}\quad  x_{_N}>0,\\
&(\lambda + \alpha |\xi'|^2) \wh{u_{_N}} - \alpha \pa_{N}^2 \wh{u_{_N}} 
-(\alpha +\beta +\zeta) \pa_{_N} (i\xi'\cdot \wh{\Bu'} 
+ \pa_{_N} \wh{u_{_N}}) = 0
		&&\quad\hbox{for}\quad  x_{_N}>0,\\
		&\alpha (\pa_{_N} \wh{u_j} + i\xi_j \wh{u_{_N}})= 0, \,\, j=1,\dots, N-1,
	    &&\quad\hbox{for}\quad  x_{_N}=0, \\
		& 2\alpha \pa_{_N} \wh{u_{_N}} +(\beta + \zeta) 
		 (i\xi'\cdot \wh{\Bu'}+\pa_{_N} \wh{u_{_N}})
+\sigma (m+|\xi'|^2) \wh{h}  = 0
	    &&\quad\hbox{for}\quad  x_{_N}=0, \\
	    &\lambda  \wh{h} + \wh{u_{_N}}  =\wh{k}
	    &&\quad\hbox{for}\quad  x_{_N}=0.
	\end{aligned}\right.
\end{equation}
For convenience, we take advantage of the following notations,
\begin{gather}\label{eq:AB}
A:= \sqrt{(2\alpha +\beta +\zeta)^{-1} \lambda +|\xi'|^2},\qquad 
 B:=\sqrt{\alpha^{-1}\lambda +|\xi'|^2},\\ \nonumber
\eta := \alpha^{-1}(\alpha +\beta +\zeta),\qquad 
\CM(x_{_N}) := (B-A)^{-1} (e^{-Bx_{_N}} - e^{-Ax_{_N}}),
\end{gather}
and the so-called \emph{Lopatinski} matrix $\BBL=[L_{ij}]_{2\times 2}$ as well, whose entries are defined by
\begin{gather}\label{eq:L_ij}
L_{11} := \frac{\alpha A (B^2 - |\xi'|^2)}{AB-|\xi'|^2}  \ccomma \qquad
L_{12}:= \frac{\alpha |\xi'|^2 (2AB - |\xi'|^2-B^2)}{AB-|\xi'|^2}  \ccomma \\ \nonumber
L_{21}:=\frac{2\alpha A(B-A)-(\beta +\zeta) (A^2-|\xi'|^2)}{AB-|\xi'|^2}  \ccomma
\qquad
L_{22}:=\frac{(2\alpha+\beta+\zeta)B (A^2-|\xi'|^2)}{AB-|\xi'|^2} \cdot
\end{gather}
By direct calculations, we can conclude the formulas (see \cite[Sec.4]{GS2014}),
\begin{align}\label{eq:GR_half_uj_1}
\wh{u_j} (\xi', x_{_N}) 
=&- \frac{\eta (i\xi_j)(L_{12}+B L_{11})}{(A+B) B \det{\BBL}}
 \frac{|\xi'|^2 -A^2}{AB- |\xi'|^2} 
 \Big( B\CM(x_{_N}) -e^{-Bx_{_N}} \Big)
\sigma (m + |\xi'|^2) \widehat{h} (\xi',0) \\\nonumber
&+\frac{(i\xi_j)L_{11}}{B\det{\BBL}}e^{-Bx_{_N}} 
    \sigma (m + |\xi'|^2) \widehat{h} (\xi',0),
  \quad  \forall\,\, j =1, \dots,N-1, \\ \label{eq:GR_half_uN_1}
\widehat{u_{_N}} (\xi', x_{_N})  
=&\Big( \frac{\eta A(L_{12}+B L_{11})}{(A+B) \det{\BBL}}
 \frac{|\xi'|^2 -A^2}{AB- |\xi'|^2} \CM(x_{_N})
 + \frac{L_{11}}{ \det{\BBL} }e^{-Bx_{_N}} \Big)
\sigma (m + |\xi'|^2) \widehat{h} (\xi',0).   
\end{align}
Then \eqref{eq:GR_half_uN_1} and $\eqref{eq:GR_half_4}_5$ imply formally that
\begin{equation}\label{eq:GR_half_h_1}
\widehat{h} (\xi',0) =\frac{\det \BBL}{N(A,B)}\widehat{k}(\xi',0)
\,\,\,\hbox{with}\,\,\,
N(A, B) := \lambda \det{\BBL} +\sigma L_{11} (m + |\xi'|^2).
\end{equation}

Next, insert \eqref{eq:GR_half_h_1} back into \eqref{eq:GR_half_uj_1}, \eqref{eq:GR_half_uN_1}, 
and take partial Fourier inverse transformation,
\begin{align}\label{eq:GR_half_uj_2}
u_{j}(x) =\CW_{j} (\lambda,\BBR^N_+)k&:= \CF_{\xi'}^{-1} \Big[ n_{j1}(\lambda,\xi') (m+|\xi'|^2) \big( B\CM(x_{_N}) -e^{-Bx_{_N}} \big) \widehat{k} (\xi',0) \Big](x')\\ \notag
& \quad +\CF_{\xi'}^{-1} \Big[ n_{j2}(\lambda,\xi') (m + |\xi'|^2) e^{-Bx_{_N}} \widehat{k} (\xi',0) \Big](x'),\\
u_{_N}(x) = \CW_{N} (\lambda,\BBR^N_+)k&:= \CF_{\xi'}^{-1} \Big[ n_{N1}(\lambda,\xi') (m+|\xi'|^2)B\CM(x_{_N})  \widehat{k} (\xi',0) \Big](x')  \label{eq:GR_half_uN_2}\\
& \quad +\CF_{\xi'}^{-1} \Big[ n_{N2}(\lambda,\xi')  (m+|\xi'|^2) e^{-Bx_{_N}} \widehat{k} (\xi',0) \Big](x'), \notag
\end{align}
where the symbols $n_{J1}(\lambda,\xi')$ and  $n_{J2}(\lambda,\xi')$ ($J=1,\dots,N$) are given by
\begin{align}\label{eq:n_Jk_half}
 n_{j1}(\lambda,\xi') :=&- \frac{\sigma\eta (i\xi_j) (L_{12}+B L_{11})}{  B (A+B) N(A,B)}
 \frac{|\xi'|^2 -A^2}{AB- |\xi'|^2} \ccomma \\ 
 n_{j2}(\lambda,\xi') :=& \frac{\sigma(i\xi_j)L_{11}}{BN(A,B)}, 
 \quad \forall\,\, j=1,\dots, N-1, \notag \\ 
 n_{N1}(\lambda,\xi') :=& \frac{\sigma \eta A (L_{12}+B L_{11})}{B(A+B) N(A,B)}
 \frac{|\xi'|^2 -A^2}{AB- |\xi'|^2}  \ccomma \notag\\
 n_{N2}(\lambda,\xi') :=& \frac{\sigma L_{11}}{N(A,B)} \cdot \notag 
\end{align}

On the other hand,  according to \eqref{eq:GR_half_h_1} and Lemma \ref{lemma:N_AB}, we set $h$ as follows,
\begin{equation}\label{eq:GR_half_h_2}
h (x) =\big(\CH(\lambda,\BBR^N_+) k\big) (x):=
\varphi(x_{_N}) \CF^{-1}_{\xi'}\Big[ 
\frac{\det \BBL}{N(A,B)} \, e^{-|\xi'|x_{_N}}
\widehat{k}(\xi',0)\Big] (x')
\end{equation}
for any $x=(x',x_{_N}) \in \BBR^N_+,$
where the cut-off function $\varphi(s) \in C^{\infty}_0( ]-2,2[)$ satisfies 
$0\leq \varphi(s)\leq 1$ and 
$\varphi(s)=1$ for $|s| <1.$

\subsection{Some preliminary results}
In this subsection, we summarize some results to study the $\CR$-boundedness properties of the operator families in \eqref{eq:GR_half_uj_2}, \eqref{eq:GR_half_uN_2} and \eqref{eq:GR_half_h_2}. Firstly, recall the definitions of the class of the symbols.
\begin{defi} 
Assume that $\Lambda\subset \BBC,$  $s \in \BBR$ and $\kappa' \in \BBN^{N-1}_0.$
Consider some smooth function $m(\lambda,\xi')$ defined in $  \wt\Lambda := \Lambda \times (\BBR^{N-1} \backslash \{0\}) $ such that
\begin{equation*}
|m(\lambda,\xi')|  \leq C_{\Lambda} ( |\lambda|^{1\slash 2} + |\xi'| )^s 
\,\,\,\text{for any}\,\,(\lambda, \xi') \in \wt \Lambda.
\end{equation*} 
\begin{itemize}
\item $m(\lambda, \xi')$ is called a multiplier of order $s$ with type 1 on $\wt\Lambda,$ 
denoted by $m \in \BM_{s,1}(\wt\Lambda),$ if there exists a constant $C_{\kappa',s,\Lambda}$ such that 
\begin{equation*}
\big| \pa_{\xi'}^{\kappa'} (\tau \pa_\tau)^{\ell}m(\lambda,\xi')\big|
  \leq C_{\kappa',s,\Lambda}   ( |\lambda|^{1\slash 2} + |\xi'| )^{s-|\kappa'|}
\end{equation*}
for any $(\lambda, \xi') \in \wt \Lambda,$ $\tau := \Im \lambda$ and $\ell =0,1.$
\item $m(\lambda, \xi')$ is called a multiplier of order $s$ with type 2 on $\wt\Lambda,$ 
denoted by $m \in \BM_{s,2}(\wt\Lambda),$ if there exists a constant $C_{\kappa',s,\Lambda}$ such that 
\begin{equation*}
\big| \pa_{\xi'}^{\kappa'} (\tau \pa_\tau)^{\ell}m(\lambda,\xi')\big|
  \leq C_{\kappa',s,\Lambda}  ( |\lambda|^{1\slash 2} + |\xi'| )^{s} |\xi'|^{-|\kappa'|}
\end{equation*}
for $(\lambda, \xi') \in \wt \Lambda,$ $\tau := \Im \lambda$ and $\ell =0,1.$
\end{itemize}
For any $m\in \BM_{s,i}(\wt\Lambda),$ $i=1,2,$ we denote 
$M(m,\Lambda):= \max_{|\kappa'|\leq N}C_{\kappa',s,\Lambda} .$
\end{defi} 
\begin{rema} \label{rmk:M1M2}
Let $s,s_1, s_2 \in \BBR$ and $i=1,2.$ Denote
$\BM_{s,i} :=\BM_{s,i} \big(\Lambda \times (\BBR^{N-1} \backslash \{0\})\big)$ 
for short. As it was pointed out in \cite[Lemma 5.1]{ShiShi2012}, we have
\begin{enumerate}
\item $m_1m_2 \in \BM_{s_1+s_2,1}$ 
(or $\BM_{s_1+s_2,2}$) for
$m_i \in \BM_{s_i,1}$ (or $m_i \in \BM_{s_i,2}$ resp.);
\item $m_1m_2 \in \BM_{s_1+s_2,2}$ for
$m_i \in \BM_{s_i,i}$ and $i=1,2.$
\end{enumerate}
\end{rema}

Now, let us state some useful results proved in \cite{ES2013,GS2014} on $A,$ $B$ and $\BBL=[L_{jk}]_{2\times 2}$ in \eqref{eq:AB} and \eqref{eq:L_ij}. 
In the following,
we write $\wt \Gamma_{\ep,\lambda_0,\zeta}:= \Gamma_{\ep,\lambda_0,\zeta} \times (\BBR^{N-1}\backslash\{0\})$
\footnote{For the case (C1), $\zeta:=\gamma_1^{-1}\gamma_3 \lambda^{-1}.$} for simplicity.
\begin{lemm}\label{lemma:basic}
Let $0<\ep<\pi \slash 2,$  $\lambda_0,\mu, \nu>0,$ $s\geq 1$ and $\xi'\in \BBR^{N-1}.$
Assume that \eqref{hyp:half_space} is satisfied.
Then the following assertions hold true.
\begin{enumerate}
\item For any $\lambda \in \Sigma_\ep$ and $a>0,$  we have 
$\big| a \lambda + |\xi'|^2 \big| 
\geq \sin (\ep \slash 2) \big( a  |\lambda|+|\xi'|^2 \big).$
\item There exists a constant $0<\ep'<\pi\slash 2$ such that 
\begin{equation*}
(s \alpha + \beta +\zeta)^{-1} \lambda \in \Sigma_{\ep'},
 \,\,\, \forall\,\, \lambda \in \Gamma_{\ep,\lambda_0,\zeta},
\end{equation*}
where the choice of $\ep'$ depends solely on $\ep,s,\mu,\nu,\lambda_0,\zeta_0,\rho_1,\rho_2,\rho_3.$ 

\item For any $\lambda \in \Gamma_{\ep,\lambda_0,\zeta},$ there exist constants $c,C>0$ such that
\begin{equation*}
c \big( |\lambda| + |\xi'|^2 \big)
 \leq \big| (s \alpha + \beta +\zeta)^{-1} \lambda + |\xi'|^2 \big|
  \leq C\big( |\lambda| + |\xi'|^2 \big) ,
\end{equation*}
with $c=c(\ep,s,\mu,\nu,\lambda_0,\zeta_0,\rho_1,\rho_2,\rho_3)$
and $C=C(\ep,s, \mu, \nu,\rho_2).$
\end{enumerate}
\end{lemm}
\begin{lemm}\label{lemma:ABL}
Let $0<\ep<\pi \slash 2,$  $\lambda_0,\mu, \nu>0,$ $s\in \BBR,$ $\kappa'\in \BBN_0^{N-1}$ and $\ell=0,1.$  Assume that  \eqref{hyp:half_space} is satisfied.
Then the following assertions hold true.
\begin{enumerate}
\item For any $(\lambda,\xi')\in \wt \Gamma_{\ep,\lambda_0,\zeta},$ we have
\begin{gather*}
  c_{\ep,\mu,\nu,\lambda_0,\zeta_0,\rho_1,\rho_2,\rho_3}( |\lambda|^{1\slash 2}+|\xi'|) \leq |A| 
  \leq C_{\ep,\mu, \nu,\rho_2} (|\lambda|^{1\slash 2} + |\xi'| ),\\
c_{\ep,\mu,\rho_1,\rho_2} (|\lambda|^{1\slash 2} + |\xi'| ) \leq |B| 
  \leq C_{\ep,\mu,\rho_1,\rho_2} ( |\lambda|^{1\slash 2} + |\xi'| ).
\end{gather*}
In fact, $A^s,B^s \in \BM_{s,1}(\wt\Gamma_{\ep,\lambda_0,\zeta})$ 
for any $s \in \BBR$ with
\begin{gather*}
\big| \pa_{\xi'}^{\kappa'} (\tau \pa_\tau)^{\ell} A^s  \big|
\leq C_{\kappa',s,\ep,\mu,\nu,\lambda_0,\zeta_0,\rho_1,\rho_2,\rho_3} 
( |\lambda|^{1\slash 2} + |\xi'| )^{s-|\kappa'|},\\
\big| \pa_{\xi'}^{\kappa'} (\tau \pa_\tau)^{\ell} B^s \big|
\leq C_{\kappa',s,\ep,\mu,\rho_1,\rho_2} ( |\lambda|^{1\slash 2} + |\xi'| )^{s-|\kappa'|},
\end{gather*}
for any $(\lambda,\xi')\in \wt \Gamma_{\ep,\lambda_0,\zeta}.$
Moreover, there exists some constant $c'=c'(\ep,\mu,\rho_1,\rho_2)$ such that
\begin{equation*}
\big| \pa_{\xi'}^{\kappa'} (\tau \pa_\tau)^{\ell}e^{-B x_{_N}}\big|
\leq  C_{\kappa',\ep,\mu,\rho_1,\rho_2}  
(|\lambda|^{1\slash 2} + |\xi'| )^{-|\kappa'|} e^{-c'(|\lambda|^{1\slash 2} +|\xi'|)x_{_N}},
\end{equation*}
for any $(\lambda,\xi', x_{_N}) \in \wt\Gamma_{\ep,\lambda_0,\zeta} \times \BBR_+.$ 

\item $L_{11}, L_{22} \in \BM_{1,1}(\wt\Gamma_{\ep,\lambda_0,\zeta}),$ 
$L_{12} \in \BM_{2,1}(\wt\Gamma_{\ep,\lambda_0,\zeta}),$
$L_{21} \in \BM_{0,1}(\wt\Gamma_{\ep,\lambda_0,\zeta})$
and 
$$L^\pm :=(\det \BBL)^{\pm 1} \in \BM_{\pm 2,1}(\wt\Gamma_{\ep,\lambda_0,\zeta}).$$
Moreover, $M(L_{jk}, \Gamma_{\ep,\lambda_0,\zeta} ),$ $j,k=1,2,$
$M(L^\pm,\Gamma_{\ep,\lambda_0,\zeta})$ depend solely on $N,\ep,\mu,\nu,\lambda_0,\zeta_0,\rho_1,\rho_2,\rho_3.$

\item Set $Q(\lambda,\xi'):=(|\xi'|^2 -A^2) \slash (AB-|\xi'|^2)$ and we have  
$Q \in \BM_{0,1}(\wt\Gamma_{\ep,\lambda_0,\zeta})$ 
with
$M(Q, \Gamma_{\ep,\lambda_0,\zeta} )$ depending solely on
$N,\ep,\mu,\nu,\lambda_0,\zeta_0,\rho_1,\rho_2,\rho_3.$
\end{enumerate}
\end{lemm}

Thanks to Lemma \ref{lemma:basic} and Lemma \ref{lemma:ABL} above, it is easy to show that:
\begin{lemm}[Interaction]\label{lemma:int_AB}
Under the same assumptions in Lemma \ref{lemma:ABL}, we have
$$Q':=(AB + |\xi'|^2)^{-1} \in \BM_{-2,1}(\Gamma_{\ep,\lambda_0}),$$ 
with $M(Q', \Gamma_{\ep,\lambda_0,\zeta} )$ depending solely on $N,\ep,\mu,\nu,\lambda_0,\zeta_0,\rho_1,\rho_2,\rho_3.$
\end{lemm}
\begin{proof}
In fact, it is sufficient to verify that there exists a constant $c=c(\ep,\mu,\nu,\lambda_0,\zeta_0,\rho_1,\rho_2,\rho_3)$ 
such that 
\footnote{We omit the details of the reduction here, because the similar arguments will be employed in the proof of Lemma \ref{lemma:N_AB} later.}
\begin{equation*}
\big| AB + |\xi'|^2 \big| \geq c(|\lambda| + |\xi'|^2), 
\,\,\,\forall \,\, (\lambda,\xi') \in \wt\Gamma_{\ep,\lambda_0,\zeta}.
\end{equation*}
The lower bound of $AB+|\xi'|^2$ above immediately follows from the fact 
\begin{equation}\label{claim:AB}
AB \in \Sigma_{\ep_0}, \,\,\,\text{for some}\,\,0<\ep_0 <\pi\slash 2.
\end{equation}
Indeed, Lemma \ref{lemma:basic} and Lemma \ref{lemma:ABL} imply that
\begin{equation*}
\big| AB + |\xi'|^2 \big| \geq \sin(\ep_0 \slash 2) (|AB| + |\xi'|^2)
 \geq  \sin(\ep_0 \slash 2) c^2 (|\lambda|^{1\slash 2} + |\xi'|^2)^2.
\end{equation*}

Now, we prove \eqref{claim:AB}. Thanks to  Lemma \ref{lemma:basic}, there exists $0<\ep'<\pi\slash 2$ such that 
$z_1:=(2\alpha +\beta +\zeta)^{-1} \lambda$ belongs to $\Sigma_{\ep'}.$
Moreover, for any $\lambda\in \Gamma_{\ep,\lambda_0,\zeta},$ we have 
\begin{align*}
|\arg A | &= \big| \arg (z_1 + |\xi'|^2 )| \slash 2 \leq |\arg z_1|\slash 2 \leq (\pi-\ep') \slash 2,\\
|\arg B | &= \big| \arg (\alpha^{-1}\lambda + |\xi'|^2 )| \slash 2 \leq |\arg \lambda| \slash 2 \leq (\pi-\ep) \slash 2.
\end{align*}
Then $|\arg (AB)| = |\arg A+  \arg B| < \pi -\ep_0$ for any $\ep_0 < (\ep+\ep') \slash 2.$
\end{proof}

Next, we pick up several standard results on $\CR$-boundedness  (see \cite{ES2013,GS2014, Shi2019} for instance).
\begin{lemm}\label{lemma:basic_es}
Under the same assumptions in Lemma \ref{lemma:ABL},
we take $n_1(\lambda,\xi') \in \BM_{-2,1} (\wt\Gamma_{\ep,\lambda_0,\zeta}),$ 
and consider the operators
\begin{align*}
\Psi_1(\lambda) f (x)&:=\int_0^{\infty}   \CF^{-1}_{\xi'}\Big[ n_1(\lambda,\xi') 
B e^{-B(x_{_N}+y_{_N})}\CF_{y'}[f](\xi',y_{_N})\Big] (x')  \,d y_{_N},\\
\Psi_2(\lambda) f (x) &:=\int_0^{\infty}   \CF^{-1}_{\xi'}\Big[ n_1(\lambda,\xi') 
B^2 \CM (x_{_N}+y_{_N}) \CF_{y'}[f](\xi',y_{_N})\Big] (x')  \,d y_{_N}.
\end{align*}
Then we have
\begin{equation*}
\CR_{\CL \big(L_q(\BBR_+^{N});H^{2-j}_q(\BBR^N_+)\big)} 
\big( \{(\tau \pa_{\tau})^{\ell}\ \lambda^{j\slash 2}\Psi_k(\lambda): \lambda \in \Gamma_{\ep,\lambda_0,\zeta} \} \big) \leq r_k,
\end{equation*}
 for $j=0,1,2,$ $k=1,2,$ and $\ell =0,1.$ Above the constants $r_1, r_2$ depend on  $M(n_1,\Gamma_{\ep,\lambda_0,\zeta}),$ $N,$ $q,$ $\ep,$ $\mu,$ $\lambda_0,$ $\rho_1,$ $\rho_2.$
\end{lemm}

\begin{lemm}\label{lemma:basic_es_2}
Assume that $n_2(\lambda,\xi')\in \BM_{0,2} \big(\Lambda \times (\BBR^{N-1} \backslash \{0\}) \big)$
for some domain $\Lambda \subset \BBC,$ 
and the cut-off functions $\varphi,$ $\psi \in C_0^{\infty}(]-2,2[).$ 
Consider the operators
\begin{align*}
\Psi_3(\lambda) f &:=\varphi(x_{_N}) \int_0^{\infty} \CF^{-1}_{\xi'}\Big[ n_2(\lambda,\xi') 
e^{-|\xi'|(x_{_N}+y_{_N})} \psi(y_{_N}) \CF_{y'}[f](\xi',y_{_N})\Big] (x')  \,d y_{_N},\\
\Psi_4(\lambda) f &:=\varphi(x_{_N}) \int_0^{\infty} \CF^{-1}_{\xi'}\Big[ n_2(\lambda,\xi') 
|\xi'|e^{-|\xi'|(x_{_N}+y_{_N})}\psi(y_{_N})\CF_{y'}[f](\xi',y_{_N})\Big] (x')  \,d y_{_N}.
\end{align*}
Then we have
\begin{equation*}
\CR_{\CL \big(L_q(\BBR_+^{N})\big)} 
\big( \{(\tau \pa_{\tau})^{\ell} \Psi_k(\lambda): \lambda \in \Lambda \} \big) \leq r_k,
\end{equation*}
 for $k=3,4,$ $\ell =0,1$ and some constants $r_3,r_4$ depending on  $M(n_2,\Lambda),$ $\varphi,\psi,$ $N,q.$
\end{lemm}

Finally, recall $N(A,B)$ in \eqref{eq:GR_half_h_1}, and we shall see  $N(A,B)^{-1}$ is well defined by choosing suitable $\Gamma_{\ep,\lambda_0,\zeta}$ in the next lemma.
\begin{lemm}\label{lemma:N_AB}
Under the same assumptions in Lemma \ref{lemma:ABL},
there exist some $\lambda_0 \geq 1$ and $C_{\kappa'}>0$ such that 
\begin{equation}\label{es:N_AB}
\big| \pa_{\xi'}^{\kappa'}  (\tau \pa_\tau)^\ell \big( N(A,B)^{-1}\big) \big| 
\leq C_{\kappa'} (|\lambda|+|\xi'|)^{-1} ( |\lambda|^{1\slash 2} +|\xi'|)^{-2-|\kappa'|},
\end{equation}
for all $\lambda \in \Gamma_{\ep,\lambda_0,\zeta},$ $\kappa'\in \BBN_0^{N-1}$ and $\ell =0,1.$ 
Here $\lambda_0$ and $C_{\kappa'}$ depend on $\sigma,m,\ep,\mu,\nu,\zeta_0,\rho_1, \rho_2, \rho_3.$
\end{lemm}
\begin{proof}
We will focus on the proof of \eqref{es:N_AB} with $\ell=0,$ from which the case $\ell =1$ will be derived without any difficulty. 
Now let us introduce
\begin{equation*}
P(\lambda, \xi') := \frac{\lambda}{AB-|\xi'|^2}
= \frac{\alpha(2\alpha+\beta+\zeta)}{3\alpha+\beta+\zeta}\frac{AB+|\xi'|^2}{(3\alpha+\beta+\zeta)^{-1}\lambda +|\xi'|^2}
\cdot
\end{equation*}
By the definition of $P,$ the matrix $\BBL$ and $\det \BBL$ are formulated by
\begin{gather}\label{eq:L_ij_2}
L_{11} := A P, \qquad
L_{12}:=  |\xi'|^2 (2\alpha-P), \\ \nonumber
L_{21}:=\Big(\frac{A}{A+B} -\frac{\beta +\zeta}{2\alpha +\beta +\zeta}\frac{B}{A+B}\Big)P,
\quad L_{22}:=BP,  \quad \det \BBL = P D(A,B),  \\ \nonumber
D(A,B):= AB P - |\xi'|^2 (2\alpha -P) \big(\frac{A}{A+B} -\frac{\beta +\zeta}{2\alpha +\beta +\zeta}\frac{B}{A+B}\big).
\end{gather}

\underline{Case: $\ell=0.$}
{\bf Step 1.} In order to show \eqref{es:N_AB} for $\ell =0,$ let us make some reduction.
Firstly, it is easy to see from the definition of $N(A,B)$ that 
\begin{equation*}
N(A,B) = L_{11} E_{\sigma,m}(\lambda,\xi') -\lambda L_{12}L_{21}
\end{equation*}
for $E_{\sigma,m}(\lambda,\xi') : = \lambda L_{22} +\sigma (m+|\xi'|^2).$
According to Lemma \ref{lemma:ABL} and Remark \ref{rmk:M1M2}, there exists a constant 
$C_{\kappa'} = C(\kappa',\ep,\mu,\nu,\zeta_0,\rho_1, \rho_2, \rho_3)>0$ such that
\begin{align*}
\big| \pa_{\xi'}^{\kappa'}  (\tau \pa_\tau)^{\ell'} (\lambda L_{12}L_{21}) \big| 
& \leq C_{\kappa'}  |\lambda| ( |\lambda|^{1\slash 2} +|\xi'|)^{2-|\kappa'|},\\
\big| \pa_{\xi'}^{\kappa'}  (\tau \pa_\tau)^{\ell'} E_{\sigma,m}(\lambda,\xi')  \big| 
& \leq C_{\kappa'} \max\{1,\sigma,\sigma m\} (|\lambda| +|\xi'|) ( |\lambda|^{1\slash 2} +|\xi'|)^{1-|\kappa'|},
\end{align*}
for any $(\lambda,\xi') \in \wt \Gamma_{\ep,1,\zeta},$ $\kappa'\in \BBN_0^{N-1}$ and $\ell'=0,1.$
Then it is easy to see from  Lemma \ref{lemma:ABL}  that
\begin{equation}\label{es:N_AB_upper}
\big| \pa_{\xi'}^{\kappa'}  (\tau \pa_\tau)^{\ell'} N(A,B) \big| 
\leq C_{\kappa'} \max\{1,\sigma,\sigma m\}   (|\lambda| +|\xi'|) ( |\lambda|^{1\slash 2} +|\xi'|)^{2-|\kappa'|},
\end{equation}
for any $(\lambda,\xi') \in \wt \Gamma_{\ep,1,\zeta},$ $\kappa' \in \BBN_0^{N-1}$ and $\ell'=0,1.$ 
\smallbreak 

On the other hand, by the Bell formula, we have
\begin{equation*}
\pa_{\xi'}^{\kappa'} \big( N(A,B)^{-1}\big) = \sum_{j =1}^{\kappa'} (-1)^j j ! \big(N(A,B)\big)^{-(j+1)} 
\sum_{ \substack{\kappa'_1 +\cdots +\kappa'_j =\kappa'\\|\kappa'_k| \geq 1}} C^j_{\kappa'_1,\dots \kappa'_j} 
\big( \pa_{\xi'}^{\kappa'_1}   N(A,B)\big) \cdots  \big( \pa_{\xi'}^{\kappa'_j}   N(A,B)\big).
\end{equation*}
Then \eqref{es:N_AB} with $\ell =0$  holds true so long as
\begin{equation}\label{es:N_AB_low}
|N(A,B)| \geq c (|\lambda| +|\xi'|) (|\lambda|^{1\slash 2} +|\xi'|)^2
\end{equation}
for some constant $c,$ because the derivatives of $N(A,B)$ is bounded by \eqref{es:N_AB_upper}.
Furthermore, thanks to the rules in \eqref{eq:L_ij_2}, $N(A,B)= P \wt{N}(A,B)$ with
\begin{equation*}
 \wt{N}(A,B) := \lambda D(A,B)+ \sigma A (m +|\xi'|^2).
\end{equation*}
By Lemma \ref{lemma:basic} and  Lemma \ref{lemma:int_AB}, 
$|P(\lambda,\xi')| \geq C_{\ep,\mu,\nu,\zeta_0,\rho_1,\rho_2,\rho_3}$ for any 
$(\lambda,\xi') \in \wt \Gamma_{\ep,1,\zeta}.$ 
Therefore, it is sufficient to show that $\wt N(A,B)$ is bounded below by the r.h.s. of \eqref{es:N_AB_low}.
\medskip

{\bf Step 2.} In order to study $\wt N(A,B),$ we need the following technical results.
There exist $\lambda_0' \geq 1$ and $\ep_0 \in ]0,\pi\slash 2[$ such that 
\begin{equation}\label{eq:z_N_AB}
z:= (\alpha+\beta+\zeta)(2\alpha+\beta+\zeta)^{-1}\lambda \in \Sigma_{\ep_0},
\,\,\, \forall \,\, (\lambda,\xi') \in \wt\Gamma_{\ep,\lambda_0',\zeta},
\end{equation}
with the choices of $\lambda_0'$ and $\ep_0$ depending only on $\ep,\mu,\nu,\zeta_0,\rho_1,\rho_2,\rho_3.$
\smallbreak 

For the Case (C1) where $|\zeta|=|\gamma_1^{-1}\gamma_3 \lambda| \leq \min\{ \rho_1^{-1}\rho_3 |\lambda|^{-1},\zeta_0\},$ 
we infer from Lemma \ref{lemma:basic} that
\begin{equation*}
(2\alpha + \beta +\zeta)^{-1} \lambda \in \Sigma_{\ep'},
\,\,\, \forall \,\, \lambda \in  \Gamma_{\ep,1,\zeta} 
\end{equation*}
for some $\ep'=\ep'(\ep,\mu,\nu,\zeta_0,\rho_1,\rho_2,\rho_3)$ in $]0,\pi\slash 2[.$ 
On the other hand, there exists $\ep''\in ]0,\ep'\slash 2[$ such that
$$|\arg (\alpha+\beta+\zeta)| \leq \ep'', \,\,\, \forall \,\, \lambda \in  \Gamma_{\ep,\lambda_0',\zeta},$$
by taking $\lambda_0'$ large enough. Thus $|\arg z| \leq \pi -\ep'-\ep''.$

For the Case (C2) where $|\arg \zeta|\in ]\pi \slash 2, \pi-\ep[,$ we denote $\omega:= (\alpha+\beta+\zeta)(2\alpha+\beta+\zeta)^{-1}.$ Assume that
\begin{equation*}
\Re \zeta \leq -(2\alpha +\beta) \,\,\,
\hbox{or}\,\,\,  -\alpha -\beta \leq  \Re \zeta <0,
\end{equation*}
where we have 
\begin{equation*}
|\arg w| \leq \arctan \frac{\alpha}{(\alpha + \beta)\tan \ep} = \arctan \frac{\mu}{(\mu + \nu)\tan \ep} \cdot
\end{equation*}
Then  $|\arg z| \leq |\arg w| + \pi \slash 2  \leq  \pi-\ep_0$ for some $\ep_0=\ep_0(\mu,\nu,\ep).$ 
For the situation $\Re \zeta \in ]-(2\alpha+\beta), -(\alpha+\beta)[,$ we note that $|\arg \lambda| \leq \pi -\arg \zeta$ and then conclude that
\begin{equation*}
|\arg z| =  |\arg w + \arg \lambda| \leq \pi - \arctan \frac{(\alpha+\beta)\tan \ep}{\alpha}
= \pi - \arctan (\mu^{-1}\nu \tan \ep).
\end{equation*}
At last, \eqref{eq:z_N_AB} is valid for the case (C3) as $|\arg w| <  \pi \slash 4$ and
$|\arg \lambda| < \pi\slash 2.$
\medskip

{\bf Step 3.} By assuming $|\lambda| |\xi'|^{-2} \leq r \ll 1,$ there exists some 
$c_1=c_1(\ep,\mu,\nu,\gamma_1,\gamma_2,\zeta_0,\sigma)>0$ such that 
\begin{equation}\label{es:N_AB_wt}
\big| \wt N(A,B) \big| \geq c_1 (|\lambda| +|\xi'|) (|\lambda|^{1\slash 2} +|\xi'|)^2,
\end{equation}
for any $\lambda\in \Gamma_{\ep, \lambda_0',\zeta}.$
By Lemma \ref{lemma:basic}, it is not hard to see that
\begin{equation*}
A, B=|\xi'|\big(1+O(r)\big)  \,\,\,\hbox{and}\,\,\,
P = \frac{2\alpha(2\alpha +\beta +\zeta)}{3\alpha +\beta +\zeta} +O(r).
\end{equation*}
Then \eqref{eq:z_N_AB} and Lemma \ref{lemma:basic} furnish that
\begin{align}\label{es:N_AB_wt_1}
\big| |\xi'|^{-2}\wt{N}(A,B) \big|
&=\Big|  \lambda \Big(\frac{2\alpha(\alpha+\beta+\zeta)}{2\alpha+\beta+\zeta} +O(r)\Big) 
+ \sigma \big(1+O(r) \big) \big( m|\xi'|^{-1} + |\xi'| \big) \Big|\\
& \geq \Big|\frac{2\alpha(\alpha+\beta+\zeta)}{2\alpha+\beta+\zeta}\lambda 
+\sigma \big( m|\xi'|^{-1} + |\xi'| \big) \Big|
- C_1 r \big(|\lambda|+\sigma m |\xi'|^{-1}  +\sigma |\xi'| \big) \notag \\ \notag
& \geq \sin \big(\frac{\ep_0}{2} \big) \Big( 2\alpha\Big|\frac{\alpha+\beta+\zeta}{2\alpha+\beta+\zeta}\Big| |\lambda| +\sigma (m|\xi'|^{-1} + |\xi'|) \Big) \\ \notag
& \quad - C_1 r \big(|\lambda|+\sigma m |\xi'|^{-1}  +\sigma |\xi'| \big) 
\geq C_2 \big(|\lambda|+\sigma m |\xi'|^{-1}  +\sigma |\xi'| \big),
\end{align}
for some $C_2 =C_2(\ep,\mu,\nu,\zeta_0,\rho_1,\rho_2,\rho_3)$ and any 
$(\lambda,\xi') \in \wt\Gamma_{\ep,\lambda_0',\zeta}.$
Due to smallness of $r$ and $|\lambda| \leq r |\xi'|^2,$ we have 
\begin{equation*}
(|\lambda| +|\xi'|) (|\lambda|^{1\slash 2} +|\xi'|)^2 
\leq 4 (|\xi'|^3  + |\lambda||\xi'|^2) \leq 4 (\min\{1,\sigma\})^{-1} (|\lambda| + \sigma |\xi'|^3).
\end{equation*}
Thus \eqref{es:N_AB_wt} holds with $c_1:= C_2 \min\{1,\sigma\} \slash 4.$
\medskip

{\bf Step 4.} For the fixed $r=r(\ep,\mu,\nu,\zeta_0,\rho_1,\rho_2,\rho_3)<1$ in the last step, 
\eqref{es:N_AB_wt} still holds true for $|\lambda|  > r |\xi'|^2,$ $\lambda \in \Gamma_{\ep,\lambda_0'',\zeta}$ and some $\lambda_0'' \geq 1.$
Indeed, thanks to the boundedness of $P$ and Lemma \ref{lemma:ABL}, there exists a constant $c_2$ such that 
\begin{equation*}
 |D(A,B)| = |P^{-1}\det \BBL|\geq c_2 (|\lambda|^{1\slash2} + |\xi'|)^2,
 \,\,\, \forall \,\, (\lambda,\xi') \in \wt\Gamma_{\ep,1,\zeta}.
\end{equation*}
As $|A| \leq C_3 (|\lambda|^{1\slash 2} +|\xi'|)$ for some constant $C_3=C_3(\ep,\mu,\nu,\rho_2),$ we have
\begin{align*}
\big| \wt N(A,B) \big| 
\geq c_2 |\lambda| ( |\lambda|^{1\slash 2} +|\xi'|)^2 -\sigma C_3 ( |\lambda|^{1\slash 2} +|\xi'|)(m + |\xi'|^2).
\end{align*}
Next,  set  $C_\sigma :=  \sigma (1 + 2C_3 \slash {c_2})$ and we claim that there exists $\lambda_0'' \geq 1$ such that
 \begin{equation}\label{eq:f_N_AB}
 f(\lambda,|\xi'|):=|\lambda| ( |\lambda|^{1\slash 2} +|\xi'|) - C_\sigma (m+ |\xi'|^2) \geq 0,
  \,\,\, \forall \,\, \lambda \in \wt\Gamma_{\ep,\lambda_0'',\zeta}.
 \end{equation}
Then \eqref{eq:f_N_AB} implies that 
\begin{align*}
|\wt N(A,B)| & \geq \frac{c_2}{2}  \big(  |\lambda| (|\lambda|^{1\slash 2} +|\xi'|)^2 
+ \sigma (m+|\xi'|^2) (|\lambda|^{1\slash 2} +|\xi'|) \big) \\
 & \geq \frac{c_2}{4} \min\{1,\sigma\} (|\lambda| +|\xi'|) (|\lambda|^{1\slash 2} +|\xi'|)^2,
  \,\,\, \forall \,\, \lambda \in \wt\Gamma_{\ep,\lambda_0'',\zeta}.
\end{align*}

For \eqref{eq:f_N_AB}, $f(\lambda,|\xi'|) \geq (\lambda_0'')^{3\slash 2} -C_\sigma m  + \big((r\lambda_0'')^{1\slash 2} -C_\sigma \big)  |\xi'|^2$ since $|\lambda| \geq \max\{r|\xi'|^2, \lambda_0'' \}.$
Thus \eqref{eq:f_N_AB} holds true by choosing $\lambda_0''$ large. 
Finally, we obtain the lower bound of $\wt N(A,B)$ for $\ell=0,$  by taking $\lambda_0 := \max\{\lambda_0',\lambda_0''\}.$
\medskip

\underline{Case: $\ell=1.$} In fact, the conclusion for $\ell =1$ is a straightforward result from above discussion.
Note that
$(\tau \pa_\tau) \big( N(A,B)^{-1} \big) = - N(A,B)^{-2}  \big(\tau \pa_\tau N(A,B)\big).$
Then our results for the case $\ell=0,$ Leibniz formula and \eqref{es:N_AB_upper} yield the desired bound of $\pa_{\xi'}^{\kappa'}(\tau \pa_\tau) \big( N(A,B)^{-1} \big) $ in \eqref{es:N_AB} for any $\kappa' \in \BBN_0^{N-1}.$
\end{proof}

Thanks to Lemma \ref{lemma:N_AB} above, we end up with the following result. 
\begin{coro}\label{coro:n_Jk}
Under the same assumptions in Lemma \ref{lemma:ABL}, there exists $\lambda_0$ depending solely on $\sigma,$ $m,$ $\ep,$ $\mu,$
$\nu,$ $\zeta_0,$ $\rho_1,$ $\rho_2,$ $\rho_3,$ such that
the symbols $n_{Jk}(\lambda,\xi')$ defined in \eqref{eq:n_Jk_half} belong to 
$\BM_{-2,1}(\wt\Gamma_{\ep,\lambda_0,\zeta})$ for all $J=1,\dots,N$ and $k=1,2.$  
\end{coro}
\begin{proof}
The study of symbols are nothing but applying directly Lemma \ref{lemma:ABL} and Remark \ref{rmk:M1M2}, and we omit the details here.
\end{proof}

\subsection{Proof of Proposition \ref{prop:GR_half_2}}
In this part, we are studying the operators $\CW(\lambda,\BBR^N_+)$ and $\CH(\lambda,\BBR^N_+)$ respectively according to different technical results in last subsection.  
For $Z \in \{B,|\xi'|\},$ let us recall the following equalities, namely the \emph{Volevich trick}, 
 \begin{align*}
e^{-Zx_{_N}} \wh{k}(\xi',0) 
=& \int_0^{\infty}Z e^{-Z(x_{_N}+y_{_N})}\wh{k}(\xi',y_{_N})  \,d y_{_N}
  - \int_0^{\infty}e^{-Z(x_{_N}+y_{_N})}\wh{\pa_{N} k}(\xi',y_{_N}) \,d y_{_N},\\
M(x_{_N})\wh{k}(\xi',0) 
= &\int_0^{\infty}e^{-B(x_{_N}+y_{_N})}\wh{k}(\xi',y_{_N})  \,d y_{_N}
+\int_0^{\infty}AM(x_{_N}+y_{_N})\wh{k}(\xi',y_{_N})  \,d y_{_N}\\
 & - \int_0^{\infty}M(x_{_N}+y_{_N}) \wh{\pa_{N} k}(\xi',y_{_N}) \,d y_{_N},
\end{align*}
where we have used $\pa_{_N} M(z_{_N})=-e^{-Bz_{_N}}-AM(z_{_N})$ for any $z_{_N}>0.$

\subsubsection{The $\CR-$boundedness of $\CW(\lambda,\BBR^N_+)$}
According to the identity $|\xi'|^2 =-\sum_{\ell=1}^{N-1} (i\xi_{\ell})^2,$
the solution operator
\begin{equation*}
\Bu= \CW (\lambda,\BBR^N_+)k =\big( \CW_{1} (\lambda,\BBR^N_+)k, \dots, \CW_{N} (\lambda,\BBR^N_+) k\big)^{\top},
\end{equation*}
in \eqref{eq:GR_half_uj_2}  is rewritten by
\begin{equation*}
\CW_{J} (\lambda,\BBR^N_+) k :=  \CW_{J1}(\lambda,\BBR^N_+)\big((m-\Delta') k\big) 
+\CW_{J2}(\lambda,\BBR^N_+)(\pa_{_N} k) 
+\CW_{J3}(\lambda,\BBR^N_+)(\nabla'\pa_{_N} k) 
\end{equation*} 
for any $J=1,\dots,N.$
The operators $\CW_{Jk} \equiv \CW_{Jk}(\lambda,\BBR^N_+)$ ($k=1,2,3$) above are given by 
\begin{align*}
\CW_{j1}(F_1) :=&\int_0^{\infty} \CF_{\xi'}^{-1} \Big[ n_{j1}(\lambda,\xi')\frac{A}{B} B^2 \CM(x_{_N} + y_{_N})  \wh{F_1}(\xi',y_{_N}) \Big](x') \,d y_{_N}\\
&+\int_0^{\infty} \CF_{\xi'}^{-1} \Big[ n_{j2}(\lambda,\xi')
 B e^{-B(x_{_N}+y_{_N})}  \wh{F_1}(\xi',y_{_N}) \Big](x') \,d y_{_N},\\
 \CW_{j2}(F_2) := & - \int_0^{\infty} \CF_{\xi'}^{-1} \Big[ n_{j1}(\lambda,\xi')\frac{m}{B} 
\big( B^2 \CM(x_{_N} + y_{_N}) - B e^{-B(x_{_N}+y_{_N})} \big) 
 \wh{F_2}(\xi',y_{_N}) \Big](x') \,d y_{_N}\\
&-\int_0^{\infty} \CF_{\xi'}^{-1} \Big[ n_{j2}(\lambda,\xi') \frac{m}{B}  
Be^{-B(x_{_N}+y_{_N})}  \wh{F_2}(\xi',y_{_N}) \Big](x') \,d y_{_N},\\
\CW_{j3}(\BF') :=&\sum_{\ell=1}^{N-1}\int_0^{\infty} \CF_{\xi'}^{-1} \Big[ 
n_{j1}(\lambda,\xi')\frac{i\xi_{\ell}}{B} \big( B^2 \CM(x_{_N} + y_{_N}) - B e^{-B(x_{_N}+y_{_N})} \big) 
 \wh{F_{\ell}}(\xi',y_{_N}) \Big](x') \,d y_{_N}\\
&+\sum_{\ell=1}^{N-1}\int_0^{\infty} \CF_{\xi'}^{-1} \Big[n_{j2}(\lambda,\xi')\frac{i\xi_{\ell}}{B} 
 Be^{-B(x_{_N}+y_{_N})} \wh{F_{\ell}}(\xi',y_{_N}) \Big](x') \,d y_{_N}, 
\end{align*}
for any $j =1, \dots,N-1,$ and
\begin{align*}
\CW_{N1}(F_1) :=&\int_0^{\infty} \CF_{\xi'}^{-1} \Big[ (n_{N1}+n_{N2})(\lambda,\xi') 
Be^{-B(x_{_N}+y_{_N})}  \wh{F_1}(\xi',y_{_N}) \Big](x') \,d y_{_N}\\
&+\int_0^{\infty} \CF_{\xi'}^{-1} \Big[ n_{N1}(\lambda,\xi')\frac{A}{B} B^2 \CM(x_{_N} + y_{_N})  \wh{F_1}(\xi',y_{_N}) \Big](x') \,d y_{_N},\\
\CW_{N2}(F_2) := & - \int_0^{\infty} \CF_{\xi'}^{-1} \Big[ n_{N1}(\lambda,\xi')\frac{m}{B} 
B^2 \CM(x_{_N} + y_{_N})  \wh{F_2}(\xi',y_{_N}) \Big](x') \,d y_{_N}\\
&-\int_0^{\infty} \CF_{\xi'}^{-1} \Big[ n_{N2}(\lambda,\xi') \frac{m}{B} 
Be^{-B(x_{_N}+y_{_N})}  \wh{F_2}(\xi',y_{_N}) \Big](x') \,d y_{_N} ,\\
 \CW_{N3}(\BF') :=&\sum_{\ell=1}^{N-1}\int_0^{\infty} \CF_{\xi'}^{-1} \Big[ 
n_{N1}(\lambda,\xi')\frac{i\xi_{\ell}}{B}  B^2 \CM(x_{_N} + y_{_N}) 
 \wh{F_{\ell}}(\xi',y_{_N}) \Big](x') \,d y_{_N}\\
&+\sum_{\ell=1}^{N-1}\int_0^{\infty} \CF_{\xi'}^{-1} \Big[n_{N2}(\lambda,\xi')\frac{i\xi_{\ell}}{B} 
 Be^{-B(x_{_N}+y_{_N})} \wh{F_{\ell}}(\xi',y_{_N}) \Big](x') \,d y_{_N}.
\end{align*}
Note that $\BM_{-3,1}(\wt\Gamma_{\ep,\lambda_0,\zeta}) \subset \BM_{-2,1}(\wt\Gamma_{\ep,\lambda_0,\zeta})$ and  
$A\slash B,$ $i\xi_{\ell}\slash B$ for $\ell =1,\dots N-1,$ are in $\BM_{0,1}(\wt\Gamma_{\ep,\lambda_0,\zeta}).$ 
Then Lemma \ref{lemma:basic_es} and Corollary \ref{coro:n_Jk} imply that
\begin{gather*}
\CR_{\CL\big(L_q(\BBR^N_+); H^{2-j}_q(\BBR^N_+) \big)}
 \Big( \Big\{ (\tau \pa_{\tau})^{\ell}\big( \lambda^{j\slash 2}\CW_{Jk}(\lambda,\BBR^N_+)\big) : \lambda \in 
 \Gamma_{\ep,\lambda_0,\zeta} \Big\}\Big) \leq C,\\
\CR_{\CL\big(L_q(\BBR^N_+)^{N-1}; H^{2-j}_q(\BBR^N_+) \big)}
 \Big( \Big\{ (\tau \pa_{\tau})^{\ell}\big( \lambda^{j\slash 2}\CW_{J3}(\lambda,\BBR^N_+)\big) : \lambda \in 
\Gamma_{\ep,\lambda_0,\zeta} \Big\}\Big) \leq C,
\end{gather*}
for $k=1,2,$ and $j=0,1,2.$ Thus Definition \ref{def:R-bounded} gives us 
\begin{equation*}
\CR_{\CL\big(H_q^2(\BBR^N_+); H^{2-j}_q(\BBR^N_+)^N \big)}
 \Big( \Big\{ (\tau \pa_{\tau})^{\ell}\big( \lambda^{j\slash 2}\CW(\lambda,\BBR^N_+)\big) : \lambda \in 
\Gamma_{\ep,\lambda_0,\zeta} \Big\}\Big) \leq C.
\end{equation*}

\subsubsection{The $\CR-$boundedness of $\CH(\lambda,\BBR^N_+)$}
Let us apply the Volevich trick to \eqref{eq:GR_half_h_2} and obtain 
\begin{equation*}
h(x) =\big(\CH(\lambda,\BBR^N_+)k\big)(x)
=\big(\CH_{1}(\lambda,\BBR^N_+)k\big)(x) + \big(\CH_{2}(\lambda,\BBR^N_+) k\big)(x),
\end{equation*}
where the operators $\CH_{\Fa} \equiv\CH_{\Fa}(\lambda,\BBR^N_+)$ with $\Fa=1,2,$ are given by 
\begin{align*}
(\CH_{1}k)(x)&:=  \varphi(x_{_N}) \int_{0}^{\infty} \CF^{-1}_{\xi'}\Big[ 
\frac{\det \BBL}{N(A,B)} \,|\xi'| e^{-|\xi'|(x_{_N}+y_{_N})}
\varphi(y_{_N})  \widehat{k}(\xi',y_{_N})\Big] (x') \,dy_{_N},\\
(\CH_{2}k)(x)&:=-\varphi(x_{_N}) \int_{0}^{\infty} \CF^{-1}_{\xi'}\Big[ 
\frac{\det \BBL}{N(A,B)} \,e^{-|\xi'|(x_{_N}+y_{_N})}
\pa_{{_N}} \big( \varphi(y_{_N})  \widehat{k}(\xi',y_{_N})\big) \Big] (x') \,dy_{_N}.
\end{align*}
Clearly, we have for $j=0,1,$ 
\begin{align*}
(\lambda^j \CH_{1}k)(x)&:=  \varphi(x_{_N}) \int_{0}^{\infty} \CF^{-1}_{\xi'}\Big[ 
\frac{\lambda^j\det \BBL}{N(A,B)} \,|\xi'| e^{-|\xi'|(x_{_N}+y_{_N})}
\varphi(y_{_N})  \widehat{k}(\xi',y_{_N})\Big] (x') \,dy_{_N},\\
(\lambda^j\CH_{2}k)(x)&:=-\varphi(x_{_N}) \int_{0}^{\infty} \CF^{-1}_{\xi'}\Big[ 
\frac{\lambda^j \det \BBL}{N(A,B)} \,e^{-|\xi'|(x_{_N}+y_{_N})}
\pa_{{_N}} \big( \varphi(y_{_N})  \widehat{k}(\xi',y_{_N})\big) \Big] (x') \,dy_{_N}.
\end{align*}
\medskip

On the other hand, assume that $1\leq |\alpha'|+n\leq3-j$ with $j=0,1.$ By the identity 
\begin{equation*}
1 = \frac{1+|\xi'|^2}{1+|\xi'|^2} = \frac{1}{1+|\xi'|^2}  - \sum_{\ell=1}^{N-1} \frac{i\xi_{\ell}}{1+|\xi'|^2} i\xi_{\ell},
\end{equation*}
we have
\begin{align*}
(\lambda^j \pa_{x'}^{\alpha'}\pa_{_N}^{n} \CH_1 k)(x) 
=& \sum_{n_1+n_2=n}  \binom{n}{n_1}(\pa_{_N}^{n_1} \varphi)(x_{_N})
\int_0^{\infty} \CF_{\xi'}^{-1} \Big[\frac{\det \BBL}{N(A,B)} \frac{\lambda^{j}(i\xi')^{\alpha'}(-|\xi'|)^{n_2}}{1+|\xi'|^2}
 \\
&\hspace*{2cm} \times\,|\xi'| e^{-|\xi'|(x_{_N}+y_{_N})}\varphi(y_{_N})  \rwh{(1-\Delta')k}(\xi',y_{_N})
 \Big](x')\,dy_{_N},\\
 (\lambda^j \pa_{x'}^{\alpha'}\pa_{_N}^{n} \CH_2 k)(x) 
=& -\sum_{n_1+n_2=n}  \binom{n}{n_1}(\pa_{_N}^{n_1} \varphi)(x_{_N})
\int_0^{\infty} \CF_{\xi'}^{-1} \Big[\frac{\det \BBL}{N(A,B)} \frac{\lambda^{j}(i\xi')^{\alpha'}(-|\xi'|)^{n_2}}{1+|\xi'|^2}
 \\
&\hspace*{2cm} \times\, e^{-|\xi'|(x_{_N}+y_{_N})}
\pa_{{_N}} \big( \varphi(y_{_N})  \widehat{k}(\xi',y_{_N})\big) 
 \Big](x')\,dy_{_N}\\
 & +\sum_{n_1+n_2=n}  \sum_{\ell=1}^{N-1}\binom{n}{n_1}(\pa_{_N}^{n_1} \varphi)(x_{_N})
\int_0^{\infty} \CF_{\xi'}^{-1} \Big[\frac{\det \BBL}{N(A,B)} 
\frac{\lambda^{j}(i\xi')^{\alpha'}(-|\xi'|)^{n_2}(i\xi_{\ell})}{(1+|\xi'|^2)|\xi'|} \\
&\hspace*{2cm} \times\, |\xi'|e^{-|\xi'|(x_{_N}+y_{_N})}
\pa_{{_N}} \big( \varphi(y_{_N})  \widehat{\pa_{\ell}k}(\xi',y_{_N})\big) 
 \Big](x')\,dy_{_N}.
\end{align*}

Next, Lemma \ref{lemma:ABL} and Lemma \ref{lemma:N_AB} imply that
\begin{equation*}
\Big| \pa_{\xi'}^{\kappa'}  (\tau \pa_\tau)^\ell \Big(\frac{\det \BBL}{N(A,B)} \Big) \Big|
 \leq \frac{C_{\kappa'}}{|\lambda| + |\xi'|}(|\lambda|^{1\slash 2} + |\xi'|)^{-|\kappa'|}, 
 \,\,\,\forall\,\, \kappa' \in \BBN_0^{N-1} \,\,\,\hbox{and}\,\,\,\ell =0,1.
\end{equation*}
Then it is not hard to see that  the symbols
\begin{equation*}
\frac{\lambda^j\det \BBL}{N(A,B)},\,\,
\frac{\det \BBL}{N(A,B)} \frac{\lambda^{j}(i\xi')^{\alpha'}|\xi'|^{n_2}}{1+|\xi'|^2}, \,\,
\frac{\det \BBL}{N(A,B)} \frac{\lambda^{j}(i\xi')^{\alpha'}|\xi'|^{n_2}(i\xi_{\ell})}{(1+|\xi'|^2)|\xi'|} 
\end{equation*}
are of the class $\BM_{0,2}(\wt\Gamma_{\ep,\lambda_0,\zeta})$ for $j+|\alpha'|+n_1+n_2 \leq 3$ and $j=0,1.$ 
Thus we infer from Lemma \ref{lemma:basic_es_2} that 
\begin{gather} \label{eq:H_half_Rbdd}
\CR_{\CL \big(L_q(\BBR_+^{N})\big)} 
\big( \{ \lambda^{j}\CH_1(\lambda,\BBR^N_+): \lambda \in \Gamma_{\ep,\lambda_0,\zeta} \} \big) \leq C,\\ \notag
\CR_{\CL \big(H^{1}_q(\BBR^N_+);L_q(\BBR_+^{N})\big)} 
\big( \{ \lambda^{j}\CH_2(\lambda,\BBR^N_+): \lambda \in \Gamma_{\ep,\lambda_0,\zeta} \} \big) \leq C,\\ \nonumber
\CR_{\CL \big(H^{2}_q(\BBR^N_+); L_q(\BBR_+^{N})\big)} 
\big( \{ \lambda^{j}\pa_{x'}^{\alpha'}\pa_{_N}^n\CH_{\Fa}(\lambda,\BBR^N_+): \lambda \in \Gamma_{\ep,\lambda_0,\zeta}  \} \big) \leq C, 
\end{gather}
with  $1\leq |\alpha'|+n\leq3-j,$ $j=0,1,$ and $\Fa=1,2.$ 
Then $\CH(\lambda,\BBR^N_+)$ has the desired property due to \eqref{eq:H_half_Rbdd} and the definition of  $\CR-$boundedness. This completes the proof of Proposition \ref{prop:GR_half_2}.

\section{Generalized model problem in the bent half space} 
\label{sec:bh}
Let $\BPhi$ be a $C^1$ diffeomorphism from $\BBR^N_{\xi}$ onto $\BBR^N_x$ and  $\BPhi^{-1}$ be the inverse of $\BPhi.$ 
Assume that 
\begin{gather*}
\nabla_{\xi} \BPhi^\top (\xi) := \BBA + \BBB(\xi) = \begin{bmatrix}
a_{ij}
\end{bmatrix}_{N\times N} 
+ \begin{bmatrix}
b_{ij}(\xi)
\end{bmatrix}_{N\times N}, \\
\sA_{\Phi} := \nabla_x (\BPhi^{-1})^\top|_{x=\BPhi(\xi)} = \BBA_{-} + \BBB_-(\xi) 
= \begin{bmatrix}
\bar{a}_{ij}
\end{bmatrix}_{N\times N} 
+ \begin{bmatrix}
\bar{b}_{ij}(\xi)
\end{bmatrix}_{N\times N},
\end{gather*}
for some constant orthogonal matrices $\BBA$ and $\BBA_-.$
Moreover, denote $\Omega_+ := \BPhi(\BBR^N_+)$ and $\Gamma_+ := \pa \Omega_+ = \BPhi (\BBR^N_0).$ 
Then $\Gamma_+$ is characterized by the equation 
$(\Phi^{-1})_N (x) = 0,$ 
and the unit outer normal $\Bn_{+}$ to $\Gamma_+$ is given by 
\begin{equation*}
\Bn_{+} \big(\BPhi(\xi)\big) 
=\frac{\sA_{\Phi}\Bn_0}{|\sA_{\Phi}\Bn_0|}
= -\frac{(\nabla_x \Phi_{N}^{-1})\big(\BPhi(\xi)\big) }{\big| (\nabla_x \Phi_{N}^{-1})\big(\BPhi(\xi)\big) \big|}
= -\frac{ \big( \bar{a}_{1N}+\bar{b}_{1N}(\xi),\dots, \bar{a}_{NN}+\bar{b}_{NN}(\xi)\big)^{\top} }{\Big( \sum_{j=1}^N \big( \bar{a}_{Nj}+\bar{b}_{Nj}(\xi) \big)^2 \Big)^{1\slash 2} },
\end{equation*}
with $\Bn_0 := (0,\dots,0,-1)^{\top}.$
\smallbreak

Now we assume that $\BPhi \in H^3_r(\BBR^N)^N$ for some $N<r<\infty,$  
and there exist constants $0<M_1<1\leq M_2 \leq M_3<\infty$ such that 
\begin{equation}\label{hyp:M1M2_BH}
\|(\BBB, \BBB_-)\|_{L_{\infty}(\BBR^N)} \leq M_1, \,\,\,
\|\nabla(\BBB, \BBB_-)\|_{L_{r}(\BBR^N)} \leq M_2, \,\,\,
\|\nabla^2(\BBB, \BBB_-)\|_{L_{r}(\BBR^N)} \leq M_3.
\end{equation}
For such $\Gamma_+$ characterized by $H^3_r(\BBR^N)$ mapping, we consider the following model problem,
\begin{equation}\label{eq:RR_CNS_BH}
\left\{\begin{aligned}
& \lambda  \Bv -\gamma_{1}^{-1}\Di \big( \BBS(\Bv) + \zeta \gamma_3 \di \Bv \BBI \big)
= \Bf &&\quad\hbox{in}\quad  \Omega_+,\\
&\BBS(\Bv)  \Bn_{+} +\zeta \gamma_3 \di \Bv \,\Bn_{+}
+\sigma (m-\Delta_{\Gamma_+})h \,\Bn_{+} = \Bg &&\quad\hbox{on}\quad  \Gamma_+, \\
&\lambda  h - \Bv \cdot \Bn_{+}  = k &&\quad\hbox{on}\quad  \Gamma_+.\\
	\end{aligned}\right.
\end{equation}
In \eqref{eq:RR_CNS_BH}, $\gamma_1$ and $\gamma_3$ are uniformly continuous functions on $\overline{\Omega}_+,$ and there exist some constants $\rwh \gamma_1,\rwh\gamma_3$ such that
\begin{gather}\label{hyp:gamma_RR_BH}
0<\rho_1  \leq \gamma_1(x), \rwh{\gamma_1} \leq \rho_2,\quad 
0 < \gamma_3(x), \rwh{\gamma_3} \leq \rho_3, \,\,\, \forall \,\, x \in \overline{\Omega}_+,
\\ \notag
\sum_{\Fa=1,3}\|\gamma_{\Fa} -\rwh{\gamma_{\Fa}}\|_{L_{\infty}(\Omega_+)}  \leq M_1<1,\quad 
\sum_{\Fa=1,3} \|\nabla \gamma_{\Fa}\|_{L_r(\Omega_+)}  \leq CM_2,
\end{gather}
for some constants $\rho_1, \rho_2, \rho_3>0.$
The main result in this section for \eqref{eq:RR_CNS_BH} reads:
\begin{theo}\label{thm:GR_BH}
Let $0<\ep<\pi\slash 2,$ $\sigma,m, \mu, \nu,\zeta_0>0,$ $1<q<\infty,$ $N<r<\infty$ and $r\geq q.$ 
Assume that \eqref{hyp:gamma_RR_BH} is satisfied. 
For $\Omega_+$ given above, we set that 
\begin{equation*}
Y_q(\Omega_+) := L_q(\Omega_+)^N \times H^{1}_q(\Omega_+)^N\times H^{2}_q(\Omega_+), \quad 
\CY_q(\Omega_+) := L_q(\Omega_+)^N \times Y_q(\Omega_+).
\end{equation*}
Then for any $(\Bf,\Bg,k) \in Y_q(\Omega_+),$ there exist constants $\lambda_0,r_b \geq 1$ and operator families
\begin{align*}
\CA_0(\lambda,\Omega_+) & \in 
\Hol\Big( \Gamma_{\ep,\lambda_0,\zeta} ; \CL\big(\CY_q(\Omega_+);H^2_q(\Omega_+)^N \big) \Big),\\
\CH_0(\lambda,\Omega_+) & \in 
 \Hol\Big( \Gamma_{\ep,\lambda_0,\zeta} ; \CL\big(\CY_q(\Omega_+);H^3_q(\Omega_+) \big) \Big),
\end{align*}
such that $(\Bv,h):=\big(\CA_0(\lambda,\Omega_+), \CH_0(\lambda,\Omega_+) \big)(\Bf,\lambda^{1\slash 2}\Bg,\Bg,k)$
is a solution of \eqref{eq:RR_CNS_BH}. Moreover, we have 
\begin{gather*}
\CR_{\CL\big(\CY_q(\Omega_+); H^{2-j}_q(\Omega_+)^N \big)}
 \Big( \Big\{ (\tau \pa_{\tau})^{\ell}\big( \lambda^{j\slash 2}\CA_0(\lambda,\Omega_+)\big) : \lambda \in 
\Gamma_{\ep,\lambda_0,\zeta} \Big\}\Big) \leq r_b,\\
 \CR_{\CL\big(\CY_q(\Omega_+); H^{3-j'}_q(\Omega_+) \big)}
 \Big( \Big\{ (\tau \pa_{\tau})^{\ell}\big( \lambda^{j'}\CH_0(\lambda,\Omega_+)\big) : \lambda \in 
\Gamma_{\ep,\lambda_0,\zeta} \Big\}\Big) \leq r_b,
\end{gather*}
for $\ell, j'=0,1,$ $j=0,1,2,$ and $\tau := \Im \lambda.$
Above the constants $\lambda_0$ and $r_b$ depend solely on
$\ep,$ $\sigma,$ $m,$ $\mu,$ $\nu,$ $q,$ $r,$ $N,$ $\zeta_0,$ $\rho_1,$ $\rho_2,$ $\rho_3.$
\end{theo}

\subsection{Reduction of the \eqref{eq:RR_CNS_BH}}
Introduce $\Bw (\xi):=\BBA^{\top}_- \Bv \big(\BPhi(\xi)\big)$
and $H(\xi):= h\big(\BPhi(\xi)\big).$ 
Let us derive the equations of $\Bw$ and $H$ according to \eqref{eq:RR_CNS_BH}. 
Firstly, thanks to the facts that
\begin{equation*}
\nabla_x \Bv^{\top} \big(\BPhi(\xi)\big) =\sAP \nabla_\xi \Bw^\top\BBA_-^{\top},\quad 
(\di_x \Bv) \big(\BPhi(\xi)\big) =\sAP: (\nabla_\xi \Bw^\top\BBA_-^{\top}),
\end{equation*}
we obtain that
\begin{align}
\BBF_0(\Bw) &:= \big( \BBS(\Bv) + \zeta \gamma_3 \di \Bv \BBI \big)  \big(\BPhi(\xi)\big)  \notag \\ \notag
&=\mu (\sAP \nabla_{\xi} \Bw^{\top}\BBA^{\top}_- + \BBA_- \nabla^{\top}_{\xi} \Bw\sAP^{\top}) 
 +\big( \nu-\mu + \zeta (\gamma_3 \circ \BPhi) \big) \big( (\BBA_-^{\top} \sAP) : \nabla_{\xi} \Bw^{\top} \big) \BBI,\\
 \BBF_0(\Bw)\sAP&:=\BBA_- \big( \BBS(\Bw) +    \zeta \wh\gamma_3 \di_{\xi} \Bw \, \BBI \big) +  \BBF(\Bw),\label{eq:RR_tensor_1}
\end{align}
where $\BBF(\Bw):=\BBF_1(\Bw)+\BBF_2(\Bw)$ and
\begin{align*}
\BBF_1(\Bw) &:=  \BBA_- \BBS(\Bw) \BBA_-^{\top} \BBB_- 
+ \mu \big( \BBB_-\nabla_{\xi} \Bw^{\top}\BBA_-^{\top}+\BBA_-\nabla_{\xi}^{\top} \Bw\, \BBB_-^{\top} \big) \sAP \\
& \quad + (\nu-\mu) \big( (\BBA_-^{\top} \BBB_-) : \nabla_{\xi} \Bw^{\top} \big) \sAP,\\
\BBF_2(\Bw)&:= \zeta \wh\gamma_3  (\di_{\xi}\Bw) \BBB_-
+ \zeta \wh\gamma_3  \big( (\BBA_-^{\top} \BBB_-) : \nabla_{\xi} \Bw^{\top} \big)  \sAP \\
&\quad + \zeta  (\gamma_3\circ \BPhi -\wh\gamma_3 )\big( (\BBA_-^{\top} \sAP) : \nabla_{\xi} \Bw^{\top} \big)  \sAP.
\end{align*}

On the other hand,  $\Gamma_+:= \Phi(\BBR^N_0)$ and 
$\nabla_{\xi} \BPhi^\top (\xi',0) = \begin{bmatrix}
a_{ij}
\end{bmatrix}_{N\times N} 
+ \begin{bmatrix}
b_{ij}(\xi',0)
\end{bmatrix}_{N\times N}$ by previous convention.
 For simplicity, we write for $i,j=1,\dots,N-1,$
\begin{equation*}
\Bphi(\xi'):= \BPhi(\xi',0),\quad 
\pa_{i}:= \pa_{\xi_i}, \quad
\Btau_i:=\pa_{i} \Bphi \quad
\hbox{and}\quad 
\Btau_{ij}:=\pa_{j}\pa_{i} \Bphi.
\end{equation*}
Then the first fundamental form of $\Gamma_+$ and its inverse are given by
 \begin{equation*}
 \BBG_+ = \begin{bmatrix} g_{ij}
\end{bmatrix}_{(N-1)\times (N-1)}
:=\begin{bmatrix}
\Btau_i \cdot \Btau_j
\end{bmatrix}_{(N-1)\times (N-1)}, \quad 
\Fg_+ :=\det \BBG_+, \quad 
\BBG_+^{-1} :=  \begin{bmatrix} g^{ij}
\end{bmatrix}_{(N-1)\times (N-1)}.
 \end{equation*}
Introduce that 
 \begin{equation*}
\wt g_{ij} := g_{ij} - \delta_{ij} 
=\sum_{k=1}^{N} (b_{ik}a_{jk}+a_{ik}b_{jk} + b_{ik} b_{jk}), \quad
\wt g^{ij} := g^{ij} -\delta_{ij} .
 \end{equation*}
 Then \eqref{hyp:M1M2_BH} implies that 
 \begin{gather}\label{es:geo_1_BH}
\|(\wt g_{ij}, \wt g^{ij})\|_{L_{\infty}(\BBR^{N-1})} +\|\Fg_+   -  1 \|_{L_{\infty}(\BBR^{N-1})} \leq C M_1, \quad
\|\nabla_{\xi'} (\wt g_{ij}, \wt g^{ij})\|_{L_{r}(\BBR^{N-1})} \leq C M_2,\\ \notag
 \|\nabla_{\xi'}^2 \wt g_{ij}\|_{L_{r}(\BBR^{N-1})} \leq C M_2 M_3, \quad 
  \|\nabla_{\xi'}^2 \wt g^{ij}\|_{L_{r}(\BBR^{N-1})} \leq C M_2^3 M_3.
 \end{gather}
 \smallbreak 
 
Next, recall the definition of the \emph{Laplace-Beltrami} operator on $\Gamma_+,$
\begin{align*}
 \big(\Delta_{\Gamma_+} h\big)  \big( \Bphi(\xi') \big)
 &:= \frac{1}{ \sqrt{\Fg_+}} \pa_i (\sqrt{\Fg_+} g^{ij} \pa_j H)
 = g^{ij} \pa_{i}\pa_{j} H -  g^{ij}\Lambda^k_{ij} \pa_k H\\
 &= \Delta'_{\xi'} H + \CG(H),
\end{align*}
where the \emph{Christoffel symbols} $\Lambda^k_{ij}:=g^{kr} \Btau_{ij} \cdot \Btau_r,$ for $i,j,k,r=1,\dots,N-1,$ and 
\begin{align*}
\CG(H) := \wt g^{ij} \pa_{i}\pa_{j} H -  g^{ij}\Lambda^k_{ij} \pa_k H.
\end{align*}
Furthermore, \eqref{hyp:M1M2_BH} and \eqref{es:geo_1_BH} yield that
\begin{equation}\label{es:geo_2_BH}
\|\Lambda_{ij}^k\|_{L_\infty(\BBR^{N-1})} \leq CM_3, \quad 
\|\nabla_{\xi'}\Lambda_{ij}^k\|_{L_r(\BBR^{N-1})} \leq CM_2M_3.
\end{equation}

Now it is not hard to see that \eqref{eq:RR_CNS_BH} turns to 
\begin{equation}\label{eq:RR_CNS_BH_2}
\left\{\begin{aligned}
& \lambda  \Bw -\wh\gamma_{1}^{-1} \Di_{\xi}
\big(  \BBS(\Bw) + \zeta  \wh\gamma_3 \di_{\xi} \Bw\, \BBI_{N} \big)  +\CF_1(\Bw)
= \BF_+ &&\quad\hbox{in}\quad  \BBR^N_+,\\
&\big( \BBS(\Bw)+\zeta \wh\gamma_3 \di_{\xi} \Bw \BBI_{N} \big)\Bn_{0}
+\sigma (m-\Delta')H \,\Bn_{0} +\CF_2(\Bw,H) \Bn_0 = \BG_+ &&\quad\hbox{on}\quad  \BBR^N_0, \\
&\lambda  H - \Bw \cdot \Bn_{0} + \CF_3(\Bw) \cdot \Bn_0  = K_+  &&\quad\hbox{on}\quad  \BBR^N_0,\\
	\end{aligned}\right.
\end{equation}
where $\BF_+,$ $\BG_+$ and $K_+$ are defined by
\begin{equation*}
\BF_+(\xi):= \BBA^\top_-(\Bf \circ \BPhi )(\xi), \quad 
\BG_+(\xi):= |\sA_{\Phi}\Bn_0|\BBA_-^{\top} (\Bg \circ \BPhi )(\xi), \quad 
K_+(\xi):= (k \circ \BPhi )(\xi),
\end{equation*}
and the operators 
\begin{align*}
\CF_1(\Bw)&:=(\wh\gamma_{1}^{-1}-\gamma_1^{-1}\circ \BPhi)
\Di_{\xi}\big(\BBS(\Bw) + \zeta \wh\gamma_3 \di_{\xi} \Bw \BBI_{N}\big)
-(\gamma_1^{-1}\circ\BPhi) \BBA_-^{\top} \Di_{\xi} \BBF(\Bw)\\
& \quad +(\gamma_1^{-1}\circ\BPhi) \BBA_-^{\top}  \BBF_0(\Bw) \Di \sAP,\\
\CF_2(\Bw,H)&:= \BBF_b(\Bw) + \BBG_b(H),\\
\BBF_b(\Bw)&:=\BBA_-^{\top} \BBF(\Bw),\\
\BBG_b(H)&:= \sigma m H \BBA_-^{\top} \BBB_- 
-\sigma \CG(H) \BBI_{N}  -\sigma (g^{ij} \pa_{i}\pa_{j} H -  g^{ij}\Lambda^k_{ij} \pa_k H) \BBA^{\top}_-\BBB_-,\\
\CF_3(\Bw)&:= (1 - |\sA_{\Phi} \Bn_0|^{-1})  \Bw 
-   |\sA_{\Phi} \Bn_0|^{-1} \BBB_-^{\top}\BBA_- \Bw. 
\end{align*}
\smallbreak

By \eqref{hyp:M1M2_BH}, $\BZ_+:=(\BF_+,\lambda^{1\slash 2}\BG_+,\BG_+,K_+)$ belongs to 
$\CY_q(\BBR^N_+)$  with
\begin{equation*}
\|\BZ_+\|_{\CY_q(\BBR^N_+)}
=\| (\BF_+,\lambda^{1\slash 2}\BG_+,\BG_+,K_+)\|_{\CY_q(\BBR^N_+)}
 \leq C_{N,q} \|(\Bf,\lambda^{1\slash 2}\Bg,\Bg,k)\|_{\CY_q(\Omega_+)}.
\end{equation*}
Thus, according to Theorem \ref{thm:GR_half_0}, 
$\Bu:=\CA_0(\lambda,\BBR^N_+) \BZ_+$ and $\Fh:=\CH_0(\lambda,\BBR^N_+)  \BZ_+$ satisfy 
\begin{equation}\label{eq:RR_CNS_BH_3}
\left\{\begin{aligned}
& \lambda  \Bu -\wh\gamma_{1}^{-1} \Di_{\xi}
\big(  \BBS(\Bu) + \zeta  \wh\gamma_3 \di_{\xi} \Bu\, \BBI_{N} \big)  +\CF_1(\Bu)
= \BF_+ +\CR_1(\lambda)\BZ_+
 &&\quad\hbox{in}\quad  \BBR^N_+,\\
&\big( \BBS(\Bu)+\zeta \wh\gamma_3 \di_{\xi} \Bu \BBI_{N} \big)\Bn_{0}
+\sigma (m-\Delta')\Fh \,\Bn_{0} +\CF_2(\Bu,\Fh) \Bn_0 = \BG_+ + \CR_2(\lambda)\BZ_+ 
&&\quad\hbox{on}\quad  \BBR^N_0, \\
&\lambda  \Fh - \Bu \cdot \Bn_{0} + \CF_3(\Bu) \cdot \Bn_0  = K_+ +\CR_3(\lambda)  \BZ_+
&&\quad\hbox{on}\quad  \BBR^N_0,\\
	\end{aligned}\right.
\end{equation}
where $\CR_k(\lambda)\BZ_+,$ for $k=1,2,3,$ are defined by
\begin{align*}
\CR_1(\lambda)\BZ_+    &:= \CF_1 \big( \CA_0 (\lambda,\BBR^N_+) \BZ_+\big),\\
\CR_2 (\lambda) \BZ_+ &:= \CF_2 \big(  \CA_0 (\lambda,\BBR^N_+)\BZ_+, \CH_0 (\lambda,\BBR^N_+)\BZ_+\big)\Bn_0,\\
\CR_3(\lambda)\BZ_+  & :=\CF_3 \big(  \CA_0 (\lambda,\BBR^N_+)\BZ_+\big) \cdot \Bn_0.
\end{align*}
\smallbreak

Next, set that $\CR(\lambda) := \big( \CR_1(\lambda),\CR_2(\lambda), \CR_3(\lambda) \big)^{\top}$ and 
$F_{\lambda}(\Bf,\Bg,k):= ( \Bf,\lambda^{1\slash 2}\Bg, \Bg, k)^{\top}$ for any $(\Bf,\Bg,k)$ in $Y_q(\BBR^N_+).$
Then we claim that there exist some constants $C$ and $\lambda_0$ depending only on $\ep,$ $\sigma,$ $m,$ $\mu,$ $\nu,$ $q,$ $r,$ $N,$ $\zeta_0,$ $\rho_1,$ $\rho_2,$ $\rho_3,$ such that
\begin{multline}\label{es:key_Rbd_BH}
\CR_{\CL\big(\CY_q(\BBR^N_+)\big)} 
\big\{  (\tau \pa_{\tau})^{\ell}F_{\lambda}\CR(\lambda):\lambda\in \Gamma_{\ep,\lambda_0,\zeta}\big\} 
\leq C (M_1 + \sigma_0) (1+\sigma+\sigma  m\lambda_0^{-1} ) \\
+ C \lambda_0^{-1 \slash 2}   (\sigma +1) 
\big( M_3 + \sigma_0^{-\frac{N}{r-N}} M_2^{\frac{r}{r-N}} (M_3^{\frac{r}{r-N}}+m )\big),
\hspace*{1cm}
\end{multline}
for any $0<\sigma_0 <1.$
The proof of \eqref{es:key_Rbd_BH} is postponed to the next subsection.
\smallbreak

Let us continue the proof of Theorem \ref{thm:GR_BH} by admitting \eqref{es:key_Rbd_BH} for a while.
By choosing $\sigma_0, M_1$ small and $\lambda_0$ large enough in \eqref{es:key_Rbd_BH}, it holds that 
\begin{equation}\label{es:FR_BH_1}
\CR_{\CL\big(\CY_q(\BBR^N_+)\big)} 
\big\{  (\tau \pa_{\tau})^{\ell}F_{\lambda}\CR(\lambda):\lambda\in \Gamma_{\ep,\lambda_0,\zeta}\big\} \leq 1\slash 2.
\end{equation} 
In particular, \eqref{es:FR_BH_1} yields that
\begin{equation}\label{es:RF_BH_1}
\|F_{\lambda} \CR(\lambda) F_{\lambda}(\BF_+,\BG_+,K_+)\|_{\CY_q(\BBR^N_+)} 
\leq 1\slash 2 \|F_{\lambda}(\BF_+,\BG_+,K_+)\|_{\CY_q(\BBR^N_+)},
\,\,\, \forall \,\, \lambda \in \Gamma_{\ep,\lambda_0,\zeta},
\end{equation} 
and that
\begin{equation}\label{es:FR_BH_2}
\CR_{\CL\big(\CY_q(\BBR^N_+)\big)} 
\big\{  (\tau \pa_{\tau})^{\ell}\big(Id + F_{\lambda}\CR(\lambda) \big)^{-1}:
\lambda\in \Gamma_{\ep,\lambda_0,\zeta}\big\} \leq 1\slash 2.
\end{equation} 

Now note that  the  norm $\|\cdot\|_{Y_q(\BBR^N_+)}$ is equivalent to $\|\cdot\|_{Y_{q,\lambda}(\BBR^N_+)}$ 
for any fixed $\lambda \in \Gamma_{\ep,\lambda_0,\zeta}$ in $Y_q(\BBR^N_+)$ with 
\begin{equation*}
\|(\Bf,\Bg,k)\|_{Y_{q,\lambda}(\BBR^N_+)} := \|F_{\lambda} (\Bf,\Bg,k)\|_{\CY_q(\BBR^N_+)}, \,\,\,
\forall (\Bf,\Bg,k) \in Y_q(\BBR^N_+).
\end{equation*}
Denote $Y_{q,\lambda}(\BBR^N_+):= \big(Y_q(\BBR^N_+),\|\cdot\|_{Y_{q,\lambda}(\BBR^N_+)}\big)$
and \eqref{es:RF_BH_1} implies that 
\begin{equation*}
\|\CR(\lambda)F_{\lambda}\|_{\CL\big(Y_{q,\lambda}(\BBR^N_+)\big)} \leq 1 \slash 2.
\end{equation*}
Thus the inverse operator $(Id + \CR(\lambda)F_{\lambda})^{-1}$ exists for any fixed $\lambda\in \Gamma_{\ep,\lambda_0,\zeta}$ with
\begin{equation*}
\| (Id + \CR(\lambda)F_{\lambda})^{-1}\|_{\CL\big(Y_{q,\lambda}(\BBR^N_+)\big)} \leq 2.
\end{equation*}
Then for any $(\BF_+,\BG_+,K_+) \in Y_q(\BBR^N_+),$ we set that 
\begin{align*}
\Bw&:= \CA_0(\lambda,\BBR^N_+)F_{\lambda} (Id + \CR(\lambda)F_{\lambda})^{-1} (\BF_+,\BG_+,K_+),\\
H&:= \CH_0(\lambda,\BBR^N_+)F_{\lambda} (Id + \CR(\lambda)F_{\lambda})^{-1} (\BF_+,\BG_+,K_+),
\end{align*}
which solve \eqref{eq:RR_CNS_BH_2} by keeping \eqref{eq:RR_CNS_BH_3} in mind. As 
\begin{equation*}
F_{\lambda} (Id + \CR(\lambda)F_{\lambda})^{-1} 
=\sum_{j=0}^{\infty} F_{\lambda} \big(-\CR(\lambda)F_{\lambda}\big)^j
= \big(Id + F_{\lambda}\CR(\lambda) \big)^{-1} F_{\lambda},
\end{equation*}
the solution $(\Bw, H)$ above can be written by   
\begin{align*}
\Bw= \CA_0(\lambda,\BBR^N_+)  \big(Id + F_{\lambda}\CR(\lambda) \big)^{-1}  \BZ_+, \quad
H= \CH_0(\lambda,\BBR^N_+) \big(Id + F_{\lambda}\CR(\lambda) \big)^{-1} \BZ_+
\end{align*} 
for $\BZ_+ := (\BF_+,\lambda^{1\slash 2}\BG_+,\BG_+,K_+).$

Now we introduce that operators $\CT_k$ for $k=1,2,3,$
\begin{align*}
\CT_1 F_\lambda (\Bf, \Bg, k)
&:= F_{\lambda}\big( \BBA^\top_-(\Bf \circ \BPhi ),   |\sA_{\Phi}\Bn_0|\BBA_-^{\top} (\Bg \circ \BPhi ),  k \circ \BPhi \big),\\
\CT_2 \,\Bw(x)& := \BBA^{\top}_- (\Bw \circ \BPhi^{-1}) \,\,\, \hbox{and}\,\,\,
\CT_3 H (x) := H\circ \BPhi^{-1},
\end{align*}
which, according to assumptions on $\BPhi$ and Lemma \ref{lemma:ab_BH}, satisfy 
\begin{equation}\label{eq:T123_BH}
\CT_1 \in \CL\big( \CY_q(\Omega_+); \CY_q(\BBR^N_+)\big), \,\,\, 
\CT_2 \in \CL \big( H^2_q(\BBR^N_+)^N; H^2_q(\Omega_+)^N \big),\,\,\,
\CT_3 \in \CL \big( H^3_q(\BBR^N_+); H^3_q(\Omega_+)\big).
\end{equation}
At last, according to \eqref{es:FR_BH_2}, \eqref{eq:T123_BH}, Theorem \ref{thm:GR_half_0} 
and Remark \ref{rmk:R-bounded},
\begin{align*}
\CA_0(\lambda,\Omega_+) := \CT_2 \,  \CA_0(\lambda,\BBR^N_+)  \big(Id + F_{\lambda}\CR(\lambda) \big)^{-1} \CT_1,\\
\CH_0(\lambda,\Omega_+) := \CT_3 \,  \CH_0(\lambda,\BBR^N_+)  \big(Id + F_{\lambda}\CR(\lambda) \big)^{-1} \CT_1,
\end{align*}
are the desired solution operators.


\subsection{Proof of \eqref{es:key_Rbd_BH}}
Here we are proving the technical estimate \eqref{es:key_Rbd_BH}. In the rest of this subsection, we always admit that $\tau:=\Im \lambda$ for $\lambda \in \Gamma_{\ep,\lambda_0,\zeta},$ $\ell,\ell'=0,1,$ and the constants $M_1,M_2,M_3$ given by \eqref{hyp:M1M2_BH}.
For brevity, we say the constant $K>0$ is \emph{admissible} if the choice of $K$ depends only on 
the parameters
$\ep,$ $\sigma,$ $m,$ $\mu,$ $\nu,$ $q,$ $r,$ $N,$ $\zeta_0,$ $\rho_1,$ $\rho_2,$ $\rho_3.$
\subsubsection*{The study of $\CR_1(\lambda)$}
Recall $\Bw := \CA_0(\lambda,\BBR^N_+) \BZ_+$ for 
 $\BZ_+:=(\BF_+,\lambda^{1\slash 2}\BG_+,\BG_+,K_+) \in \CY_q(\BBR^N_+).$ 
Firstly, we introduce the operator families $\CT_{ij} (\lambda),$ $\CT_{ijk}(\lambda),$ $i,j,k=1,\dots,N,$ such that
\begin{align*}
\BBF(\Bw)_{ij}= \wt \CT_{ij}(\lambda) ( \lambda^{1\slash 2} \nabla \Bw) 
= \CT_{ij}(\lambda) (\BZ_+), \quad 
\pa_{k} \BBF(\Bw)_{ij} = \CT_{ijk}(\lambda) (\BZ_+).
\end{align*}
Moreover, it is easy to see from \eqref{hyp:M1M2_BH} and \eqref{hyp:gamma_RR_BH} that
\begin{equation*}
\CR_{\CL\big(L_q(\BBR^N_+)^{N^2}; L_q(\BBR^N_+)\big)} 
\big\{(\tau \pa_{\tau})^{\ell} \lambda^{\ell'\slash 2} \wt\CT_{ij}(\lambda):\lambda\in \Sigma_{\ep,\lambda_0}\big\}
\leq C \lambda_0^{(\ell'-1)\slash 2} M_1,
\end{equation*}
which, together with Theorem \ref{thm:GR_half_0} and Remark \ref{rmk:R-bounded}, yields that 
\begin{equation}\label{es:R_Tij_BH}
\CR_{\CL\big(\CY_q(\BBR^N_+); L_q(\BBR^N_+)\big)} 
\big\{(\tau \pa_{\tau})^{\ell} \lambda^{\ell'\slash 2} \CT_{ij}(\lambda):\lambda\in \Gamma_{\ep,\lambda_0,\zeta}\big\}
\leq C \lambda_0^{(\ell'-1)\slash 2}  M_1,
\end{equation}
for some admissible constants $C$ and $\lambda_0 \geq 1.$ 
\smallbreak

To study $\CT_{ijk}(\lambda),$ we define 
$$\wt \CT_{ijk}(\lambda)(\BZ_+):=\pa_i w_j \pa_k \BBB_-
= \pa_i \big(  \CA_{0j}(\lambda, \BBR^N_+) \BZ_+ \big) \pa_k \BBB_- .$$
For any $N \in \BBN,$ $\lambda_{\alpha} \in \Gamma_{\ep,\lambda_0,\zeta},$ $\BZ_{+\alpha} \in \CY_q(\BBR^N_+),$ 
$w_{\alpha j}:= \CA_{0j}(\lambda, \BBR^N_+) \BZ_{+\alpha}$ for all $\alpha =1,\dots,N,$
Lemma \ref{lemma:ab_BH} and Remark \ref{rmk:R-bounded} imply that 
\begin{align*}
\Big\|\sum_{\alpha=1}^N \ep_\alpha \wt\CT_{ijk}(\lambda_\alpha) \BZ_{+\alpha}
\Big\|_{L_q \big(\Omega;L_q(\BBR^N_+)^{N^2}\big)}
&\leq \sigma_0 \Big\| \nabla \pa_i \big( \sum_{\alpha=1}^N \ep_\alpha  w_{\alpha j} \big)
\Big\|_{L_q \big(\Omega;L_q(\BBR^N_+)^{N}\big)} \\
& \quad + C_r \lambda_0^{-1\slash 2} \sigma_0^{-\frac{N}{r-N}} M_2^{\frac{r}{r-N}} 
\Big\|\sum_{\alpha=1}^N \ep_\alpha  \lambda_{\alpha}^{1\slash 2} \pa_i  w_{\alpha j}
\Big\|_{L_q \big(\Omega;L_q(\BBR^N_+)\big)}.
\end{align*}
Then we can conclude from Theorem \ref{thm:GR_half_0} that 
\begin{equation} \label{es:R_Tijk_BH_1}
\CR_{\CL\big(\CY_q(\BBR^N_+); L_q(\BBR^N_+)\big)} 
\big\{(\tau \pa_{\tau})^{\ell} \wt \CT_{ijk}(\lambda):\lambda\in \Gamma_{\ep,\lambda_0,\zeta}\big\}
\leq C \Big(\sigma_0 + \lambda_0^{-1\slash 2}\sigma_0^{-\frac{N}{r-N}}M_2^{\frac{r}{r-N}}\Big),
\end{equation}
which gives us that 
\begin{equation}\label{es:R_Tijk_BH_2}
\CR_{\CL\big(\CY_q(\BBR^N_+); L_q(\BBR^N_+)\big)} 
\big\{(\tau \pa_{\tau})^{\ell}\CT_{ijk}(\lambda):\lambda\in \Gamma_{\ep,\lambda_0,\zeta}\big\}
\leq C \Big( M_1 +\sigma_0 + \lambda_0^{-1\slash 2}\sigma_0^{-\frac{N}{r-N}}M_2^{\frac{r}{r-N}}\Big),
\end{equation}
for some admissible constants $C$ and $\lambda_0 \geq 1.$ 
Moreover,  the bound of $\CR_1(\lambda)$ is clear by \eqref{es:R_Tijk_BH_1} and \eqref{es:R_Tijk_BH_2},
\begin{equation}\label{es:R1_BH}
\CR_{\CL\big(\CY_q(\BBR^N_+); L_q(\BBR^N_+)^N\big)} 
\big\{(\tau \pa_{\tau})^{\ell}\CR_1(\lambda):\lambda\in \Gamma_{\ep,\lambda_0,\zeta}\big\} \leq C 
\Big( M_1 +\sigma_0 + \lambda_0^{-1\slash 2}\sigma_0^{-\frac{N}{r-N}}M_2^{\frac{r}{r-N}}\Big).
\end{equation}
\medskip 

\subsubsection*{The study of $\CR_2(\lambda)$}
Set that
\begin{equation*}
\CR_{21}(\lambda) \BZ_+ :=\BBF_b\big(  \CA_0 (\lambda,\BBR^N_+)\BZ_+\big)\Bn_0, \quad
\CR_{22}(\lambda) \BZ_+ :=\BBG_b\big(  \CH_0 (\lambda,\BBR^N_+)\BZ_+\big)\Bn_0.
\end{equation*}
According to \eqref{es:R_Tij_BH} and \eqref{es:R_Tijk_BH_2},  we have 
\begin{equation*}
\CR_{\CL\big(\CY_q(\BBR^N_+); L_q(\BBR^N_+)^N\big)} 
\big\{ (\tau \pa_{\tau})^{\ell} \lambda^{\ell'\slash 2}\CR_{21}(\lambda):\lambda\in \Gamma_{\ep,\lambda_0,\zeta}\big\}
\leq C \lambda_0^{(\ell'-1)\slash 2}  M_1,
\end{equation*}
\begin{equation*}
\CR_{\CL\big(\CY_q(\BBR^N_+); H^1_q(\BBR^N_+)^N\big)} 
\big\{(\tau \pa_{\tau})^{\ell}\CR_{21}(\lambda):\lambda\in \Gamma_{\ep,\lambda_0,\zeta}\big\}
\leq C \Big( M_1 +\sigma_0 + \lambda_0^{-1\slash 2}\sigma_0^{-\frac{N}{r-N}}M_2^{\frac{r}{r-N}}\Big),
\end{equation*}
for some admissible constants $C$ and $\lambda_0 \geq 1.$ 
\smallbreak

Now let us study $\CR_{22}(\lambda)$ and recall that
\begin{equation*}
\BBG_b(H) = \sigma(m  -\Delta')H \BBA_-^{\top} \BBB_-
 - \sigma \CG(H) \BBA_-^{\top} \sAP,
 \,\,\, \text{with}\,\,\, H :=  \CH_0 (\lambda,\BBR^N_+)\BZ_+.
\end{equation*}
Next, define $\CT_{\CG}^0(\lambda)(\BZ_+) := \CG(H)$ and 
$\CT_{\CG,\alpha}^1(\lambda)(\BZ_+) :=\pa_{\alpha} \CG(H)$ for $\alpha=1,\dots,N.$
Then arguing as \eqref{es:R_Tij_BH} and \eqref{es:R_Tijk_BH_1}, 
we infer from \eqref{es:geo_1_BH}, \eqref{es:geo_2_BH} and Lemma \ref{lemma:ab_BH} that
\begin{equation}\label{es:G_BH_1}
\CR_{\CL\big(\CY_q(\BBR^N_+);L_q(\BBR^N_+)\big)} 
\big\{  (\tau \pa_{\tau})^{\ell}\lambda^{\ell'\slash 2}\CT_{\CG}^{0}(\lambda):\lambda\in \Gamma_{\ep,\lambda_0,\zeta}\big\}
\leq C  \lambda_0^{-1+\ell'\slash 2}M_3,
\end{equation}
\begin{multline}\label{es:G_BH_2}
\CR_{\CL\big(\CY_q(\BBR^N_+);L_q(\BBR^N_+)\big)} 
\big\{  (\tau \pa_{\tau})^{\ell}\CT_{\CG,\alpha}^{1}(\lambda):\lambda\in \Gamma_{\ep,\lambda_0,\zeta}\big\} \\
\leq C \Big( M_1 +\sigma_0 + \lambda_0^{-1} 
\big( M_3 +  \sigma_0^{-\frac{N}{r-N}}  (M_2M_3)^{\frac{r}{r-N}} \big)  \Big), \hspace*{2cm}
\end{multline}
for some admissible constants $C$ and $\lambda_0.$ 
Thus \eqref{es:G_BH_1} and \eqref{es:G_BH_2} yield that
\begin{equation*}
\CR_{\CL\big(\CY_q(\BBR^N_+); L_q(\BBR^N_+)^N\big)} 
\big\{ (\tau \pa_{\tau})^{\ell} \lambda^{\ell'\slash 2}\CR_{22}(\lambda):\lambda\in \Gamma_{\ep,\lambda_0,\zeta}\big\}
\leq C \sigma \lambda_0^{-1+\ell'\slash 2} (M_1 \max\{1,m\} + M_3),
\end{equation*}
\begin{multline*}
\CR_{\CL\big(\CY_q(\BBR^N_+); H^1_q(\BBR^N_+)^N\big)} 
\big\{(\tau \pa_{\tau})^{\ell}\CR_{22}(\lambda):\lambda\in \Gamma_{\ep,\lambda_0,\zeta}\big\}
\leq C\sigma (M_1 + \sigma_0) \max\{1, m\lambda_0^{-1}\}\\
+ C \sigma  \lambda_0^{-1} 
\big( M_3 + \sigma_0^{-\frac{N}{r-N}} M_2^{\frac{r}{r-N}} (M_3^{\frac{r}{r-N}}+\max\{1,m\} )\big). 
\hspace*{1cm}
\end{multline*}
Finally, combing the discussions on $\CR_{21}(\lambda)$ and $\CR_{22}(\lambda)$ yields that 
\begin{equation}\label{es:R2_BH_1}
\CR_{\CL\big(\CY_q(\BBR^N_+); L_q(\BBR^N_+)^N\big)} 
\big\{ (\tau \pa_{\tau})^{\ell} \lambda^{1\slash 2}\CR_{2}(\lambda):\lambda\in \Gamma_{\ep,\lambda_0,\zeta}\big\}
\leq C \big( M_1 +  \sigma  \lambda_0^{-1\slash 2}(m+ M_3) \big),
\end{equation}
\begin{multline}\label{es:R2_BH_2}
\CR_{\CL\big(\CY_q(\BBR^N_+); H^1_q(\BBR^N_+)^N\big)} 
\big\{(\tau \pa_{\tau})^{\ell}\CR_{2}(\lambda):\lambda\in \Gamma_{\ep,\lambda_0,\zeta}\big\} 
\leq C (M_1 + \sigma_0) (1+\sigma+\sigma  m\lambda_0^{-1} ) \\
+ C \lambda_0^{-1 \slash 2}   (\sigma +1) 
\big( M_3 + \sigma_0^{-\frac{N}{r-N}} M_2^{\frac{r}{r-N}} (M_3^{\frac{r}{r-N}}+m )\big). 
\hspace*{1cm}
\end{multline}

\subsubsection*{The study of $\CR_3(\lambda)$}
By direct calculations, we have 
\begin{equation*}
\big| 1- |\sA_{\Phi} \Bn_0|^{-1} \big| \lesssim |\BBB_-|,  \,\,\,
\big| \nabla |\sA_{\Phi} \Bn_0|^{-1} \big| \lesssim |\nabla \BBB_-|, \,\,\,
\big| \nabla^2 |\sA_{\Phi} \Bn_0|^{-1} \big| \lesssim |\nabla \BBB_-|^2 + |\nabla^2\BBB_-|.
\end{equation*}
Thus \eqref{hyp:M1M2_BH} and $r>N$ imply that 
\begin{equation}\label{es:An_0_BH}
\big\| 1- |\sA_{\Phi} \Bn_0|^{-1} \big\|_{L_{\infty}(\BBR^N)} \lesssim M_1,  \,\,\,
\big\| \nabla |\sA_{\Phi} \Bn_0|^{-1} \big\|_{L_r(\BBR^N)} \lesssim M_2, \,\,\,
\big\| \nabla^2 |\sA_{\Phi} \Bn_0|^{-1} \big\|_{L_r(\BBR^N)}  \lesssim M_2M_3.
\end{equation}
For $j,k=1,\dots,N,$ we denote that
\begin{gather*}
\CT_3^0(\lambda)(\BZ_+)=\CF_3(\Bw), \quad
\CT_{3,j}^{1}(\lambda)(\BZ_+)=\pa_j \CF_3(\Bw), \quad 
\CT_{3,jk}^{2}(\BZ_+)=\pa_k \pa_j \CF_3(\Bw).
\end{gather*}
Then \eqref{es:An_0_BH}, Lemma \ref{lemma:ab_BH} and Theorem \ref{thm:GR_half_0} immediately yield that
\begin{equation*}
\CR_{\CL\big( \CY_q(\BBR^N_+); L_q(\BBR^N_+)^{N} \big)} 
\big\{(\tau \pa_{\tau})^{\ell}\CT_{3}^{0}(\lambda):\lambda\in \Gamma_{\ep,\lambda_0,\zeta}\big\}
\leq C  \lambda_0^{-1}M_1,
\end{equation*}
\begin{equation*}
\CR_{\CL\big(\CY_q(\BBR^N_+); L_q(\BBR^N_+)^{N}\big)} 
\big\{(\tau \pa_{\tau})^{\ell}\CT_{3,j}^{1}(\lambda):\lambda\in \Gamma_{\ep,\lambda_0,\zeta}\big\}
\leq C \Big(  \lambda_0^{-1\slash 2}(M_1 + \sigma_0) 
+  \lambda_0^{-1} \sigma_0^{-\frac{N}{r-N}} M_2^{\frac{r}{r-N}}  \Big),
\end{equation*}
\begin{equation*}
\CR_{\CL\big(\CY_q(\BBR^N_+) ; L_q(\BBR^N_+)^{N}\big)} 
\big\{(\tau \pa_{\tau})^{\ell}\CT_{3,jk}^{2}(\lambda):\lambda\in \Gamma_{\ep,\lambda_0,\zeta}\big\}
\leq C \Big( M_1 + \sigma_0
+ \lambda_0^{-1\slash 2} \big(  \sigma_0 M_3 +  \sigma_0^{-\frac{N}{r-N}} (M_2M_3)^{\frac{r}{r-N}} \big) \Big),
\end{equation*}
 for any $0<\sigma_0<1.$ Therefore we infer from  Remark \ref{rmk:R-bounded} and Theorem \ref{thm:GR_half_0} that 
\begin{equation}\label{es:R3_BH}
\CR_{\CL\big(\CY_q(\BBR^N_+);H^2_q(\BBR^N_+)\big)} 
\big\{(\tau \pa_{\tau})^{\ell}\CR_{3}(\lambda):\lambda\in \Gamma_{\ep,\lambda_0,\zeta}\big\}
\leq C \Big( M_1 + \sigma_0
+ \lambda_0^{-1\slash 2} \big( M_3 +  \sigma_0^{-\frac{N}{r-N}} (M_2M_3)^{\frac{r}{r-N}} \big) \Big).
\end{equation}
At last, we can easily conclude \eqref{es:key_Rbd_BH} by combining \eqref{es:R1_BH}, \eqref{es:R2_BH_1}, \eqref{es:R2_BH_2} and \eqref{es:R3_BH}. This completes the proof of \eqref{es:key_Rbd_BH}.

\subsection{Review of other model problems}
To study the model problem in the general domain, let us review some results in \cite{EvBS2014}.
By the notations $\BPhi,$ $\BBB$ and $\BBB_-$ in the beginning of Section \ref{sec:bh}, 
we assume that $\BPhi$ is a $H^2_r$ diffeomorphism for some $N<r<\infty,$
and there exist constants $0<M_1<1\leq M_2<\infty$ such that 
\begin{equation*}
\|(\BBB, \BBB_-)\|_{L_{\infty}(\BBR^N)} \leq M_1,\quad
\|\nabla(\BBB, \BBB_-)\|_{L_{r}(\BBR^N)} \leq M_2.
\end{equation*}
For  $\Omega_+:= \BPhi(\BBR^N_+)$ and $\Gamma_+:= \BPhi(\BBR^N_0),$ we consider the following model problem
\begin{equation}\label{eq:RR_CNS_BH_D}
\left\{\begin{aligned}
& \lambda  \Bv -\gamma_{1}^{-1}\Di \BBS(\Bv) - \lambda^{-1} \gamma_1^{-1} \nabla (\gamma_3 \di \Bv) 
= \Bf &&\quad\hbox{in}\quad  \Omega_+,\\
&\Bv=0 &&\quad\hbox{on}\quad  \Gamma_+, 
	\end{aligned}\right.
\end{equation}
where $\gamma_1$ and $\gamma_3$ satisfy the conditions in \eqref{hyp:gamma_RR_BH}.
\begin{theo}\label{thm:RR_BH_D}
Let $0<\ep<\pi\slash 2,$ $\mu, \nu,\zeta_0>0,$ $1<q<\infty,$ $N<r<\infty$ and $r\geq q.$ 
Assume that $\Omega_+$ is given as above and \eqref{hyp:gamma_RR_BH} is satisfied. 
Then for any $\Bf\in L_q(\Omega_+)^N,$ there exist constants $\lambda_0, r_b \geq 1$ and a family of operators
\begin{align*}
\CA_1(\lambda,\Omega_+) \in 
\Hol\Big(\Gamma_{\ep,\lambda_0,\zeta}; \CL\big(L_q(\Omega_+)^N;H^2_q(\Omega_+)^N \big) \Big),
\end{align*}
such that $\Bv:=\CA_1(\lambda,\Omega_+) \Bf$ is a solution of \eqref{eq:RR_CNS_BH_D}. Moreover, we have 
\begin{gather*}
\CR_{\CL\big(L_q(\Omega_+)^N; H^{2-j}_q(\Omega_+)^N \big)}
 \Big( \Big\{ (\tau \pa_{\tau})^{\ell}\big( \lambda^{j\slash 2}\CA_1(\lambda,\Omega_+)\big) : \lambda \in 
 \Gamma_{\ep,\lambda_0,\zeta} \Big\}\Big) \leq r_b,
\end{gather*}
for $\ell=0,1,$ $j=0,1,2,$ and $\tau := \Im \lambda.$ Above the constants $\lambda_0$ and $r_b$ depend solely on
$\ep,$ $\mu,$ $\nu,$ $q,$ $r,$ $N,$ $\zeta_0,$ $\rho_1,$ $\rho_2,$ $\rho_3.$
\end{theo}
\medskip

Next, let us consider the generalized system in the whole space $\BBR^N,$
\begin{equation}\label{eq:RR_CNS_whole}
 \lambda  \Bv -\gamma_{1}^{-1}\Di \BBS(\Bv) - \lambda^{-1} \gamma_1^{-1} \nabla (\gamma_3 \di \Bv) 
= \Bf \quad\hbox{in}\quad  \BBR^N.
\end{equation}
For \eqref{eq:RR_CNS_whole}, we recall \cite[Theorem 3.11]{EvBS2014} here.
\begin{theo}\label{thm:RR_CNS_whole}
Let $0<\ep<\pi\slash 2,$ $\mu, \nu,\zeta_0>0,$ $1<q<\infty,$ $N<r<\infty$ and $r\geq q.$ 
Assume that $\gamma_1$ and $\gamma_3$ are uniformly continuous functions in $\BBR^N$ and
satisfy \eqref{hyp:gamma_RR_BH} by changing $\Omega_+$ to $\BBR^N.$ 
For any $\Bf\in L_q(\BBR^N)^N,$
there exist constants $\lambda_0,r_b \geq 1$ and a family of operators
\begin{align*}
\CA_2(\lambda,\BBR^N) & \in 
\Hol\Big(\Gamma_{\ep,\lambda_0,\zeta}; \CL\big(L_q(\BBR^N)^N;H^2_q(\BBR^N)^N \big) \Big),
\end{align*}
such that $\Bv:=\CA_2(\lambda,\BBR^N) \Bf$ is a solution of \eqref{eq:RR_CNS_whole}. Moreover, we have 
\begin{gather*}
\CR_{\CL\big(L_q(\BBR^N)^N; H^{2-j}_q(\BBR^N)^N \big)}
 \Big( \Big\{ (\tau \pa_{\tau})^{\ell}\big( \lambda^{j\slash 2}\CA_2(\lambda,\BBR^N)\big) : \lambda \in 
 \Gamma_{\ep,\lambda_0,\zeta} \Big\}\Big) \leq r_b,
\end{gather*}
for $\ell=0,1,$ $j=0,1,2,$ and $\tau := \Im \lambda.$ Above the constants $\lambda_0$ and $r_b$ depend solely on
$\ep,$ $\mu,$ $\nu,$ $q,$ $r,$ $N,$ $\zeta_0,$ $\rho_1,$ $\rho_2,$ $\rho_3.$
\end{theo}
 

\section{Full model problem in the general domain}
This section is dedicated to the study of \eqref{eq:RR_CNS_0}.  After the review of some auxiliary results, we will construct the solutions for model problems by the localization procedure due to Section \ref{sec:bh}, and then establish the leading part of the solution of \eqref{eq:RR_CNS_0}. At last, we will see the remainder part from the parametrix is harmless in the sense of $\CR-$boundedness. 
\subsection{Some auxiliary results}
According to \cite{ES2013}, let us list some properties for the uniformly smooth domain $\Omega$ in the class $W^{3,2}_r.$  
\begin{prop}\label{prop:domain}
Let $\Omega$ is of type $W^{3,2}_r$ by Definition \ref{def:domain} for some $N<r<\infty$ and  $B_{d}(x):=\{y\in \BBR^N: |x-y|<d\}$ be the ball in $\BBR^N$ with radius $d>0.$ 
Then there exist the constants $0<d_1, d_2, d_3<1,$  $0<M_1<1\leq \min\{M_2,M_3\}<\infty,$ at most countably many mappings $\BPhi^i_j$ ($i=0,1$) of the class $H^{3-i}_r(\BBR^N;\BBR^N)$ and points $x_j^0 \in \Gamma_0,$ $x_j^1 \in \Gamma_1,$ $x_j^2 \in \Omega$ such that the following assertions hold:
\begin{enumerate}
\item The mappings $\BPhi^i_j:\BBR^N \rightarrow \BBR^N$ ($i=0,1, j\in \BBN$) are $H^{3-i}_r$ diffeomorphism.
\item Denote $B^i_j:= B_{d_i}(x^i_j)$ ($i=0,1,2, j\in \BBN$) for simplicity and we have 
\begin{gather*}
\Omega = \bigcup_{i=0,1}   \Big( \bigcup_{j\in \BBN}\big( \BPhi_j^i(\BBR^N_+) \cap B^i_j\big) \Big)
\bigcup  \Big( \bigcup_{j\in \BBN} B^2_j \Big), \quad
\bigcup_{j\in \BBN}  B^2_j \subset \Omega,\\
\BPhi_j^i(\BBR^N_+) \cap B^i_j =\Omega \cap B^i_j, \quad \BPhi_j^i(\BBR^N_0) \cap B^i_j =\Gamma_i \cap B^i_j, \,\,\, i=0,1, j \in \BBN.
\end{gather*}
\item There exist $C^{\infty}$ functions $\zeta^i_j, \wt\zeta^i_j$ ($i=0, 1,2, j\in\BBN$) such that 
\begin{gather*}
0 \leq \zij,\tzij \leq 1, \,\,\, \supp \zij,\, \supp \tzij \subset B^i_j,\,\,\, \tzij\zij=\zij,\\
  \|(\zeta^0_j,\wt\zeta^0_j)\|_{C^3(\BBR^N)}  +\sum_{i=1,2} \|(\zij,\tzij)\|_{C^2(\BBR^N)} \leq c_0,\\
\sum_{i=0,1,2}\sum_{j\in\BBN} \zij =1 \,\,\text{on}\,\,\,\overline{\Omega},\quad
\sum_{j\in \BBN}\zeta^0_j=1 \,\,\text{on}\,\,\,\Gamma_0,\quad
\sum_{j\in \BBN}\zeta^1_j=1 \,\,\text{on}\,\,\,\Gamma_1. 
\end{gather*} 
Here the choice of $c_0$ is dependent on $N,r, M_1,M_2,M_3,$ but independent on $j.$ 
\item Denote $\BPsi^i_j:= (\BPhi^i_j)^{-1}$ for $i=0,1$ and $j\in \BBN.$ Then 
\begin{gather*}
\nabla_{\xi} (\BPhi^i_j)^\top (\xi) = \BBA^i_j + \BBB^i_j(\xi),\quad
\nabla_x (\BPsi^i_j)^\top|_{x=\BPhi^i_j(\xi)} = \BBA^i_{j,-} + \BBB^i_{j,-}(\xi) ,
\end{gather*}
where  $\BBA^i_j, \BBA^i_{j,-}$ are orthogonal constant matrices and $\BBB^i_j, \BBB^i_{j,-}$ satisfy
 \begin{equation*}
\|(\BBB^i_j, \BBB^i_{j,-})\|_{L_{\infty}(\BBR^N)} \leq M_1,\,\,\,
\|\nabla(\BBB^i_j, \BBB^i_{j,-})\|_{L_{r}(\BBR^N)} \leq M_2,\,\,\,
\|\nabla^2(\BBB^0_j, \BBB^0_{j,-})\|_{L_{r}(\BBR^N)} \leq M_3.
\end{equation*}
\item There exists an integer $L \geq 2$ such that any $L+1$ distinct balls in $\{B^i_j, i=0,1,2,j\in \BBN\}$ have an empty intersection.
\end{enumerate}
\end{prop}
Let us give some useful comments on Proposition \ref{prop:domain}.
\begin{itemize}
\item Thanks to \eqref{hyp:gamma_GR_2}, we assume that 
\begin{equation}\label{hyp:gamma_Om_2}
\sum_{\alpha=1,3} \|\gamma_{\alpha}(\cdot) - \gamma_{\alpha}(x^i_j) \|_{L_{\infty} ( \Omega \cap B^i_j)} \leq M_1,  
\,\,\, \sum_{\alpha=1,3} \|\nabla \gamma_\alpha\|_{L_r(\Omega \cap B^i_j)} \leq M_2,
\end{equation}
up the choices of $d_i$ and $M_2.$ 

\item On the other hand, by the last property of Proposition \ref{prop:domain}, we have 
\begin{equation}\label{es:fip}
 \big(  \sum_{i=0,1,2} \sum_{j\in \BBN} \|f\|^r_{L_r(\Omega \cap B^i_j)}  \big) ^{1\slash r}\leq C_{r,L} \|f\|_{L_r(\Omega)}
\end{equation}
for any $f\in L_r(\Omega)$ and $1\leq r<\infty.$ 
In particular, we infer from \eqref{es:fip} that 
\begin{equation}\label{es:fip_2}
 \big(  \sum_{i=0,1,2} \sum_{j\in \BBN} \|\BF\|^q_{\CY_q(\Omega \cap B^i_j)}  \big) ^{1\slash q}\leq C_{q,L} \|\BF\|_{\CY_q(\Omega)}, 
\end{equation}
for any $\BF \in \CY_q(\Omega)$ and $1<q<\infty.$
\end{itemize}

For $\Omega$ given by Proposition \ref{prop:domain}, we adopt the notations
\begin{equation}\label{eq:notation_Om}
\Omega^i_j := \BPhi^i_j(\BBR^N_+), \,\,\, 
\Gamma^i_j := \pa \Omega^i_j = \BPhi^i_j(\BBR^N_0), \,\,\,
\Omega^2_j:= \BBR^N, \,\,\, \forall \,\, i=0,1, \,\, j \in \BBN,
\end{equation}
and denote $\Bn_j^0$ for the unit normal vectors subject to $\Gamma^0_j$ in what follows. 
Now we recall the following results proved in \cite[Section 9.5.1]{Shi2016b}.
\begin{prop}\label{prop:sum}
Assume that $\Omega$ satisfies Proposition \ref{prop:domain}, $1 < q <\infty,$ $i=0,1,2,$ and $n \in \BBN_0.$ 
For any $j\in \BBN$ and $i=0,1,2,$ 
take $\eta^i_j \in C_0^{\infty}(B^i_j)$ and  $\{f^i_j : f^i_j  \in H^n_q(\Omega^i_j)\}_{j\in \BBN}$  such that 
\begin{equation*}
\|\eta^i_j\|_{C^n(\BBR^N)} \leq c_1, \quad
\big\| \big\{  \|f^i_j\|_{H^n_q(\Omega^i_j)} \big\}_{j\in \BBN} \big\|_{\ell_q(\BBN)} < \infty,
\end{equation*}
with some constant $c_1$ independent on $i$ and $j.$
Then the infinite sum $f^i:= \sum_{j\in \BBN} \eta^i_j f^i_j $ exists in $H^n_q(\Omega)$ for any $i=0,1,2,$  fulfilling 
\begin{equation*}
\|f^i\|_{H^n_q(\Omega)} \leq C_{n,q,L}c_1
\big\| \|f^i_j\|_{H^n_q(\Omega^i_j)} \big\|_{\ell_q(\BBN)}.
\end{equation*}
\end{prop}
Let us end up this part with some comment on the unit normal vector $\Bn_{\Gamma_0}$ to $\Gamma_0.$
We regard $\Bn_{\Gamma_0}$ as its natural extension to $\Omega$ through 
$\Bn_{\Gamma_0}= \sum_{j\in \BBN} \zeta^0_j \Bn^0_j.$
In addition, Proposition \ref{prop:domain} yields that
\begin{equation}\label{es:normal_0j}
\|\Bn^0_j\|_{L_{\infty}(\BBR^N)} \leq 1, \quad
\| \nabla \Bn^0_j\|_{L_r(\BBR^N)} \leq C_{N} M_2, \quad 
\|\nabla \Bn^0_j\|_{H^1_r(\BBR^N)} \leq C_{N} M_2 M_3.
\end{equation}
Thanks to Proposition \ref{prop:sum}, \eqref{es:fip} and \eqref{es:normal_0j}, it is not hard to see that
\begin{equation}\label{es:normal_0}
\|f \Bn_{\Gamma_0}\|_{H^m_q(\Omega)} \leq C_{N,q,L} c_0 M_2M_3 \|f\|_{H^m_q(\Omega)}, \,\,\, \forall\,\,f \in H^m_q(\Omega),
\end{equation}
with $m=0,1,2,$ and $1<q\leq r.$

\subsection{Localization}
Recall the notations in \eqref{eq:notation_Om} and set that
\begin{equation}\label{eq:def_gamma_ij}
\gamma^i_{j\alpha}(x):= \big( \gamma_{\alpha} (x) - \gamma_{\alpha} (x^i_j) \big) \tzij (x) + \gamma_{\alpha} (x^i_j), 
\,\,\forall \,\, x\in \Omega^i_j, \,\, \alpha=1,3,\,\, i=0,1,2.
\end{equation}
Then it is not hard to see from \eqref{hyp:gamma_GR_2}, \eqref{hyp:gamma_Om_2} and Proposition \ref{prop:domain} that 
\begin{gather}\label{hyp:gamma_ij}
0<\rho_1 \leq \gamma^i_{j1} (x) \leq  \rho_2, \,\,\,
0 < \gamma^i_{j3}(x) \leq \rho_3, \,\,\, \forall \,\,x \in \overline{\Omega^i_j}, \\ \nonumber
\sum_{\alpha=1,3} \|\gamma^i_{j\alpha}(\cdot) - \gamma^i_{j\alpha}(x^i_j) \|_{L_{\infty} ( \Omega^i_j)} \leq M_1,  
\,\,\, \sum_{\alpha=1,3} \|\nabla \gamma_\alpha\|_{L_r(\Omega^i_j)} \leq M_2+C_{r,N}c_0.
\end{gather}
Then we consider the following model problems for any $j\in \BBN,$
\begin{equation}\label{eq:RR_CNS_loc_0}
\left\{\begin{aligned}
& \lambda  \Bv_j^0 -(\gamma_{j1}^0)^{-1}\Di \big( \BBS(\Bv^0_j) + \zeta \gamma_{j3}^0 \di \Bv^0_j \BBI\big)  =\wt\zeta^0_j \Bf 
&&\quad\hbox{in}\quad  \Omega^0_j,\\
& \big( \BBS(\Bv^0_j) +\zeta \gamma^0_{j3} \di \Bv^0_j \BBI\big)\Bn^0_{j}
+\sigma (m-\Delta_{\Gamma^0_j})h^0_j \,\Bn^0_{j} =\wt\zeta^0_j  \Bg 
&&\quad\hbox{on}\quad  \Gamma^0_j, \\
&\lambda  h^0_j - \Bv^0_j \cdot \Bn^0_{j}  = \wt\zeta^0_j  k 
&&\quad\hbox{on}\quad  \Gamma^0_j,
\end{aligned}\right.
\end{equation}

\begin{equation}\label{eq:RR_CNS_loc_1}
\left\{\begin{aligned}
& \lambda  \Bv_j^1 -(\gamma_{j1}^1)^{-1}\Di \big( \BBS(\Bv^1_j) +\zeta \gamma_{j3}^1 \di \Bv^1_j \BBI \big) =\wt\zeta^1_j \Bf 
&&\quad\hbox{in}\quad  \Omega^1_j,\\
& \Bv^1_j =0
&&\quad\hbox{on}\quad  \Gamma^1_j, 
\end{aligned}\right.
\end{equation}

\begin{equation}\label{eq:RR_CNS_loc_2}
\begin{aligned}
& \lambda  \Bv_j^2 -(\gamma_{j1}^2)^{-1} \Di \big( \BBS(\Bv^2_j) + \zeta \gamma_{j3}^2 \di \Bv^2_j \BBI \big)=\wt\zeta^2_j \Bf 
&&\quad\hbox{in}\quad  \Omega^2_j.
\end{aligned}
\end{equation}
Thanks to \eqref{hyp:gamma_ij},  Theorem \ref{thm:GR_BH}, Theorem \ref{thm:RR_BH_D} and Theorem \ref{thm:RR_CNS_whole}, there exist constants $\lambda_0,r_b \geq 1$ and the families of operators 
\begin{align*}
\CA_0(\lambda,\Omega^0_j) & \in 
\Hol\Big(\Gamma_{\ep,\lambda_0,\zeta}; \CL\big(\CY_q(\Omega^0_j);H^2_q(\Omega^0_j)^N \big) \Big),\\
\CH_0(\lambda,\Omega^0_j ) & \in 
 \Hol\Big(\Gamma_{\ep,\lambda_0,\zeta}; \CL\big(\CY_q(\Omega^0_j);H^3_q(\Omega^0_j)\big) \Big),\\
 \CA_{i}(\lambda,\Omega^{i}_j) & \in 
\Hol\Big(\Gamma_{\ep,\lambda_0,\zeta}; \CL\big(L_q(\Omega^i_j)^N;H^2_q(\Omega^i_j)^N \big) \Big) 
\quad (i=1,2),
\end{align*}
such that 
\begin{align*}
\Bv^0_j &:= \CA_0(\lambda,\Omega^0_j) 
\big( \wt\zeta^0_j\Bf,\lambda^{1\slash 2}\wt\zeta^0_j\Bg,\wt\zeta^0_j \Bg,\wt\zeta^0_j k \big), \\
h^0_j &:= \CH_0(\lambda,\Omega^0_j) 
\big( \wt\zeta^0_j\Bf,\lambda^{1\slash 2}\wt\zeta^0_j\Bg,\wt\zeta^0_j \Bg,\wt\zeta^0_j k \big), \\
\Bv^i_j &:= \CA_i(\lambda,\Omega^i_j) \wt\zeta^i_j\Bf \quad (i=1,2),
\end{align*}
satisfy \eqref{eq:RR_CNS_loc_0}, \eqref{eq:RR_CNS_loc_1} and \eqref{eq:RR_CNS_loc_2} respectively.
Moreover, we have  
 \begin{gather}
\CR_{\CL\big(\CY_q(\Omega^0_j); H^{2-k}_q(\Omega^0_j)^N \big)}
 \Big( \Big\{ (\tau \pa_{\tau})^{\ell}\big( \lambda^{k\slash 2}\CA_0(\lambda,\Omega^0_j)\big) : \lambda \in 
 \Gamma_{\ep,\lambda_0,\zeta} \Big\}\Big) \leq r_{b}, \nonumber \\
 \CR_{\CL\big(\CY_q(\Omega^0_j); H^{3-k'}_q(\Omega^0_j) \big)}
 \Big( \Big\{ (\tau \pa_{\tau})^{\ell}\big( \lambda^{k'}\CH_0(\lambda,\Omega^0_j)\big) : \lambda \in 
 \Gamma_{\ep,\lambda_0,\zeta} \Big\}\Big) \leq r_{b},  \label{es:Rbdd_loc}  \\  
 \CR_{\CL\big(L_q(\Omega^i_j)^N; H^{2-k}_q(\Omega^i_j)^N \big)}
 \Big( \Big\{ (\tau \pa_{\tau})^{\ell}\big( \lambda^{k\slash 2}\CA_i(\lambda,\Omega^i_j)\big) : \lambda \in 
 \Gamma_{\ep,\lambda_0,\zeta} \Big\}\Big) \leq r_{b}, \nonumber
\end{gather}
 for $i=1,2,$ $k=0,1,2,$ $k',\ell=0,1,$  $\tau := \Im \lambda$ and $j\in \BBN.$
 The constants $\lambda_0$ and $r_b$ depend solely on
$\ep,$ $\sigma,$ $m,$ $\mu,$ $\nu,$ $\zeta_0,$ $q,$ $r,$ $N,$ $\rho_1,$ $\rho_2,$ $\rho_3$ and $\Omega.$ 
In particular, we have
\begin{gather}
\sum_{k=0,1,2} \|\lambda^{k\slash 2} \Bv^0_j \|_{H^{2-k}_q (\Omega^0_j)}
+ \sum_{k' =0,1} \| \lambda^{k'} h^0_j\|_{H^{3-k'}_q(\Omega^0_j)}
\leq 5r_b (1+2c_0)\| (\Bf,\lambda^{1\slash 2}\Bg,\Bg,k ) \|_{\CY_q(\Omega \cap B^0_j)},
\nonumber \\ \label{es:RR_loc_1}
\sum_{i=1,2}\sum_{k=0,1,2} \|\lambda^{k\slash 2} \Bv^i_j \|_{H^{2-k}_q (\Omega^i_j)}
\leq 6r_b \|\Bf\|_{L_q(\Omega \cap B^i_j)}, \,\,\,\forall j\in \BBN,
\end{gather}
where we have used the fact that
\begin{equation*}
\big\| \wt\zeta^i_j \BF \big\|_{\CY_q(\Omega^i_j)}
\leq (1+2c_0) \|\BF\|_{\CY_q(\Omega \cap B^i_j)}, 
 \,\,\, \text{for all}\,\,\, \BF \in \CY_q(\Omega)
 \,\,\, \text{and}\,\,\,i=0,1,2 .
\end{equation*}

\subsection{Construction of a parametrix}
Now we define $F_{\lambda}(\Bf,\Bg,k):= ( \Bf,\lambda^{1\slash 2}\Bg, \Bg, k)^{\top}$ 
for any $(\Bf,\Bg,k) \in Y_q(\Omega)$ and $\lambda \in \Gamma_{\ep,\lambda_0,\zeta},$ 
and then introduce 
 \begin{align*}
 \Bv&=\CA_{\Fp}(\lambda) F_{\lambda}(\Bf, \Bg,k) :=\sum_{i\in\{0,1,2\}} \sum_{j\in \BBN} \zeta^i_j \Bv^i_j \\
 &= \sum_{j\in \BBN} \zeta^0_j \CA_0(\lambda,\Omega^0_j) 
\big( \wt\zeta^0_j\Bf,\lambda^{1\slash 2}\wt\zeta^0_j\Bg,\wt\zeta^0_j \Bg,\wt\zeta^0_j k \big)
+\sum_{i\in\{1,2\}} \sum_{j\in \BBN}  \zeta^i_j \CA_i(\lambda,\Omega^i_j) \wt\zeta^i_j\Bf,\\
h &= \CH_{\Fp}(\lambda) F_{\lambda}(\Bf, \Bg,k)  :=  \sum_{j\in \BBN} \zeta^0_j h^0_j
=\sum_{j\in \BBN} \zeta^0_j  \CH_0(\lambda,\Omega^0_j) 
\big( \wt\zeta^0_j\Bf,\lambda^{1\slash 2}\wt\zeta^0_j\Bg,\wt\zeta^0_j \Bg,\wt\zeta^0_j k \big).
 \end{align*}
 By Proposition \ref{prop:sum}, \eqref{es:fip_2} and \eqref{es:RR_loc_1}, we have $\Bv \in H^2_q(\Omega)^N$ and $h\in H^3_q(\Omega).$ 
 Moreover, according to \eqref{eq:def_gamma_ij}, \eqref{eq:RR_CNS_loc_0}, \eqref{eq:RR_CNS_loc_1} and \eqref{eq:RR_CNS_loc_2}, 
 $\Bv$ and $h$ satisfy 
\begin{equation}\label{eq:RR_CNS_para}
\left\{\begin{aligned}
& \lambda  \Bv -\gamma_{1}^{-1}\Di \big( \BBS(\Bv) + \zeta \gamma_3 \di \Bv \BBI \big)
= \Bf - \CV_1 (\lambda) F_{\lambda}(\Bf,\Bg,k) &&\quad\hbox{in}\quad  \Omega,\\
& \big( \BBS(\Bv)  +\zeta \gamma_3 \di \Bv \BBI \big)\Bn_{\Gamma_{0}}
+\sigma (m-\Delta_{\Gamma_{0}})h \,\Bn_{\Gamma_{0}} = \Bg- \CV_2(\lambda) F_{\lambda}(\Bf,\Bg,k)  &&\quad\hbox{on}\quad  \Gamma_0, \\
&\lambda  h - \Bv \cdot \Bn_{\Gamma_0}  = k  &&\quad\hbox{on}\quad  \Gamma_0,\\
&\Bv  = \0 &&\quad\hbox{on}\quad  \Gamma_1,\\
	\end{aligned}\right.
\end{equation} 
 where the operators $\CV_1 (\lambda)$ and $\CV_2(\lambda)$ are given by
\begin{align*}
 \CV_1(\lambda)F_{\lambda} (\Bf,\Bg,k) 
 := &\sum_{i\in \{0,1,2\}}\sum_{j\in \BBN} (\gamma^i_{j1})^{-1} \Big(   \Di \BBS(\zeta^i_j \Bv^i_j) - \zeta^i_j \Di \BBS(\Bv^i_j)  \\
  &  +  \zeta \Di \big( \gamma^i_{j3} \di (\zeta^i_j \Bv^i_j) \BBI \big)  
  - \zeta \zeta^i_j \Di \big(\gamma^i_{j3}\di \Bv^i_j \BBI \big) \Big),\\
\CV_2(\lambda) F_{\lambda}(\Bf,\Bg,k)  
:=&-\sum_{j\in \BBN} \Big( \BBS(\zeta^0_j \Bv^0_j) - \zeta^0_j \BBS(\Bv^0_j)  
+ \zeta \gamma^0_{j3} \big( \di (\zeta^0_j \Bv^0_j) -\zeta^0_j \di \Bv^0_j \big) \BBI\Big) \Bn^0_j \\
 &  +  \sum_{j\in \BBN} \big( \Delta_{\Gamma^0_j} (\zeta^0_j h^0_j) -\zeta^0_j \Delta_{\Gamma^0_j}  h^0_j  \big) \Bn^0_j.
\end{align*}

By the definitions of $\CA_{\Fp}(\lambda)$ and $\CH_{\Fp}(\lambda),$ we can infer from Proposition \ref{prop:sum}, \eqref{es:Rbdd_loc} and \eqref{es:fip_2} that,
\begin{align*}
\CA_{\Fp}(\lambda) & \in 
\Hol\Big(\Gamma_{\ep,\lambda_0,\zeta}; \CL\big(\CY_q(\Omega);H^2_q(\Omega)^N \big) \Big),\\
\CH_{\Fp}(\lambda) & \in 
 \Hol\Big(\Gamma_{\ep,\lambda_0,\zeta}; \CL\big(\CY_q(\Omega);H^3_q(\Omega) \big) \Big).
\end{align*}
In addition, there exists some constant $C$ depending solely on $q,L$ and $c_0$ such that
 \begin{gather}\label{es:R-bdd_para}
\CR_{\CL\big(\CY_q(\Omega); H^{2-k}_q(\Omega)^N \big)}
 \Big( \Big\{ (\tau \pa_{\tau})^{\ell}\big( \lambda^{k\slash 2}\CA_{\Fp}(\lambda)\big) : \lambda \in 
 \Gamma_{\ep,\lambda_0,\zeta} \Big\}\Big) \leq C r_b,\\ \nonumber
 \CR_{\CL\big(\CY_q(\Omega); H^{3-k'}_q(\Omega) \big)}
 \Big( \Big\{ (\tau \pa_{\tau})^{\ell}\big( \lambda^{k'}\CH_{\Fp}(\lambda)\big) : \lambda \in 
 \Gamma_{\ep,\lambda_0,\zeta} \Big\}\Big) \leq C r_b,
\end{gather}
 for any $\ell,k'=0,1,$ $k=0,1,2,$ and $\tau := \Im \lambda.$
 \smallbreak 
 
 Here we just prove the estimates of $\CH_{\Fp}(\lambda)$ for instance. Take any $N_0 \in \BBN,$ $\Fa=1,\dots,N_0,$ $\BF_{\Fa} \in \CY_q(\Omega),$ the Rademacher functions $r_{\Fa}.$ Then Proposition \ref{prop:sum} gives us that
\begin{equation*}
\Big\| \sum_{\Fa=1}^{N_0} r_{\Fa} (\cdot) (\tau \pa_\tau )^{\ell} \big(\lambda^{k'}_{\Fa} \CH_{\Fp}(\lambda_{\Fa})\BF_{\Fa} \big) \Big\|_{H^{3-k'}_q(\Omega)}^q   \leq C \sum_{j\in \BBN}
\Big\| \sum_{\Fa=1}^{N_0} r_{\Fa} (\cdot) (\tau \pa_\tau )^{\ell} 
\big(\lambda^{k'}_{\Fa} \CH_0(\lambda_{\Fa},\Omega^0_j)(\wt\zeta^0_j\BF_{\Fa}) \big),
 \Big\|_{H^{3-k'}_q(\Omega^0_j)}^q 
\end{equation*}
for some constant $C =C (q,L,c_0).$
Combining above bound, \eqref{es:Rbdd_loc}, Minkowski inequalities and \eqref{es:fip_2} yields that 
\begin{align*}
\Big\| \sum_{\Fa=1}^{N_0} r_{\Fa} (\cdot) (\tau \pa_\tau )^{\ell} 
\big(\lambda^{k'}_{\Fa} \CH_{\Fp}(\lambda_{\Fa})\BF_{\Fa} \big) \Big\|_{ L_q([0,1]; H^{3-k'}_q(\Omega))} 
& \leq C r_b \Big( \sum_{j\in \BBN}
\Big\| \sum_{\Fa=1}^{N_0} r_{\Fa} (\cdot) \BF_{\Fa}  \Big\|_{L_q([0,1];\CY_q(\Omega \cap B^0_j)) }^q \Big) ^{1\slash q}\\
& \leq C r_b \Big\| \sum_{\Fa=1}^{N_0} r_{\Fa} (\cdot) \BF_{\Fa}  \Big\|_{L_q([0,1];\CY_q(\Omega))}.
\end{align*}
This completes the study of $\CH_{\Fp}(\lambda).$ 
\medskip

Next we denote $\CV(\lambda):= \big( \CV_1(\lambda), \CV_2(\lambda),0 \big)^{\top}$ and claim
 \begin{equation} \label{es:key_Rbd_FV}
\CR_{\CL(\CY_q(\Omega))} 
\big\{  (\tau \pa_{\tau})^{\ell}F_{\lambda}\CV(\lambda):\lambda\in \Gamma_{\ep,\lambda_0,\zeta}\big\} 
\leq C r_b \lambda_0^{-1\slash 2} , 
\end{equation}
 whose proof is postponed to the next subsection. 
 By choosing $\lambda_0$ in \eqref{es:key_Rbd_FV} large enough, we have 
\begin{equation}\label{es:FV_1}
\|F_{\lambda} \CV(\lambda) F_{\lambda}(\Bf,\Bg,k)\|_{\CY_q(\Omega)} 
\leq 1\slash 2 \|F_{\lambda}(\Bf,\Bg,k)\|_{\CY_q(\Omega)},
\,\,\, \forall \,\, \lambda \in \Gamma_{\ep,\lambda_0,\zeta}.
\end{equation} 
 \begin{equation}\label{es:FV_2}
\CR_{\CL(\CY_q(\Omega))} 
\big\{  (\tau \pa_{\tau})^{\ell}\big(Id - F_{\lambda}\CV(\lambda) \big)^{-1}:
\lambda\in \Gamma_{\ep,\lambda_0,\zeta}\big\} \leq 2.
\end{equation} 
Thanks to \eqref{es:FV_1}, $Id-\CV(\lambda)F_{\lambda}$ is invertible on the space $Y_{q,\lambda}(\Omega)$ (cf. $Y_{q,\lambda}(\BBR^N_+)$ in Sect. \ref{sec:bh}). Let us set
 \begin{equation*}
\Bv:= \CA_{\Fp}(\lambda) F_{\lambda} \big( Id  -\CV(\lambda)F_{\lambda} \big)^{-1}(\Bf,\Bg,k) , \quad 
h:= \CH_{\Fp}(\lambda) F_{\lambda} \big( Id  -\CV(\lambda)F_{\lambda} \big)^{-1}(\Bf,\Bg,k).
 \end{equation*}
Then $\Bv$ and $h$ satisfy \eqref{eq:RR_CNS_0}. Note the fact that 
 \begin{equation*}
F_{\lambda} \big( Id  -\CV(\lambda)F_{\lambda} \big)^{-1} 
=\sum_{j=0}^{\infty} F_{\lambda} \big(\CV(\lambda)F_{\lambda}\big)^j
= \big( Id -F_{\lambda}\CV(\lambda) \big)^{-1} F_{\lambda}.
\end{equation*}
Then $\CA_0(\lambda,\Omega):= \CA_{\Fp}(\lambda) \big( Id -F_{\lambda}\CV(\lambda) \big)^{-1}$ and 
$\CH_0(\lambda,\Omega):= \CH_{\Fp}(\lambda) \big( Id -F_{\lambda}\CV(\lambda) \big)^{-1}$ are desired operators due to 
\eqref{es:R-bdd_para} and \eqref{es:FV_2}.

 \subsection{Proof of \eqref{es:key_Rbd_FV}}
 For any $\BF= (\BF^1,\BF^2,\BF^3,F^4) \in \CY_q(\Omega),$ denote that
 \begin{equation*}
 \CS^0_j (\lambda)\BF:= \CA_0(\lambda,\Omega^0_j)(\wt \zeta^0_j \BF), \,\,\,
  \CT^0_j (\lambda)\BF:= \CH_0(\lambda,\Omega^0_j)(\wt \zeta^0_j \BF),\,\,\,
  \CS^i_j (\lambda)\BF:= \CA_i(\lambda,\Omega^i_j)(\wt \zeta^i_j \BF^1)
 \end{equation*}
for $i=1,2.$ Then we have 
 \begin{align*}
 \CV_1(\lambda) \BF =&  \sum_{i=0,1,2}\sum_{j\in \BBN} (\gamma_{j1}^i)^{-1} 
 \Big( \BBS\big( \CS^i_j (\lambda)\BF\big) 
 +\zeta \gamma^i_{j3} \di  \big( \CS^i_j(\lambda)\BF \big) \BBI \Big)\nabla \zeta^i_j\\
 &+ \sum_{i=0,1,2}\sum_{j\in \BBN}  (\gamma_{j1}^i)^{-1} \Di \Big( 
 \mu \big( \nabla \zeta^i_j \otimes \CS^i_j(\lambda)\BF +  \CS^i_j(\lambda)\BF \otimes  \nabla \zeta^i_j \big)
 \Big) \\ 
  & + \sum_{i=0,1,2}\sum_{j\in \BBN}  (\gamma_{j1}^i)^{-1} \Di \Big( 
 (\nu-\mu + \zeta \gamma^i_{j3}) \big(  \CS^i_j(\lambda)\BF  \cdot \nabla \zeta^i_j\big) \BBI \Big),
 \end{align*}
 
 \begin{align*}
 \CV_2(\lambda)\BF=&- \sum_{j\in \BBN} \mu 
 \big( \nabla \zeta^0_j \otimes \CS^0_j(\lambda)\BF +  \CS^0_j(\lambda)\BF \otimes  \nabla \zeta^0_j \big) \Bn^0_j\\ 
&- \sum_{j\in \BBN} (\nu-\mu + \zeta \gamma^0_{j3}) \big(  \CS^0_j(\lambda)\BF  \cdot \nabla \zeta^0_j\big) \Bn^0_j\\
& + \sum_{j \in \BBN} \big( (\wt\Delta_{\Gamma^0_j} \zeta^0_j) \CT^0_j (\lambda)\BF  
+ 2 \wt\nabla_{\Gamma^0_j}\zeta^0_j \cdot \wt\nabla_{\Gamma^0_j}\CT^0_j (\lambda)\BF \big)  \Bn^0_j.
 \end{align*}
 where $\wt\nabla_{\Gamma^0_j} f := \Pi_{\Gamma^0_j}\nabla f,$ $\Pi_{\Gamma^0_j} := \BBI - \Bn^0_j \otimes \Bn^0_j$ and 
\begin{equation*}
\wt \Delta_{\Gamma^0_j} f := \Delta f - \tr (\Pi_{\Gamma^0_j} \nabla (\Bn^0_j)^{\top}) (\Bn^0_j \nabla f) - (\Bn^0_j)^{\top} (\nabla^2 f) \Bn^0_j,
\end{equation*}
for any smooth function $f$ defined near $\Gamma^0_j.$
 In the formula of $\CV_2(\lambda),$ we have used the fact that 
\begin{equation*}
\nabla_{\Gamma^0_j} f = \wt \nabla_{\Gamma^0_j} f, \,\,\, 
\Delta_{\Gamma^0_j} f  = \wt \Delta_{\Gamma^0_j} f \,\,\, \text{on} \,\,\, \Gamma^0_j.
\end{equation*}

Furthermore, \eqref{es:normal_0j} immediately yields that
\begin{align} \label{es:cut-off_0j}
\big\|\wt\nabla_{\Gamma^0_j} \zeta^0_j \big\|_{H^1_{\infty}(\BBR^N)} 
+\big\|\wt\Delta_{\Gamma^0_j} \zeta^0_j \big\|_{L_{\infty}(\BBR^N)} \leq C_{N} M_2M_3 c_0.
\end{align}
Then thanks to Lemma \ref{lemma:ab_BH},  \eqref{es:Rbdd_loc}, \eqref{es:normal_0j}, \eqref{es:cut-off_0j}  and \eqref{es:fip_2}, it is not hard to see that for any $0<\sigma_0<1$ and $\ell,\ell'=0,1,$
 \begin{equation} \label{es:key_Rbd_FV_1}
\CR_{\CL(\CY_q(\Omega);L_q(\Omega)^N )} 
\big\{  (\tau \pa_{\tau})^{\ell}\CV_1(\lambda):\lambda\in \Gamma_{\ep,\lambda_0,\zeta}\big\} 
\leq C r_b\big(  \lambda_0^{-1\slash 2} + \sigma_0^{-N \slash (r-N)} \lambda_0^{-1}\big),
\end{equation}
 \begin{equation} \label{es:key_Rbd_FV_2_1}
\CR_{\CL(\CY_q(\Omega);L_q(\Omega)^N )} 
\big\{  (\tau \pa_{\tau})^{\ell}\lambda^{\ell'\slash 2}\CV_2(\lambda):\lambda\in \Gamma_{\ep,\lambda_0,\zeta}\big\} 
\leq C r_b \lambda_0^{-1+\ell'\slash 2},
\end{equation}
 \begin{equation} \label{es:key_Rbd_FV_2_2}
\CR_{\CL(\CY_q(\Omega);H^1_q(\Omega)^N )} 
\big\{  (\tau \pa_{\tau})^{\ell} \CV_2(\lambda):\lambda\in \Gamma_{\ep,\lambda_0,\zeta}\big\} 
\leq  C r_b\big(  \lambda_0^{-1\slash 2} + \sigma_0^{-N \slash (r-N)} \lambda_0^{-1}\big),
\end{equation}
with some constant $C$ depending on $N,r,q,L,c_0, \mu,\nu,\zeta_0,\rho_1,\rho_2,\rho_3,M_2,M_3.$ 

Here, for instance, we only consider
$$ \CV^{\alpha}_{2,stk}(\lambda) \BF 
:= \sum_{j\in \BBN} \big( \delta_{st} - (\Bn^0_j)_s(\Bn^0_j)_t \big) \pa_{\alpha} \pa_{t}(\Bn^0_j)_s (\Bn^0_j)_k  (\pa_k \zeta^0_j)
\CT^0_j (\lambda)\BF \Bn^0_j, $$
with $\alpha,s,t,k=1,\dots,N,$ arising from the study of
 $\sum_{j \in \BBN} (\pa_{\alpha} \wt\Delta_{\Gamma^0_j} \zeta^0_j) \CT^0_j (\lambda)\BF \Bn^0_j$ in \eqref{es:key_Rbd_FV_2_2}. 
For any $N_0\in \BBN,$ $\Fa=1,\dots,N_0,$ $\BF_{\Fa} \in \CY_q(\Omega)$ and the Rademacher functions $r_{\Fa},$ 
Proposition \ref{prop:sum}, Lemma \ref{lemma:ab_BH} and \eqref{es:normal_0j} imply that 
\begin{align*}
\Big\| \sum_{\Fa=1}^{N_0} r_{\Fa} (\cdot) (\tau \pa_\tau )^{\ell} \big(\CV^{\alpha}_{2,stk}(\lambda_{\Fa})\BF_{\Fa} \big) \Big\|_{L_q(\Omega)}
\leq C_{N,q,r,L}c_0  \Big(\Big\|\sum_{\Fa=1}^{N_0} r_{\Fa} (\cdot) (\tau \pa_\tau )^{\ell} \nabla  \CT^0_j (\lambda_{\Fa})\BF_{\Fa}\Big\|_{\ell_q(\BBN;L_{q}(\Omega^0_j))}\\
+ (\sigma_0^{-1} M_2M_3)^{N\slash (r-N)} \Big\|\sum_{\Fa=1}^{N_0} r_{\Fa} (\cdot) (\tau \pa_\tau )^{\ell}\CT^0_j (\lambda_{\Fa})\BF_{\Fa}\Big\|_{\ell_q(\BBN;L_{q}(\Omega^0_j))}
\Big),
\end{align*}
for any $\sigma_0 \in ]0,1[.$ 
Thus we see from Minkowski's inequalities, \eqref{es:Rbdd_loc} and \eqref{es:fip_2} that,
 \begin{equation*}
 \Big\| \sum_{\Fa=1}^{N_0} r_{\Fa} (\cdot) (\tau \pa_\tau )^{\ell} \big(\CV^{\alpha}_{2,stk}(\lambda_{\Fa})\BF_{\Fa} \big) \Big\|_{L_q([0,1];L_q(\Omega))}
 \leq C r_b  \sigma_0^{-N \slash (r-N)} \lambda_0^{-1}
  \Big\| \sum_{\Fa=1}^{N_0} r_{\Fa} (\cdot) \BF_{\Fa}  \Big\|_{L_q([0,1];\CY_q(\Omega))},
 \end{equation*}
for some constant $C$ solely depending on  $N,r,q,L,c_0,M_2,M_3.$ 
\medskip

Finally, put \eqref{es:key_Rbd_FV_1}, \eqref{es:key_Rbd_FV_2_1} and \eqref{es:key_Rbd_FV_2_2} together and we can conclude \eqref{es:key_Rbd_FV} by choosing $\sigma_0=1\slash 2.$

\section*{Acknowledgement}
 The author would like to thank Professor Yoshihiro Shibata for his warm discussions on this topic.
 X.Z. is supported by the Top Global University Project.
%

\end{document}